\documentclass[11pt, reqno]{amsart}
\usepackage{amsmath, amsthm, amscd, amsfonts, amssymb, graphicx, xcolor}
\usepackage[bookmarksnumbered, colorlinks, plainpages]{hyperref}
\hypersetup{
    colorlinks=true, 
    linktoc=all,     
    linkcolor=blue,  
}
\usepackage{lipsum}
\usepackage{mathrsfs}
\usepackage{mathtools}
\usepackage{tikz}
\usepackage{tikz-cd}
\usetikzlibrary{shapes.geometric, positioning, calc}
\usetikzlibrary{arrows}
\usetikzlibrary{graphs}
\usetikzlibrary{3d,calc}
\usepackage{multirow} 
\usepackage[toc,page]{appendix}
\usepackage[inline]{enumitem}
\usepackage{calc}
\usepackage{fontspec}
\usepackage{comment}
\usepackage{tabularx}
\usepackage{makecell}
\usepackage{empheq}
\usepackage{subcaption}
\usepackage{dsfont}
\usepackage{mdframed}
\usepackage{indentfirst} 
\usepackage{erewhon}
\usepackage{mathpazo}
\usepackage{amsrefs} 
\usepackage{caption}
\makeatletter

\renewcommand{\lim}{\mathsf{lim}}

\renewcommand{\hom}{\hspace{0.1em}\mathsf{Hom}}
\newtheorem{theorem}{Theorem}[section]
\newtheorem{lemma}[theorem]{Lemma}
\newtheorem{corollary}[theorem]{Corollary}
\theoremstyle{definition}
\newtheorem{definition}[theorem]{Definition}

\newtheorem{remark}[theorem]{Remark}
\newcommand{\cat}[1]{\mathbf{\boldsymbol{\mathsf{#1}}}}
\newcommand\restr[2]{{
  \left.\kern-\nulldelimiterspace 
  #1 
  \vphantom{\big|} 
  \right|_{#2} 
  }}
\newcommand\preceqdot{\mathrel{\ooalign{$\preceq$\cr\hidewidth\raise0.225ex\hbox{$\cdot\mkern0.5mu$}\cr}}} 
\numberwithin{equation}{section}
\newcommand{\key}[1]{\emph{#1}}
\theoremstyle{definition}

\DeclareMathOperator{\op}{\hspace{-0.2em}^{\mathsf{op}}}

\DeclareMathOperator{\id}{\mathsf{id}}

\DeclareMathOperator{\rank}{\mathsf{rank}}
\DeclareMathOperator{\Ima}{\mathsf{Im}}

\makeatletter

\makeatletter
\def\@tocline#1#2#3#4#5#6#7{\relax
  \ifnum #1>\c@tocdepth 
  \else
    \par \addpenalty\@secpenalty\addvspace{#2}%
    \begingroup \hyphenpenalty\@M
    \@ifempty{#4}{%
      \@tempdima\csname r@tocindent\number#1\endcsname\relax
    }{%
      \@tempdima#4\relax
    }%
    \parindent\z@ \leftskip#3\relax \advance\leftskip\@tempdima\relax
    \rightskip\@pnumwidth plus4em \parfillskip-\@pnumwidth
    #5\leavevmode\hskip-\@tempdima
      \ifcase #1
       \or\or \hskip 1em \or \hskip 2em \else \hskip 3em \fi%
      #6\nobreak\relax
    \hfill\hbox to\@pnumwidth{\@tocpagenum{#7}}\par
    \nobreak
    \endgroup
  \fi}
\makeatother
\begin{document}
\setcounter{page}{1}

\color{darkgray}{
\noindent 
{\small 
}

\centerline{}

\centerline{}
\title[Persistent Model of the Second Configuration Space of Metric Star Graphs]{Persistent Combinatorial Model of the Restricted Second Configuration Space of Metric Star Graphs}

\author[W. Li, M. \"Ozayd\i n]{Wenwen Li$^1$$^{*}$ and Murad \"Ozayd\i n$^2$}

\address{$^{1}$ Department of Mathematics, Florida State University, Tallahassee, USA.}
\email{\textcolor[rgb]{0.00,0.00,0.84}{wli13@fsu.edu}}

\address{$^{2}$ Department of Mathematics, University of Oklahoma, Norman, USA.}
\email{\textcolor[rgb]{0.00,0.00,0.84}{mozaydin@ou.edu}}




\date{\today
\newline \indent $^{*}$ Corresponding author}

\begin{abstract}
In this work, we present explicit constructions and computations of representative cycles for a nontrivial 2-parameter persistence module arising from the configuration space of metric star graphs. For all edge-length vector $\mathbf{L}=(L_1,L_2,\dots, L_k)\in(\mathbb{R}_{>0})^k$, we construct a bipartite weighted graph $(G_k)_{\mathbf{L}}$ and define filtering functions on the set of vertices and set of edges of $(G_k)_{\mathbf{L}}$ to obtain a filtration (denoted by $(G_k)_{-,\mathbf{L}}$) consisting of geometric realization of subgraphs of $(G_k)_{\mathbf{L}}$. We show that such a filtration is naturally isomorphic to the filtration of the restricted second configuration space of metric star graphs $(\mathsf{Star}_k)^2_{r,\mathbf{L}}$ concerning the restraint parameter $r$ and an (arbitrary but fixed) edge-length vector $\mathbf{L}$. Additionally, we show that the filtration $(G_k)_{-,\mathbf{L}}$ is compatible with the edge-length vector $\mathbf{L}$ up to isotopy, establishing an equivalence between the associated $(k+1)$-parameter persistence modules $PH_{i}((\mathsf{Star}_k)^2_{-,-};\mathbb{F})$ and $PH_{i}((G_k)_{-,-};\mathbb{F})$. We call the (multi-)filtration $(G_k)_{-,-}$ a \textit{persistent combinatorial model} of the multifiltration $(\mathsf{Star}_k)^2_{-,-}$. Using this model, we construct explicit compatible cycle representatives for $PH_{1}((\mathsf{Star}_k)^2_{-,-};\mathbb{F})$ in the bifiltration obtained by fixing $L_2, \dots, L_k > 0$ and varying only $r$ and $L_1$.
\end{abstract} 
\maketitle

\section{Introduction}

Configuration spaces of graphs have been an active subject of study in geometric topology and applied algebraic topology. The configuration space of a graph can be interpreted as the space of all possible arrangements of robots moving along the graph without colliding \cite{abrams2002finding}. While classical treatments model robots as points on a graph, this idealization omits the fact that real robots are thick particles. In practice, a minimal separation distance $r$ must be maintained between any pair of robots, which naturally leads to the study of \key{restricted} configuration spaces of metric graphs.

A metric graph is a geometric realization of a graph whose edges are assigned positive numbers representing their lengths. In a metric graph, each edge is identified with a closed interval of the corresponding length. We denote by $\delta$ the path metric on this space, i.e., for any points $x_1$ and $x_2$ in the metric graph, 
\begin{equation}
 \delta(x_1,x_2)=\inf\{\mathsf{Length}(\gamma): \mbox{$\gamma$ is a path connecting $x_1$ and $x_2$}\}   
\end{equation}

If the edge set of the graph is $E = \{e_1,\dots,e_k\}$, we write $\mathbf{L} = (L_1,\dots,L_k)$ for the collection of edge lengths, where $L_i > 0$ is the length assigned to the edge $e_i$. The vector $\mathbf{L}$ will be referred to as the \key{edge-length vector} of the metric graph. 

For a metric graph $(X_\mathbf{L},\delta)$ where $\delta$ is induced by the edge-length vector $\mathbf{L}$, the restricted second configuration space is 
\begin{equation}
X^2_{r,\mathbf{L}}=\{(x_1,x_2)\in X_{\mathbf{L}}^2: \delta(x_1,x_2)\geq r\}    
\end{equation}

Varying the restraint parameter $r$ and edge-length vector $\mathbf{L}$ produces a multifiltration of restricted configuration spaces, denoted by $X^2_{-,-}$. Applying the homology functor with a field coefficient to this multifiltration yields a multiparameter persistence module, denoted by $PH_i(X^2_{-,-};\mathbb{F})$. When $L_2,\dots,L_k$ are fixed, $X^2_{-,-}$ is a bifiltration and $PH_i(X^2_{-,-};\mathbb{F})$ is a $2$-parameter persistence module.


Although $2$-parameter persistence modules are difficult to compute or decompose in general, in this paper we provide explicit computations and decompositions of a concrete but highly nontrivial example---the first persistent homology of the bifiltration of the restricted configuration spaces of metric star graphs.

A \key{star graph}, denoted by $\mathsf{Star}_k$ (where $k\geq 3$), is a connected simple graph with $k+1$ vertices such that there is precisely one vertex connecting to all other vertices and all other vertices have degree one. See Figure~\ref{fig:matricstar}. For such a graph with the edge-length vector $\mathbf{L}=(L_1,L_2,\dots,L_k)$, its restricted second configuration space of metric star graph is denoted by $(\mathsf{Star}_k)^2_{r,\mathbf{L}}$, where $r$ is the restraint parameter. Figure~\ref{fig:1:intro:H1} gives the $2$-parameter persistence module $PH_1((\mathsf{Star}_k)^2_{-,-};\mathbb{F})$ when $L_2=\cdots=L_k=1$, where the dimension $PH_1((\mathsf{Star}_k)^2_{r,\mathbf{L}};\mathbb{F})$ is provided by the number labeled for each chamber and the morphism assigned for each comparable pair of parameters is induced by the inclusion on the space level~\cite{li2023persistent}.


\begin{minipage}[t]{0.4\textwidth}
    \centering
    
\resizebox{!}{3cm}{
\tikzset{every picture/.style={line width=0.75pt}} 

\begin{tikzpicture}[x=0.75pt,y=0.75pt,yscale=-1,xscale=1]

\draw    (63,90) -- (91,78) ;
\draw [shift={(91,78)}, rotate = 336.8] [color={rgb, 255:red, 0; green, 0; blue, 0 }  ][fill={rgb, 255:red, 0; green, 0; blue, 0 }  ][line width=0.75]      (0, 0) circle [x radius= 3.35, y radius= 3.35]   ;
\draw [shift={(63,90)}, rotate = 336.8] [color={rgb, 255:red, 0; green, 0; blue, 0 }  ][fill={rgb, 255:red, 0; green, 0; blue, 0 }  ][line width=0.75]      (0, 0) circle [x radius= 3.35, y radius= 3.35]   ;
\draw    (91,78) -- (126,80) ;
\draw [shift={(126,80)}, rotate = 3.27] [color={rgb, 255:red, 0; green, 0; blue, 0 }  ][fill={rgb, 255:red, 0; green, 0; blue, 0 }  ][line width=0.75]      (0, 0) circle [x radius= 3.35, y radius= 3.35]   ;
\draw [shift={(91,78)}, rotate = 3.27] [color={rgb, 255:red, 0; green, 0; blue, 0 }  ][fill={rgb, 255:red, 0; green, 0; blue, 0 }  ][line width=0.75]      (0, 0) circle [x radius= 3.35, y radius= 3.35]   ;
\draw    (91,78) -- (91,115.67) ;
\draw [shift={(91,115.67)}, rotate = 90] [color={rgb, 255:red, 0; green, 0; blue, 0 }  ][fill={rgb, 255:red, 0; green, 0; blue, 0 }  ][line width=0.75]      (0, 0) circle [x radius= 3.35, y radius= 3.35]   ;
\draw [shift={(91,78)}, rotate = 90] [color={rgb, 255:red, 0; green, 0; blue, 0 }  ][fill={rgb, 255:red, 0; green, 0; blue, 0 }  ][line width=0.75]      (0, 0) circle [x radius= 3.35, y radius= 3.35]   ;
\draw    (91,78) -- (105,37) ;
\draw [shift={(105,37)}, rotate = 288.85] [color={rgb, 255:red, 0; green, 0; blue, 0 }  ][fill={rgb, 255:red, 0; green, 0; blue, 0 }  ][line width=0.75]      (0, 0) circle [x radius= 3.35, y radius= 3.35]   ;
\draw [shift={(91,78)}, rotate = 288.85] [color={rgb, 255:red, 0; green, 0; blue, 0 }  ][fill={rgb, 255:red, 0; green, 0; blue, 0 }  ][line width=0.75]      (0, 0) circle [x radius= 3.35, y radius= 3.35]   ;
\draw    (91,78) -- (64,44) ;
\draw [shift={(64,44)}, rotate = 231.55] [color={rgb, 255:red, 0; green, 0; blue, 0 }  ][fill={rgb, 255:red, 0; green, 0; blue, 0 }  ][line width=0.75]      (0, 0) circle [x radius= 3.35, y radius= 3.35]   ;
\draw [shift={(91,78)}, rotate = 231.55] [color={rgb, 255:red, 0; green, 0; blue, 0 }  ][fill={rgb, 255:red, 0; green, 0; blue, 0 }  ][line width=0.75]      (0, 0) circle [x radius= 3.35, y radius= 3.35]   ;

\draw (113.98,40.79) node [anchor=north west][inner sep=0.75pt]  [rotate=-51.21]  {$\cdots $};
\draw (93,119.07) node [anchor=north west][inner sep=0.75pt]    {$1$};
\draw (94,80.4) node [anchor=north west][inner sep=0.75pt]    {$0$};
\draw (47,84.4) node [anchor=north west][inner sep=0.75pt]    {$2$};
\draw (52,20.4) node [anchor=north west][inner sep=0.75pt]    {$3$};
\draw (106,13.4) node [anchor=north west][inner sep=0.75pt]    {$4$};
\draw (135,72.4) node [anchor=north west][inner sep=0.75pt]    {$k$};

\end{tikzpicture}
}
    \captionof{figure}{$\mathsf{Star}_k$}
    \label{fig:matricstar}
\vspace{1em}
\end{minipage}%
\begin{minipage}[t]{0.6\textwidth}
\centering
\resizebox{!}{0.65\textwidth}{

\tikzset{every picture/.style={line width=0.75pt}} 

\begin{tikzpicture}[x=0.75pt,y=0.75pt,yscale=-1,xscale=1]

\draw  (6.75,299.04) -- (373.46,299.04)(43.42,46.75) -- (43.42,327.07) (366.46,294.04) -- (373.46,299.04) -- (366.46,304.04) (38.42,53.75) -- (43.42,46.75) -- (48.42,53.75)  ;
\draw  [draw opacity=0][fill={rgb, 255:red, 126; green, 211; blue, 33 }  ,fill opacity=0.62 ] (43.42,299.04) -- (139.57,203.47) -- (139.82,298.78) -- cycle ;
\draw  [draw opacity=0][fill={rgb, 255:red, 248; green, 231; blue, 28 }  ,fill opacity=0.8 ] (138.9,52.47) -- (138.9,106.03) -- (139.57,203.47) -- (43.42,299.04) -- (43.93,52.47) -- cycle ;

\draw (139.47,313.34) node   [align=left] {1};
\draw (236.73,312.19) node   [align=left] {2};
\draw (35.35,203.5) node   [align=left] {1};
\draw (34.21,108.53) node   [align=left] {2};
\draw (17.07,17.96) node [anchor=north west][inner sep=0.75pt]    {$L_1$};
\draw (376.76,302.52) node [anchor=north west][inner sep=0.75pt]    {$r$};
\draw (51,146.4) node [anchor=north west][inner sep=0.75pt]    {$k( k-3) +1$};
\draw (59,279.4) node [anchor=north west][inner sep=0.75pt]    {$k( k-5) +5$};

\end{tikzpicture}
}
\captionof{figure}{$PH_1((\mathsf{Star_k})^2_{-,-};\mathbb{F})$, $k\geq 4$ and \\
$\mathbf{L}=(L_1,1,\dots,1)$}
\label{fig:1:intro:H1}
  
\end{minipage}

The goal of this paper is to compute the first persistent homology of the bifiltration of second configuration spaces of a metric star graph,  $PH_1((\mathsf{Star}_k)^2_{-,-};\mathbb{F})$, and to analyze its indecomposable direct summands. Given $r$ and $\mathbf{L}$, the structure of the top-dimension cells of $(\mathsf{Star}_k)^2_{r,\mathbf{L}}$ can be characterized by systems of inequalities. However, computing $PH_1((\mathsf{Star}_k)^2_{r,\mathbf{L}};\mathbb{F})$ is difficult and analyzing the persistence module $PH_1((\mathsf{Star}_k)^2_{-,-};\mathbb{F})$ directly from the definition of the configuration space presents significant challenges. One must find a non-degenerate cycle representative for each element of $PH_1((\mathsf{Star}_k)^2_{r,\mathbf{L}};\mathbb{F})$ that is compatible with the structure maps of the persistence module as both $r$ and $\mathbf{L}$ vary, which is not straightforward from the geometry of the restricted configuration spaces.

\subsection{Motivation of the Model}
Before presenting the formal construction, we briefly describe the idea behind the model. The key observation is the following:
Each 2-cell of $(\mathsf{Star}_k)^2_{r,\mathbf{L}}$ deformation-retracts onto a 1-dimensional CW complex determined by a collection of its free facets. These local deformation retractions assemble into a global deformation retraction of the entire configuration space $(\mathsf{Star}_k)^2_{r,\mathbf{L}}$ onto a 1-dimensional CW complex.

Such a deformation retract of $(\mathsf{Star}_k)^2_{r,\mathbf{L}}$ can be viewed as a geometric realization of a \textit{weighted} graph, denoted by $(G_k)_{r,\mathbf{L}}$. Our goal is to construct a multifiltration of such graphs that reflects the multifiltration of the configuration spaces $(\mathsf{Star}_k)^2_{-,-}$. To accomplish this, we begin by fixing the edge-length vector $\mathbf{L}$ and constructing a weighted graph $(G_k)_{\mathbf{L}}$ equipped with filtering functions on its vertices and edges. The vertices of $(G_k)_{\mathbf{L}}$ encode which edges or the center vertex the robots occupy: a vertex $(x,y)$ corresponds to a configuration where one robot lies on the edge of $\mathsf{Star}_k$ incident to a vertex $x$ and the other on the edge incident to a vertex $y$. Edges of $(G_k)_{\mathbf{L}}$ record admissible simultaneous motions of the two robots respecting the proximity parameter $r$, with weights inherited from the edge-length vector $\mathbf{L}$. By defining natural filtering functions on both the vertex and edge sets of $(G_k)_{\mathbf{L}}$, we obtain a filtration $(G_k)_{-,\mathbf{L}}$ indexed by the proximity parameter $r$, where $(G_k)_{r,\mathbf{L}}$ is a weighted subgraph of $(G_k)_{\mathbf{L}}$ for all $r>0$. 


 Moreover, these deformation retractions are compatible across all parameter values of $r$ and $\mathbf{L}$. Consequently, the multifiltration of weighted graphs $(G_k)_{r,\mathbf{L}}$ along with inclusion maps between graphs with comparable parameters provides a discrete model that corresponds to the multifiltration of configuration spaces $(\mathsf{Star}_k)^2_{r,\mathbf{L}}$.

\subsection{Main Results}
To address these obstacles, we introduce a bipartite graph model $(G_k)_{\mathbf{L}}$ for each edge-length vector $\mathbf{L}$. By defining filtering functions on the vertex set and edge set of $(G_k)_{\mathbf{L}}$, we obtain a filtration consisting of subgraphs $(G_k)_{r,\mathbf{L}}$ of $(G_k)_{\mathbf{L}}$. As shown in Lemma~\ref{lemma:deformation}, $(G_k)_{r,\mathbf{L}}$ forms a combinatorial model that captures the restricted configuration space $(\mathsf{Star}_k)^2_{r,\mathbf{L}}$ up to homotopy: for all $r$ and $\mathbf{L}$,
$$\lVert (G_k)_{r,\mathbf{L}}\rVert\simeq (\mathsf{Star}_k)^2_{r,\mathbf{L}}$$
where $\lVert (G_k)_{r,\mathbf{L}}\rVert$ is a geometric realization of the graph $(G_k)_{r,\mathbf{L}}$.

These weighted subgraphs $(G_k)_{r,\mathbf{L}}$ assemble into a multifiltration indexed by the parameters $(r,\mathbf{L})$. This construction is not immediate; establishing the multifiltration requires several technical lemmas, which we verify in section~\ref{sec:3.2-model}.

A central result of this paper, Theorem~\ref{thm:2.8}, establishes an isomorphism between the $(k+1)$-parameter persistence modules $$PH_i((\mathsf{Star}_k)^2_{-,-};\mathbb{F})\cong PH_i(\lVert(G_k)_{-,-}\rVert;\mathbb{F})\qquad\text{for all } i\in\mathbb{N}$$ The $2$-parameter persistence module studied in this paper arises as a special case obtained by fixing the edge lengths $L_2,\dots,L_k$.

Our main computational results, presented in Section~\ref{sec:indecomp}, provide a complete description of the
$2$-parameter persistence module $PH_1((\mathsf{Star}_k)^2_{-,-};\mathbb{F})$ for edge-length vectors $\mathbf{L}$, where $L_2,\dots,L_k$ are fixed positive real numbers, and identify all of its indecomposable direct summands. Because the full hyperplane arrangement of the restricted second configuration spaces of $\mathsf{Star}_k$ is highly complex (as it depends on the choice of $L_2,\dots, L_k$), we first pass to the persistent model $(G_k)_{-,-}$ and simplify the
arrangement by discarding hyperplanes that do not affect first homology (Theorem~\ref{thm:4.4}, Theorem~\ref{thm:4.6-1}). Using this reduced hyperplane arrangement, we compute the rank of $PH_1((\mathsf{Star}_k)^2_{r,\mathbf{L}};\mathbb{F})$ for all $r>0$ and $L>0$. In Section~\ref{sec:4.2-decomp}, we perform a complete analysis of the $2$-parameter persistence module $PH_1((\mathsf{Star}_k)^2_{-,-};\mathbb{F})$ for $k=4$, where we identify all of its indecomposable summands explicitly (Theorem~\ref{thm:5.8}). We then use these computations as the base step of induction (on $k$), proving that
$PH_1((\mathsf{Star}_k)^2_{-,-};\mathbb{F})$ is interval decomposable for all $k\ge 4$(Theorem~\ref{thm:5.10}). In addition, we provide explicit descriptions of all interval modules that appear in this
decomposition.

\subsection{Organization of the Paper}
The paper is organized as follows. Section~\ref{sec:prelim} reviews restricted configuration spaces of metric graphs and basic notions of multiparameter persistence modules. Section~\ref{sec:bipartite-model} introduces the bipartite weighted graph $(G_k)_{\mathbf{L}}$ associated to a metric star graph with the edge-length vector $\mathbf{L}$ and construct the multifiltration $(G_k)_{-,-}$, together with its relation to the multifiltration $(\mathsf{Star}_k)^2_{-,-}$. In Section~\ref{sec:indecomp}, we study the hyperplane arrangement of $(\mathsf{Star}_k)^2_{r,\mathbf{L}}$ (when $L_2,\dots,L_k$ are fixed) in the $(r,L_1)$-parameter space and analyze the hyperplanes that control the first homology of $(\mathsf{Star}_k)^2_{r,\mathbf{L}}$. This analysis enables us to simplify the hyperplane arrangement. We then utilize the persistent bipartite model $(G_k)_{-,-}$, and the resulting chamber decomposition (of the reduced hyperplane arrangement) to determine the indecomposable direct summands of $PH_1((\mathsf{Star}_k)^2_{-, -};\mathbb{F})$ and to show that this 2-parameter persistence module is interval-decomposable. 
\section{Preliminaries}\label{sec:prelim}

\subsection{Restricted Configuration Spaces of Metric Graphs}
Throughout this paper, we assume that all graphs are finite and simple.

A weighted graph $X=(V,E,\psi,\omega)$ consists of a set of vertices $V$ and a set of edges $E$ together with an incidence function $\psi: E\rightarrow {V \choose 2}$ and a weight function $\omega: E\rightarrow \mathbb{R}_{>0}$.

We define the category of weighted graphs (denoted by $\cat{wGraph}$) as follows: objects are edge-weighted graphs and morphisms between weighted graphs are graph morphisms such that the weight of the image of an edge in the target graph is no smaller than the weight of the edge in the source graph. We use $\cat{wFinGraph}$ to denote the category of weighted finite graphs, which is a full subcategory of  $\cat{wGraph}$.

\begin{definition}\label{def:2:1}
    Let $X=(V,E,\psi,\omega)$ be a weighted graph. 
    Choose an ordered pair $(u,v)$ for each edge $e\in E$ with $\psi(e)=\{u,v\}$.
    The \key{geometric realization} of $X$ is the quotient space
    $$\lVert X_{\mathbf{L}}\rVert=\left[ \left.\bigsqcup\limits_{v\in V}\{v\} \sqcup \bigsqcup\limits_{e\in E}[0,\omega(e)]\right]\right/\sim$$
    Where $\mathbf{L}:=(\omega(e))_{e\in E}$ is an edge-length vector and $\sim$ is an equivalence relation generated by: $u\sim 0\in [0,\omega(e)]$ and $v\sim \omega(e)\in [0,\omega(e)]$ if $\psi(e)=\{u,v\}$. The topological space $\lVert X_{\mathbf{L}}\rVert$ with the path metric $\delta$ is called the \key{metric graph} of $X$. 
\end{definition}

\begin{remark}
A weighted graph $X=(V,E,\psi,\omega)$ is purely a combinatorial object: it consists of a vertex set, an edge set, an incidence map, and a weight function assigning a positive real number to each edge. In contrast, the metric graph $\lVert X_{\mathbf{L}}\rVert$ is a topological space equipped with the path metric induced by the edge-length vector $\mathbf{L}$.
\end{remark}

\begin{lemma}\label{chap:2:lemma:2}
    $\lVert \cdot \rVert$ is a functor from the
    category of weighted finite graphs to the category of metric spaces with continuous functions.
\end{lemma}
\begin{proof}
    Let  $X=(V,E,\psi,\omega)$ and  $X'=(V',E',\psi',\omega')$ be finite weighted graphs. For any morphism $f: X\rightarrow X'$ in $\cat{wFinGraph}$ (which is a map $f: V\rightarrow V'$), define $\lVert f\rVert: \lVert X_\mathbf{L}\rVert \rightarrow \lVert X'_{\mathbf{L}'}\rVert$ by $$\lVert f\rVert(x)=(1-t_x)f(u_x)+t_xf(v_x)$$
    where $x=(1-t_x)u_x+t_xv_x$ for some $\{u_x,v_x\}\in E$ and $0\leq t_x\leq 1$. It is clear that $\lVert \id_X\rVert=\id_{\lVert X\rVert}$ for all $X$. Moreover, given morphisms $f:X\rightarrow X'$ and $g:X'\rightarrow X''$ in $\cat{wFinGraph}$, \begin{equation}
        \begin{aligned}
            \lVert g\rVert(\lVert f\rVert(x))&=\lVert g\rVert((1-t_x)f(u_x)+t_xf(v_x))\\
            &=(1-t_x)g(f(u_x))+t_xg(f(v_x))\\
            &=(1-t_x)(g\circ f)(u_x)+t_x(g\circ f)(v_x)\\
            &= \lVert g\circ f\rVert(x)
        \end{aligned}
    \end{equation} 
    
    Now we show $\lVert f\rVert$ is continuous for any morphism $f:X\rightarrow X'$ in $\cat{wFinGraph}$. To avoid confusion, we use $d_{X_{\mathbf{L}}}$ (instead of $\delta$) to denote the path metric on the metric graph $X_{\mathbf{L}}$, reserving the symbol $\delta$ for the $\delta$ in the $\varepsilon-\delta$ definition of continuity.

    Define $$k=\max\left\{\frac{\omega'\{f(u),f(v)\}}{\omega\{u,v\}}: \{u,v\}\in E\right\}$$
    Then for all $\varepsilon>0$, there exists $\delta=\frac{\varepsilon}{k}$, such that for all $x\in N_{\delta}(x_0)$,
    \begin{equation}
        \begin{aligned}
            d_{X_{\mathbf{L}'}'}(\lVert f\rVert(x),\lVert f\rVert(x_0))
            &\leq k d_{X_{\mathbf{L}}}(x,x_0)<k\delta=k\frac{\varepsilon}{k}=\varepsilon\\
        \end{aligned}
    \end{equation}
Hence $\lVert f\rVert$ is continuous.

\end{proof}

\begin{remark}
For each weighted finite graph, its geometric realization carries a natural path metric induced by the edge weights, making it a metric space. As shown in Lemma~\ref{chap:2:lemma:2}, any morphism between weighted graphs induces a continuous map between their geometric realizations, and this induced map is uniformly bounded by a constant $k$ since there are only finitely many edges in the source and target graphs. Hence the induced map is Lipschitz.  Thus the functor $\lVert \cdot \rVert$ also defines a functor from weighted finite graphs to the category of metric spaces with Lipschitz maps.
\end{remark}

Given a finite metric graph $(X_{\mathbf{L}},\delta)$ where $\mathbf{L}\in\mathbb{R}_{>0}^{\lvert E\rvert}$ is the edge-length vector, the \key{nth configuration space} of $X_{\mathbf{L}}$ with \key{restraint parameter} $\mathbf{r}=(r_{ij})_{i<j}\in \mathbb{R}_{>0}^{n \choose 2}$ is 
$$X^n_{\mathbf{r},\mathbf{L}}=\{(x_1,\dots,x_n)\in (X_{\mathbf{L}})^n:\delta(x_i,x_j)\geq r_{ij}, \forall i\neq j\}$$

The cellular structure of $X^n_{\mathbf{r},\mathbf{L}}$ can be characterized by \key{parametric polytopes}. 
A parametric polytope $c_{\mathbf{b}}$ in $\mathbb{R}^n$ is determined by $A\mathbf{x}\leq \mathbf{b}$ where $A$ is an $m\times n$ matrix, $\mathbf{x}$ is a $n\times 1$ column vector and $\mathbf{b}$ is a $m\times 1$ column vector. The solution space of $Ax\leq \mathbf{b}$ determines a convex polytope $c_{\mathbf{b}}$ in $\mathbb{R}^n$. 
Note that the family of polytopes $\{c_{\mathbf{b}}\}_{\mathbf{b}\in\mathbb{R}^n}$ has only finitely many combinatorial types~\cite{dover2013homeomorphism}.

Each maximal cell of $X^n_{\mathbf{r},\mathbf{L}}$ consists of the solutions of inequality system:\\

For $1\leq i\leq n$:
\begin{equation}\label{pp1}
0\leq  x_i\leq  L_{e_i}
\end{equation}

For $1\leq i<j\leq n$:
\begin{equation}\label{pp2}
\begin{aligned}
x_i+x_j+\delta(a_i,a_j)&\geq r_{ij};\\
L_{e_i}-x_i+x_j+\delta(b_i,a_j)&\geq r_{ij};\\
x_i+L_{e_j}-x_j+\delta(a_i,b_j)&\geq r_{ij};\\
L_{e_i}-x_i+L_{e_j}-x_j+\delta({b_i,b_j})&\geq r_{ij}
\end{aligned}
\end{equation}

If there exists $1\leq i,j\leq n$ such that $e_i=e_j$, with further assumption that $x_i\leq x_j$, 
\begin{equation}\label{pp3}
\begin{aligned}
x_i-x_j&\leq 0;\\
x_i-x_j&\leq -r_{ij};\\
L_{e_i}-x_i+x_j+\delta(a_i,b_i)&\geq r_{ij}
\end{aligned}
\end{equation}

Let $X:=X_{\mathbf{L}}$ be a metric graph with edge-length vector $\mathbf{L}$. Define an equivalence relation on the parameter space as follows: $\mathbf{r}\sim\mathbf{s}$ if and only if $X^n_{\mathbf{r}}$ and $X^n_{\mathbf{s}}$ have the same combinatorial type. Hence $X^n_{\mathbf{r}}$ is homeomorphic to $X^n_{\mathbf{s}}$ if $\mathbf{r}\sim\mathbf{s}$. Moreover, $\mathbf{r}$ is called \key{critical} if the combinatorial type (face poset) of the cell of $(X_{\mathbf{L}})^n_{\mathbf{r}}$ changes at $\mathbf{r}$. When $X$ is a metric tree, $\mathbf{r}$ is critical only when one of the inequalities in the inequality system (\ref{pp1})(\ref{pp2})(\ref{pp3}) becomes an equality. We obtain critical hyperplanes in the parameter space by considering the critical $\mathbf{r}$ among all possible $\mathbf{L}$.


\begin{definition}
Let $\mathcal{H} = \{H_1, H_2, \dots, H_m\}$ be a finite arrangement of affine hyperplanes in $\mathbb{R}^n$, 
where each hyperplane is defined by
$$H_i = \{x \in \mathbb{R}^n : f_i(x) = 0\}$$
where $f_i : \mathbb{R}^n \to \mathbb{R}$ is an affine linear functional.

Define
$$
H_i^+ = \{x \in \mathbb{R}^n : f_i(x) \ge 0\},
\qquad
H_i^- = \{x \in \mathbb{R}^n : f_i(x) < 0\}.
$$

A \key{chamber} of the hyperplane arrangement $\mathcal{H}$ is a (possibly unbounded) polyhedral region of the form
$$
C = \bigcap_{i=1}^m H_i^{\sigma_i},
\qquad \sigma_i \in \{+, -\}.
$$


\end{definition}

Note that $X^n_{\mathbf{r}}$ and $X^n_{\mathbf{s}}$ have the same combinatorial type if $\mathbf{r}$ and $\mathbf{s}$ are contained in the same chamber. The following Lemma shows that the chambers (as subsets of $\mathbb{R}^n$) are convex.

\begin{lemma}[\cite{dover2013homeomorphism}]
Let $\mathbf{b}'\in\mathbb{R}^n$ and $T_{\mathbf{b}'}$ denote the set of abstract vertices of the parametric polytope $c_{\mathbf{b}}$. If $T_{\mathbf{b}'}\neq\emptyset$, then the set of all $\mathbf{b}\in\mathbb{R}^n$ with $T_{\mathbf{b}}=T_{\mathbf{b}'}$ is convex.
\end{lemma}



\subsection{Persistence Modules}
A \key{persistence module} over a partially ordered set 
$(P,\leq)$ is a functor $M:(P,\leq)\rightarrow \cat{Vect}_{\mathbb{F}}$, where $\mathbb{F}$ is a field and $\cat{Vect}_{\mathbb{F}}$ is the category of vector spaces over $\mathbb{F}$. We use $\cat{Vect}_{\mathbb{F}}^{(P,\leq)}$ to denote the category of persistence modules over $(P,\leq)$ and use $\cat{vect}_{\mathbb{F}}^{(P,\leq)}$ to denote the category of pointwise finite-dimensional persistence modules over $(P,\leq)$. In particular, when $(P,\leq)=(\mathbb{R},\leq)^n$ for some $n\in\mathbb{N}_{>0}$ with the product order, the objects of $\cat{Vect}_{\mathbb{F}}^{(\mathbb{R},\leq)^n}$ are called \key{$n$-parameter persistence modules}.

Let $M\in \cat{vect}_{\mathbb{F}}^{(P,\leq)}$. $M$ is \key{decomposable}\index{decomposable} if there exists non-trivial subrepresentations $N$ and $N'$ of $M$ such that $Mi\cong Ni\oplus N'i$ for all $i\in P$. We say $M$ is \key{indecomposable}\index{indecomposable} if it is not decomposable. The representation $M$ is \key{thin}\index{thin representation} if $\dim(Mi)=0 \mbox{ or }1$ for all $i\in P$.

Let $(P,\leq)$ be a poset. A subset $\cat{I}$ of $P$ is called \key{convex} if $i,k\in\cat{I}$ implies $j\in\cat{I}$ for any $i\leq j\leq k$ in $P$, and is called \key{connected} if for all $a,b\in\cat{I}$, there exists $x_0:=a, x_1,\dots, x_{n-1}, x_n:=b$ in P such that $x_i$ and $x_{i+1}$ are comparable for all $i=0,\dots, n-1$. The set $\cat{I}$ is called an \key{interval} of $(P,\leq)$ if $\cat{I}$ is connected and convex.

Let $(P,\leq)$ be a poset and $\cat{I}\subseteq (P,\leq)$ is an interval. Define the \key{interval module} $\mathbb{F}\cat{I}$ as 
$$\mathbb{F}\cat{I}i = \begin{cases} \mathbb{F}, & \mbox{if } i\in \cat{I};\\ 0, & \mbox{if } i\notin \cat{I}\end{cases} \mbox{\quad and \quad} \hom(\mathbb{F}\cat{I}i, \mathbb{F}\cat{I}j) = \begin{cases} \{\id_{\mathbb{F}}\}, & \mbox{if } i\leq j\in \cat{I};\\ 0, & \mbox{else} \end{cases} $$
When $(P,\leq)=(\mathbb{R},\leq)^n$ and $n\geq 2$, we call $\mathbb{F}\cat{I}$ the \key{polytope module}\index{polytope module} over $\cat{I}$.


In this paper, we will mainly focus on 2-parameter pointwise finite-dimensional persistence modules, i.e., the persistence modules with the indexing poset $(\mathbb{R},\leq)^{2}$. Let $M\in \cat{vect}_{\mathbb{F}}^{(\mathbb{R},\leq)^{2}}$. An \key{isotopy subdivision} of $(\mathbb{R},\leq)^{2}$ subordinate to $M$ is a partition of $(\mathbb{R},\leq)^{2}$ into chambers such that $M(a\leq b)$ is an isomorphism provided that $a\leq b$ are contained in the same chamber. An isotopy subdivision is finite if the number of chambers in the partition is finite. Moreover, we call the partition a convex isotopy subdivision if every chamber in the isotopy subdivision of $(\mathbb{R},\leq)^{2}$ is convex.

Assume $M\in \cat{vect}_{\mathbb{F}}^{(\mathbb{R},\leq)^{2}}$ admits a convex isotopy subdivision. We construct a set (denoted by $P$) and a homogeneous binary relation (denoted by $\rightarrow$) on $P$ as follows:
\begin{itemize}
\item The set of chambers is one-to-one corresponds to $P$. We use $J_p$ to denote the chamber assigned to $p\in P$.
\item Given $p,q\in P$. $p\rightarrow q$ if 
\begin{enumerate}
    \item $p=q$ or 
    \item $p\neq q$, and there exists $x\in J_p$ and $y\in J_q$ such that $x\leq y$.
\end{enumerate}
\end{itemize}

It is clear that $\rightarrow$ is reflexive. Since $(\mathbb{R},\leq)^{2}$ is a poset, $\rightarrow$ is also antisymmetric. Let $\leq$ be the transitive closure of $\rightarrow$. Then $(P,\leq)$ is a poset, where $\leq$ is the transitive closure of $\rightarrow$. Moreover, there is a canonical projection functor $\mathcal{F}_M:(\mathbb{R},\leq)^{2}\rightarrow (P,\leq)$ which sends $(x,y)\in (\mathbb{R},\leq)^{2}$ to $p\in P$, where $J_p$ is the chamber that contains $(x,y)$~\cite{li2024notes}.

\begin{theorem}[\cite{li2024notes}]\label{thm-sim-1'}
Given $M\in\cat{vect}_{\mathbb{F}}^{(\mathbb{R},\leq)^{2} }$. Assume there exists a finite convex isotopy subdivision of $(\mathbb{R},\leq)^{2}$ subordinate to $M$, and let $(P,\leq)$ be the poset of chambers. Then $M$ factors through $(P,\leq)$.

$$
\begin{tikzcd}
(\mathbb{R},\leq)^{2} \arrow[r, rightarrow, "\mathcal{F}_M"] \arrow[dr, rightarrow, "M"']
& (P,\leq) \arrow[d]\\
& \cat{vect}_{\mathbb{F}}
\end{tikzcd}
$$
\end{theorem}

\section{A Persistent Combinatorial Model of $(\mathsf{Star}_k)^2_{-,-}$ }\label{sec:bipartite-model}
Consider the star graph $\mathsf{Star}_k$ with weight $\mathbf{L}=(L_1,\dots,L_k)$. Let $\{0,1,\dots,k\}$ be the set of vertices of $\mathsf{Star}_k$, where $0$ is the central vertex and $i$ is a leaf for all $i=1,\dots, k$. Let $(\mathsf{Star}_k)_{\mathbf{L}}$ be a geometric realization of the weighted graph $\mathsf{Star}_k$ with weight $\mathbf{L}$ equipped with the path metric induced by $\mathbf{L}$. The $(k+1)$-parameter filtration of the second configuration space of $\mathsf{Star}_k$ is a functor $$(\mathsf{Star}_k)^2_{-,-}: (\mathbb{R}_{>0},\leq)^{\mathrm{op}}\times (\mathbb{R}_{>0},\leq)^k\to \cat{Top}$$
where $(\mathsf{Star}_k)^2_{-,-}$ sends each pair of parameters $(r,\mathbf{L})$ to $(\mathsf{Star}_k)^2_{r,\mathbf{L}}$ and sends each comparable pair $(r,\mathbf{L})\leq (r',\mathbf{L}')$ to the inclusion map $(\mathsf{Star}_k)^2_{r,\mathbf{L}}\hookrightarrow (\mathsf{Star}_k)^2_{r',\mathbf{L}'}$.

In this section we construct a \emph{persistent combinatorial model} of $(k+1)$-parameter filtration $(\mathsf{Star}_k)^2_{-,-}$, denoted by $(G_k)_{-,-}$, which provides a combinatorial representation of the multifiltration of $(\mathsf{Star}_k)^2_{-,-}$.


A key step of this section is to show that, for each fixed $r$ and $\mathbf{L}$, the geometric realization of $(G_k)_{r,\mathbf{L}}$ is a strong deformation retract to $(\mathsf{Star}_k)^2_{r,\mathbf{L}}$, see Lemma~\ref{lemma:deformation}. It implies that for arbitrary edge-length $\mathbf{L}$, such a filtration captures the same topological information as the filtration $(\mathsf{Star}_k)^2_{-,\mathbf{L}}$, as established in Theorem~\ref{thm:2.1}.

Next we extend the filtration $(G_k)_{-,\mathbf{L}}$ to a multifiltration $$(G_k)_{-,-}:(\mathbb{R}_{>0},\leq)^{\mathrm{op}}\times (\mathbb{R}_{>0},\leq)^k\to \cat{wFinGraph}$$
which encodes the simultaneous variation of the proximity parameter $r$ and the edge-length vector $\mathbf{L}$. We conclude this section by proving that  for all $i\in\mathbb{N}$ and $k\geq 3$, $$PH_i(\lVert(G_k)_{-,-}\rVert;\mathbb{F})\cong PH_i((\mathsf{Star}_k)^2_{-,-};\mathbb{F})$$ as established in Theorem~\ref{thm:2.8}.

\subsection{The Persistent Combinatorial Model $(G_k)_{-,\mathbf{L}}$ When $\mathbf{L}$ is Fixed}
Before giving the precise definition of the bipartite graph $(G_k)_{\mathbf{L}}$, we briefly indicate the idea behind the construction. Each element of $(\mathsf{Star}_k)^2_{r,\mathbf{L}}$ represents an arrangement of two distinct points with distance at least $r$. The vertices of $(G_k)_{\mathbf{L}}$ encode ordered pairs of vertices of $\mathsf{Star}_k$: a vertex $(x,y)$ of $(G_k)_{\mathbf{L}}$ represents the configuration in which the first point lies at $x$ and the second at $y$. Edges correspond to admissible motions in which one entry of the ordered pair $(x,y)$ varies while the other remains fixed, subject to the separation constraint. The weight of such an edge is given by the length of the edge (of $\mathsf{Star}_k$) along which the moving point varies.

For all $k\geq 3$ and $\mathbf{L}\in(\mathbb{R}_{>0},\leq)^k$, the weighted graph $(G_k)_{\mathbf{L}}$ consisting of the following data:
\begin{description}
    \item[\quad Set of Vertices ($V$)]  $V=\{(x,y): x,y\in\{0,1,\dots,k\}, x\neq y\}$. We use $xy$ as an abbreviation of $(x,y)$ when there is no risk of confusion.
    \item[\quad Set of Edges ($E$)] For all $x_1y_1, x_2y_2\in V$ (assume $x_1\neq x_2$ or $y_1\neq y_2$). $x_1y_1$ and $x_2y_2$ are adjacent if 
    \begin {enumerate*}[label=\itshape(\alph*\upshape)]
    \item $x_1= x_2$ and $y_1\cdot y_2=0$ or \item $y_1= y_2$ and $x_1\cdot x_2=0$. 
    \end {enumerate*} This gives the incidence function $\psi: E\rightarrow {V \choose 2}$.
     \item[\quad Weight Function $\omega$]  Define $\omega: E\rightarrow \mathbb{R}_{>0}$ by
    $$\omega(\{x_1y_1, x_2y_2\})=\begin{cases} L_{x_1}+L_{x_2}, & \mbox{if } x_1\cdot x_2=0 \mbox{ and } y_1=y_2 \\ 
    L_{y_1}+L_{y_2}, & \mbox{if } y_1\cdot y_2=0 \mbox{ and } x_1=x_2 
    \end{cases}$$
    where $L_x$ (resp. $L_y$) is the length of the edge that contains $x$ (resp. $y$) as a leaf of $\mathsf{Star}_k$ and we set $L_0 := 0$. 
\end{description}




We define a \textit{filtering function on the vertex set $V$}, denoted by $f_V:V\rightarrow \mathbb{R}_{>0}$, as follows: for all $xy\in V$, $$f_V(xy)=L_x+L_y$$


In addition, we define a \textit{filtering function on the edge set $E$}, denoted by $f_E:E\rightarrow \mathbb{R}_{>0}$, as follows: for all $\{x_1y_1, x_2y_2\}\in E$, 
$$f_E(\{x_1y_1, x_2y_2\})=\begin{cases} L_y, & \mbox{if } x_1\cdot x_2=0 \mbox{ and } y_1=y_2=y \\ 
L_x, & \mbox{if } y_1\cdot y_2=0 \mbox{ and } x_1=x_2=x 
\end{cases}$$
where $L_x$ (resp. $L_y$) is the length of the edge that contains $x$ (resp. $y$) as a leaf of $\mathsf{Star}_k$.

For each $r > 0$, define 
$$V_r=f_V^{-1}[r,\infty) \qquad E_r=f_E^{-1}[r,\infty)$$ 
Let $\psi_r$ and $\omega_r$ be the restrictions of $\psi$ and $\omega$ to $E_r$, respectively.


\begin{lemma}
The map $\psi_r$ is an incidence function and $(G_k)_{r,\mathbf{L}}=(V_r, E_r,\psi_r,\omega_r)$ is a weighted subgraph of $(G_k)_{\mathbf{L}}$. 
\end{lemma}

Now we define $(G_k)_{-,\mathbf{L}}$, which is a filtration of subgraphs of $(G_k)_{\mathbf{L}}$ indexed by the poset  $(\mathbb{R}_{>0},\leq)$. For all $\mathbf{L}\in(\mathbb{R}_{>0},\leq)^k$, the filtration $(G_k)_{-,\mathbf{L}}$ is a contravariant functor $$(G_k)_{-,\mathbf{L}}: (\mathbb{R}_{>0},\leq)\op\rightarrow \cat{wFinGraph}$$ that consists of the following data: 

\begin{itemize}
    \item Assign $(G_k)_{r,\mathbf{L}}$ for each $r\in (\mathbb{R}_{>0},\leq)$, where $(G_k)_{r,\mathbf{L}}=(V_r, E_r,\psi_r,\omega_r)$.
    \item For all $r\leq r'\in (\mathbb{R}_{>0},\leq)$, $(G_k)_{r\leq r',\mathbf{L}}:(G_k)_{ r',\mathbf{L}}\rightarrow (G_k)_{r,\mathbf{L}}$ is the inclusion $V_{r'}\hookrightarrow V_r$ (in $\cat{wFinGraph}$) where $V_r$ and $V_{r'}$ are vertex sets of $(G_k)_{r,\mathbf{L}}$ and $(G_k)_{r',\mathbf{L}}$, respectively, where $V_r$ is the vertex set of $(G_k)_{r,\mathbf{L}}$ and $V_{r'}$ is the vertex set of $(G_k)_{r',\mathbf{L}}$.

\end{itemize}

\begin{remark}
As the restraint parameter $r$ decreases, more vertex and edge configurations become feasible in the \key{superlevel filtration} determined by $f_V$ and $f_E$. Thus the functor $(G_k)_{-,\mathbf{L}}$ is contravariant with respect to $(\mathbb{R}_{>0},\le)$: if $r \leq r'$, then $(G_k)_{r',\mathbf{L}} \subseteq (G_k)_{r,\mathbf{L}}$.
\end{remark}

To justify that the geometric realization of $(G_k)_{r,\mathbf{L}}$ (i.e., $\lVert(G_k)_{r,\mathbf{L}}\rVert$) reflects the topology of the corresponding configuration spaces, we first show that each $\lVert(G_k)_{r,\mathbf{L}}\rVert$ is homotopy equivalent to $(\mathsf{Star}_k)^2_{r,\mathbf{L}}$ for all $r>0$ and $\mathbf{L}\in(\mathbb{R}_{>0})^k$.

\begin{lemma}\label{lemma:deformation}
For all $r\in\mathbb{R}_{>0}$ and $\mathbf{L}\in(\mathbb{R}_{>0},\leq)^k$, $\lVert(G_k)_{r,\mathbf{L}}\rVert\simeq (\mathsf{Star}_k)^2_{r,\mathbf{L}}$. In fact, $\lVert(G_k)_{r,\mathbf{L}}\rVert$ is a strong deformation retract of $(\mathsf{Star}_k)^2_{r,\mathbf{L}}$.
\end{lemma}

The proof of Lemma~\ref{lemma:deformation} is a lengthy but routine verification involving constructions of the deformation retraction for each cell of $(\mathsf{Star}_k)^2_{r,\mathbf{L}}$. Since it is technical and not conceptually illuminating, we defer it to Appendix~\ref{appendix:lemma-3.3}.

Lemma~\ref{lemma:deformation} implies that the following diagram is commutative up to homotopy for all $r\leq r'\in(\mathbb{R}_{>0},\leq)$: 

\begin{equation}\label{eq:model:6}
\begin{tikzcd}
\lVert(G_k)_{r',\mathbf{L}}\rVert \arrow[hookrightarrow]{r}{\lVert(G_k)_{r\leq r',\mathbf{L}}\rVert } \arrow[swap]{d}{\simeq} & \lVert(G_k)_{r,\mathbf{L}}\rVert \arrow{d}{\simeq} \\%
(\mathsf{Star}_k)^2_{r',\mathbf{L}} \arrow[hookrightarrow]{r}{(\mathsf{Star}_k)^2_{r\leq r',\mathbf{L}} }& (\mathsf{Star}_k)^2_{r,\mathbf{L}}
\end{tikzcd}
\end{equation}

Applying the functor $H_i(-;\mathbb{F})$ to Diagram~(\ref{eq:model:6}), we obtain

\begin{equation}\label{eq:model:7}
\begin{tikzcd}
H_i(\lVert(G_k)_{r',\mathbf{L}}\rVert;\mathbb{F}) \arrow{r}{} \arrow[swap]{d}{\cong} & H_i(\lVert(G_k)_{r,\mathbf{L}}\rVert;\mathbb{F}) \arrow{d}{\cong} \\%
H_i((\mathsf{Star}_k)^2_{r',\mathbf{L}};\mathbb{F}) \arrow{r}{}& H_i((\mathsf{Star}_k)^2_{r,\mathbf{L}};\mathbb{F})
\end{tikzcd}
\end{equation}

The functoriality of $H_i(-;\mathbb{F})$ ensures that the Diagram~(\ref{eq:model:7}) is commutative.

For fixed $\mathbf{L}\in\mathbb{R}_{>0}$, define a filtration of $\lVert(G_k)_{\mathbf{L}}\rVert$ by
$$\lVert(G_k)_{-,\mathbf{L}}\rVert:=\lVert\cdot\rVert\circ (G_k)_{-,\mathbf{L}}$$

More specifically, $\lVert(G_k)_{-,\mathbf{L}}\rVert$ assign a topological space $\lVert(G_k)_{r,\mathbf{L}}\rVert$ for each $r\in (\mathbb{R}_{>0},\leq)$ and assign a continuous function $\lVert(G_k)_{r\leq r',\mathbf{L}}\rVert:\lVert(G_k)_{ r',\mathbf{L}}\rVert\rightarrow \lVert(G_k)_{r,\mathbf{L}}\rVert$ that is induced by the graph monomorphism $(G_k)_{r\leq r',\mathbf{L}}: (G_k)_{r',\mathbf{L}}\hookrightarrow (G_k)_{r,\mathbf{L}}$ for all $r\leq r'\in (\mathbb{R}_{>0},\leq)$. 

Thus, the composition $PH_i(\lVert(G_k)_{-,\mathbf{L}}\rVert;\mathbb{F})=H_i(-;\mathbb{F})\circ \lVert(G_k)_{-,\mathbf{L}}\rVert$ defines a persistence module. 
The preceding discussion establishes the following theorem.
\begin{theorem}\label{thm:2.1}
   For all $\mathbf{L}\in(\mathbb{R}_{>0},\leq)^k$, $PH_i(\lVert(G_k)_{-,\mathbf{L}}\rVert;\mathbb{F})\cong PH_i((\mathsf{Star}_k)^2_{-,\mathbf{L}};\mathbb{F})$ for all $i\in\mathbb{N}$.
\end{theorem}

\subsection{The Multiparameter Persistent combinatorial Model $(G_k)_{-,-}$}\label{sec:3.2-model}
The filtration $(G_k)_{-,\mathbf{L}}$ constructed in the previous subsection captures the topology of filtration $(\mathsf{Star}_k)^2_{-,\mathbf{L}}$ for each fixed edge-length vector $\mathbf{L}$. To obtain a multifiltration of weighted graphs $(G_k)_{r,\mathbf{L}}$ with parameters $r$ and $\mathbf{L}$, we must verify that the construction of $(G_k)_{-,\mathbf{L}}$ is compatible for the edge-length vector $\mathbf{L}$ with respect to the product order.

Consider $\mathbf{L},\mathbf{L}'\in (\mathbb{R},\leq)^k$ with $\mathbf{L}\leq \mathbf{L}'$ (with respect to the product order) where $\mathbf{L}=(L_1,L_2,\dots, L_k)$ and $\mathbf{L}'=(L'_1,L'_2,\dots, L'_k)$.
Fix $r>0$, let $V((G_k)_{r,\mathbf{L}})$ and $V((G_k)_{r,\mathbf{L}'})$ be the vertex sets of $(G_k)_{r,\mathbf{L}}$ and $(G_k)_{r,\mathbf{L}'}$, respectively. Define $$(G_k)_{r, \mathbf{L}\leq \mathbf{L}'}: (G_k)_{r,\mathbf{L}}\hookrightarrow (G_k)_{r,\mathbf{L}'}$$ to be the inclusion $V((G_k)_{r,\mathbf{L}})\hookrightarrow V((G_k)_{r,\mathbf{L}'})$. The next remark explains why this inclusion is the natural candidate for a morphism in $\cat{wFinGraph}$.

\begin{remark}
All vertex and edge weights used to define $(G_k)_{r,\mathbf{L}}$ are built
directly from the edge lengths of $(\mathsf{Star}_k)_{\mathbf{L}}$.  Hence if
$\mathbf{L}\le \mathbf{L}'$, the weight of any edge in $(G_k)_{r,\mathbf{L}}$ is
less than or equal to the weight of the corresponding edge in $(G_k)_{r,\mathbf{L}'}$. The inequalities
checked in Lemma~\ref{lem:monotonicity} simply make this monotonicity explicit.
\end{remark}

\begin{lemma}\label{lem:monotonicity}
  $(G_k)_{r, \mathbf{L}\leq \mathbf{L}'}$ is a graph homomorphism in $\cat{wFinGraph}$.
\end{lemma}
\begin{proof}
    We first show that $(G_k)_{r, \mathbf{L}\leq \mathbf{L}'}$ is well-defined. Let $V$ be the set of vertices of $(G_k)_{r,\mathbf{L}}$ and let $xy\in V$. Then $f_V(xy)\geq r$. When $x\neq 0$ and $y\neq 0$, $f_V(xy)=L_x+L_y\leq L'_x+L'_y$. When $y=0$, $f_V(xy)=L_x\leq L'_x$. When $x=0$, $f_V(xy)=L_y\leq L'_y$. Therefore, $xy$ is also a vertex of $(G_k)_{r,\mathbf{L}'}$.

    Let $E$ be the set of edges of $(G_k)_{r,\mathbf{L}}$. If vertices $x_1y_1$ and $x_2y_2$ are adjacent in graph $(G_k)_{r,\mathbf{L}}$, then $f_E(\{x_1y_1,x_2y_2\})\geq r$. When $x:=x_1=x_2$ and $y_1y_2=0$, note that $f_E(\{x_1y_1,x_2y_2\})=L_x\leq L'_x$, hence $x_1y_1$ and $x_2y_2$ are also adjacent in graph $(G_k)_{r,\mathbf{L}'}$. When $y:=y_1=y_2$ and $x_1x_2=0$, note that $f_E(\{x_1y_1,x_2y_2\})=L_y\leq L'_y$, hence $x_1y_1$ and $x_2y_2$ are also adjacent in graph $(G_k)_{r,\mathbf{L}'}$. Therefore, $(G_k)_{r, \mathbf{L}\leq \mathbf{L}'}$ is a graph homomorphism. 
    
    Let $\omega_{r,\mathbf{L}}$ and $\omega_{r,\mathbf{L}'}$ denote the weight functions defined for $(G_k)_{r,\mathbf{L}}$ and $(G_k)_{r,\mathbf{L}'}$, respectively. For all $x_1y_1$ and $x_2y_2$ that are adjacent in graph $(G_k)_{r,\mathbf{L}}$ (Without loss of generality we may assume $x_1\cdot x_2=0$ and $y_1=y_2$), note that $$\omega_{r,\mathbf{L}'}(g_{r,\mathbf{L}\leq \mathbf{L}'}\{x_1y_1, x_2y_2\})=L'_{x_1+x_2}\geq L_{x_1+x_2}=\omega_{r,\mathbf{L}}(\{x_1y_1, x_2y_2\})$$
    
    Therefore, $(G_k)_{r, \mathbf{L}\leq \mathbf{L}'}$ is a graph homomorphism in $\cat{wFinGraph}$. 
\end{proof}

Since $(G_k)_{r, \mathbf{L}\leq \mathbf{L}'}: (G_k)_{r,\mathbf{L}} \to (G_k)_{r,\mathbf{L}'}$ is a morphism in
$\cat{wFinGraph}$ whenever $\mathbf{L}\le \mathbf{L}'$, and
$(G_k)_{r\leq r', \mathbf{L}}:(G_k)_{r',\mathbf{L}} \to (G_k)_{r,\mathbf{L}}$ is a morphism in
$\cat{wFinGraph}$ whenever
$r\leq r'$, the construction extends naturally to a bifiltration indexed
by $(r,\mathbf{L})\in (\mathbb{R}_{>0},\leq)\op\times (\mathbb{R}_{>0},\leq)^k$.

\begin{definition}
The multifiltration $(G_k)_{-,-}$ is a functor $$(G_k)_{-,-}: (\mathbb{R}_{>0},\leq)\op\times (\mathbb{R}_{>0},\leq)^k\rightarrow \cat{wFinGraph}$$ that consists of the following data: 

\begin{itemize}
    \item On the object level, assign $(G_k)_{r,\mathbf{L}}$ for each $r\in (\mathbb{R}_{>0},\leq)$ and $\mathbf{L}\in (\mathbb{R}_{>0},\leq)^k$, where $(G_k)_{r,\mathbf{L}}=(V_{r,\mathbf{L}}, E_{r,\mathbf{L}},\psi_{r,\mathbf{L}},\omega_{r,\mathbf{L}})$ is a subgraph of $(G_k)_{\mathbf{L}}$.
    \item For all $r\leq r'\in (\mathbb{R}_{>0},\leq)$ and $\mathbf{L}\leq \mathbf{L}'\in (\mathbb{R}_{>0},\leq)^k$, $(G_k)_{r\leq r',\mathbf{L}\leq \mathbf{L}'}:(G_k)_{ r',\mathbf{L}}\rightarrow (G_k)_{r,\mathbf{L}'}$ is the composition of graph homomorphisms $$(G_k)_{r\leq r',\mathbf{L}\leq \mathbf{L}'}=(G_k)_{r\leq r',\mathbf{L}'}\circ (G_k)_{r',\mathbf{L}\leq \mathbf{L}'}$$
\end{itemize}
\end{definition}

We now verify that $(G_k)_{-,-}$ is well-defined.

\begin{theorem}\label{thm:chap:3:8}
  $(G_k)_{-,-}: (\mathbb{R}_{>0},\leq)\op\times (\mathbb{R}_{>0},\leq)^k\rightarrow \cat{wFinGraph}$ is a functor.  
\end{theorem}
\begin{proof}
    For all $r\leq r'\in (\mathbb{R}_{>0},\leq)$ and $\mathbf{L}\leq \mathbf{L}'\in (\mathbb{R}_{>0},\leq)^k$, note that both $(G_k)_{r\leq r',\mathbf{L}'}$ and  $(G_k)_{r',\mathbf{L}\leq \mathbf{L}'}$ are morphisms in $\cat{wFinGraph}$, hence $(G_k)_{r\leq r',\mathbf{L}\leq \mathbf{L}'}=(G_k)_{r\leq r',\mathbf{L}'}\circ (G_k)_{r',\mathbf{L}\leq \mathbf{L}'}$ is a morphism in $\cat{wFinGraph}$.

    When $r=r'$ and $\mathbf{L}=\mathbf{L}'$, both $(G_k)_{r\leq r',\mathbf{L}'}$ and $(G_k)_{r',\mathbf{L}\leq \mathbf{L}'}$ are identity morphisms. Hence $(G_k)_{r\leq r',\mathbf{L}\leq \mathbf{L}'}=(G_k)_{r\leq r',\mathbf{L}'}\circ (G_k)_{r',\mathbf{L}\leq \mathbf{L}'}$ is the identity morphism. 

    Moreover, for all $r\leq r'\leq r''$ and $\mathbf{L}\leq \mathbf{L}'\leq \mathbf{L}''$, note that the following diagram (in $\cat{wFinGraph}$) is commutative:

\begin{equation}\label{eq:model:10}
\begin{tikzcd}
(G_k)_{r'',\mathbf{L}'} \arrow[r,hookrightarrow]{}{(G_k)_{r'', \mathbf{L}'\leq \mathbf{L}''}} \arrow[swap,hookrightarrow]{d}{}{(G_k)_{r'\leq r'', \mathbf{L}'}} & (G_k)_{r'',\mathbf{L}''} \arrow[d,hookrightarrow]{}{(G_k)_{r'\leq r'', \mathbf{L}''}}  \\%
(G_k)_{r',\mathbf{L}'} \arrow[r,hookrightarrow]{}{(G_k)_{r, \mathbf{L}\leq \mathbf{L}'}}& (G_k)_{r',\mathbf{L}''}
\end{tikzcd}
\end{equation}

Therefore,
    \begin{equation}
        \begin{aligned}
            (G_k)_{r\leq r',\mathbf{L}'\leq \mathbf{L}''}\circ (G_k)_{r'\leq r'',\mathbf{L}\leq \mathbf{L}'}&=((G_k)_{r\leq r',\mathbf{L}''}\circ (G_k)_{r',\mathbf{L}'\leq \mathbf{L}''})\circ ((G_k)_{r'\leq r'',\mathbf{L}'}\circ (G_k)_{r'',\mathbf{L}\leq \mathbf{L}'})\\
            &=(G_k)_{r\leq r',\mathbf{L}''}\circ ((G_k)_{r',\mathbf{L}'\leq \mathbf{L}''}\circ (G_k)_{r'\leq r'',\mathbf{L}'})\circ (G_k)_{r'',\mathbf{L}\leq \mathbf{L}'}\\
            &=(G_k)_{r\leq r',\mathbf{L}''}\circ ((G_k)_{r'\leq r'',\mathbf{L}''}\circ (G_k)_{r'',\mathbf{L}'\leq \mathbf{L}''})\circ (G_k)_{r'',\mathbf{L}\leq \mathbf{L}'}\\
            &=((G_k)_{r\leq r',\mathbf{L}''}\circ (G_k)_{r'\leq r'',\mathbf{L}''})\circ ((G_k)_{r'',\mathbf{L}'\leq \mathbf{L}''}\circ (G_k)_{r'',\mathbf{L}\leq \mathbf{L}'})\\
            &=(G_k)_{r\leq r,\mathbf{L}''}\circ \circ (G_k)_{r'',\mathbf{L}\leq \mathbf{L}''}\\
            &=(G_k)_{r\leq r'',\mathbf{L}\leq \mathbf{L}''}
        \end{aligned}
    \end{equation}

\end{proof}

Theorem~\ref{thm:chap:3:8} says that the construction of $(G_k)_{r,\mathbf{L}}$ is compatible with both parameters $r$
and $\mathbf{L}$.

Define the functor $$PH_i(\lVert(G_k)_{-,-}\rVert;\mathbb{F}):(\mathbb{R}_{>0},\leq)\op\times (\mathbb{R}_{>0},\leq)^k\rightarrow\cat{vect}_{\mathbb{F}}$$ by
$$PH_i(\lVert(G_k)_{-,-}\rVert;\mathbb{F})=H_i(-;\mathbb{F})\circ \lVert(G_k)_{-,-}\rVert$$

It sends each $(r,\mathbf{L})\in(\mathbb{R}_{>0},\leq)\op\times (\mathbb{R}_{>0},\leq)^k$ to a vector space $PH_i(\lVert(G_k)_{r,\mathbf{L}}\rVert;\mathbb{F}):=H_i(\lVert(G_k)_{r,\mathbf{L}}\rVert;\mathbb{F})$ and each $(r,\mathbf{L}')\leq (r',\mathbf{L})$ to a linear transformation $PH_i(\lVert(G_k)_{r\leq r',\mathbf{L}\leq \mathbf{L}'}\rVert;\mathbb{F}):=H_i(\lVert(G_k)_{r\leq r',\mathbf{L}'}\rVert;\mathbb{F})\circ H_i(\lVert (G_k)_{r', \mathbf{L}\leq \mathbf{L}'}\rVert;\mathbb{F})$. By the functoriality of $H_i(-;\mathbb{F})$ and $\lVert\cdot\rVert$, we immediately obtain the next theorem.
\begin{theorem}\label{thm:2.7}
    For all $i>0$ and $k\geq 3$, $PH_i(\lVert(G_k)_{-,-}\rVert;\mathbb{F})$ is a persistence module.
\end{theorem}

To relate our graph model $(G_k)_{-,-}$ to the multifiltration $(\mathsf{Star}_k)^2_{-,-})$, we now fix $r>0$ and compare $\lVert(G_k)_{r,\mathbf{L}}\rVert$ and 
$(\mathsf{Star}_k)^2_{r,\mathbf{L}}$
as $\mathbf{L}$ varies. For all $r\in \mathbb{R}_{>0}$ and $\mathbf{L}\leq \mathbf{L}'\in (\mathbb{R}_{>0},\leq)^k$, we have the following diagram (in $\cat{Top}$):

\begin{equation}\label{eq:model:8}
\begin{tikzcd}
\lVert(G_k)_{r,\mathbf{L}}\rVert \arrow{r}{\lVert (G_k)_{r, \mathbf{L}\leq \mathbf{L}'}\rVert} \arrow[swap]{d}{\simeq} & \lVert(G_k)_{r,\mathbf{L}'}\rVert \arrow{d}{\simeq} \\%
(\mathsf{Star}_k)^2_{r,\mathbf{L}}\arrow[hookrightarrow]{r}{}& (\mathsf{Star}_k)^2_{r,\mathbf{L}'}
\end{tikzcd}
\end{equation}
where both vertical arrows are homotopy equivalences by Lemma~\ref{lemma:deformation} and the morphism $\lVert (G_k)_{r, \mathbf{L}\leq \mathbf{L}'}\rVert$ is the continuous map induced by the graph homomorphism $$(G_k)_{r, \mathbf{L}\leq \mathbf{L}'}: (G_k)_{r,\mathbf{L}}\hookrightarrow (G_k)_{r,\mathbf{L}'}$$

To show that Diagram~\ref{eq:model:8} is commutative up to homotopy, we first establish the isotopy equivalence between $\lVert(G_k)_{r,\mathbf{L}}\rVert$ and its image under $\lVert (G_k)_{r,\mathbf{L}\leq \mathbf{L}'}\rVert$, which sit inside the ambient configuration space $(\mathsf{Star}_k)^2_{r,\mathbf{L}'}$.

\begin{lemma}\label{lemma:2.4}
   $\lVert(G_k)_{r,\mathbf{L}}\rVert$ is isotopic\footnote{Let $Z$ be a topological space and let $X$ and $Y$ be subspaces of $Z$. The spaces $X$ and $Y$ are isotopic if there exists a continuous map $H: X\times [0,1]\rightarrow Z$ such that $H(x,0)=x$ for all $x\in X$, $H(X,1)=Y$, and $h_t(-):= H(-,t): X\rightarrow Z$ is a homeomorphism onto its image for all $t\in [0,1]$.} to the image of $\lVert (G_k)_{r, \mathbf{L}\leq \mathbf{L}'}\rVert$. Here, both $\lVert(G_k)_{r,\mathbf{L}}\rVert$ and $\Ima\lVert (G_k)_{r, \mathbf{L}\leq \mathbf{L}'}\rVert$ are subspaces of $(\mathsf{Star}_k)^2_{r,\mathbf{L}'}$.
\end{lemma}

\begin{proof}
Let $\sigma$ be a top-dimension cell of $(\mathsf{Star}_k)^2_{r,\mathbf{L}'}$. When $r\leq \min\{L_{e_i}, L_{e_j}\}$ (where $i\neq j$), $\Ima\lVert (G_k)_{r, \mathbf{L}\leq \mathbf{L}'}\rVert = \lVert (G_k)_{r,\mathbf{L}'} \rVert$. Let $\Lambda v_0v_4v_3$ ( resp. $\Lambda v_0'v_4'v_3'$) denote the intersection of $\lVert (G_k)_{r,\mathbf{L}} \rVert$ (resp. $\Ima\lVert (G_k)_{r, \mathbf{L}\leq \mathbf{L}'}\rVert$) with $\sigma$. 
Define $f^{\sigma}: \Lambda v_0v_4v_3 \rightarrow \Lambda v_0'v_4'v_3'$ by
        $$f^{\sigma}(x)=\begin{cases} (1-s)v_0'+sv_4', & \mbox{if } x=(1-s)v_0+sv_4\\ (1-s)v_3'+sv_4', & \mbox{if }x=(1-s)v_3+sv_4\end{cases}$$
where $0\leq s\leq 1$. 
Define $g^{\sigma}: \Lambda v_0'v_4'v_3' \rightarrow \Lambda v_0v_4v_3$ by
        $$g^{\sigma}(y)=\begin{cases} (1-s)v_0+sv_4, & \mbox{if } y=(1-s)v_0'+sv_4'\\ (1-s)v_3+sv_4, & \mbox{if }y=(1-s)v_3'+sv_4'\end{cases}$$
where $0\leq s\leq 1$.

Note that $G_1^{\sigma}(1,t)=G_2^{\sigma}(1,t)$ for all $t\in [0,1]$, we obtain a well-defined continuous function 
        $$G^{\sigma}(x,t)=\begin{cases} G_1^{\sigma}(s,t), & \mbox{if } x= (1-s) v_0+ sv_4\in v_0v_4 \\ G_2^{\sigma}(s,t), & \mbox{if } x= (1-s) v_3+ sv_4\in v_3v_4  \end{cases}$$

    We construct a homotopy (denoted by $G^{\sigma}(x,t)$) for each top-dimension cell $\sigma$ of $(\mathsf{Star}_k)^2_{r,\mathbf{L}'}$ for the case $r\leq \min\{L_{e_i}, L_{e_j}\}$ such that the restriction of $G^{\sigma}(x,t)$ gives an isotopy of $\lVert(G_k)_{r,\mathbf{L}}\rVert$ and $\Ima\lVert (G_k)_{r, \mathbf{L}\leq \mathbf{L}'}\rVert$ when $r>\min\{L_{e_i}, L_{e_j}\}$, assuming all the fins of $(\mathsf{Star}_k)^2_{r,\mathbf{L}'}$ are collapsed.

\begin{figure}[htbp!]
     \begin{subfigure}[b]{0.3\textwidth}
         \centering    

\resizebox{!}{0.67\textwidth}{

\tikzset{every picture/.style={line width=0.75pt}} 

\begin{tikzpicture}[x=0.75pt,y=0.75pt,yscale=-1,xscale=1]

\draw  [draw opacity=0][fill={rgb, 255:red, 155; green, 155; blue, 155 }  ,fill opacity=0.4 ] (135,59) -- (135,144) -- (63,144) -- (37,111) -- (36,59) -- cycle ;
\draw [color={rgb, 255:red, 0; green, 0; blue, 0 }  ,draw opacity=1 ][line width=0.75]    (37,111) -- (63,144) ;
\draw [line width=0.75]    (37,18) -- (37,111) ;
\draw [line width=0.75]    (63,144) -- (161,144) ;
\draw  [draw opacity=0][fill={rgb, 255:red, 155; green, 155; blue, 155 }  ,fill opacity=0.4 ] (161,18) -- (161,144) -- (63,144) -- (37,111) -- (37,18) -- cycle ;
\draw [color={rgb, 255:red, 80; green, 227; blue, 194 }  ,draw opacity=1 ][line width=0.75]    (37,18) -- (161,18) ;
\draw [shift={(161,18)}, rotate = 0] [color={rgb, 255:red, 80; green, 227; blue, 194 }  ,draw opacity=1 ][fill={rgb, 255:red, 80; green, 227; blue, 194 }  ,fill opacity=1 ][line width=0.75]      (0, 0) circle [x radius= 3.35, y radius= 3.35]   ;
\draw [shift={(37,18)}, rotate = 0] [color={rgb, 255:red, 80; green, 227; blue, 194 }  ,draw opacity=1 ][fill={rgb, 255:red, 80; green, 227; blue, 194 }  ,fill opacity=1 ][line width=0.75]      (0, 0) circle [x radius= 3.35, y radius= 3.35]   ;
\draw [color={rgb, 255:red, 80; green, 227; blue, 194 }  ,draw opacity=1 ][line width=0.75]    (161,18) -- (161,144) ;
\draw [shift={(161,144)}, rotate = 90] [color={rgb, 255:red, 80; green, 227; blue, 194 }  ,draw opacity=1 ][fill={rgb, 255:red, 80; green, 227; blue, 194 }  ,fill opacity=1 ][line width=0.75]      (0, 0) circle [x radius= 3.35, y radius= 3.35]   ;
\draw [shift={(161,18)}, rotate = 90] [color={rgb, 255:red, 80; green, 227; blue, 194 }  ,draw opacity=1 ][fill={rgb, 255:red, 80; green, 227; blue, 194 }  ,fill opacity=1 ][line width=0.75]      (0, 0) circle [x radius= 3.35, y radius= 3.35]   ;
\draw [color={rgb, 255:red, 227; green, 85; blue, 80 }  ,draw opacity=1 ][line width=0.75]    (36,59) -- (135,59) ;
\draw [shift={(135,59)}, rotate = 0] [color={rgb, 255:red, 227; green, 85; blue, 80 }  ,draw opacity=1 ][fill={rgb, 255:red, 227; green, 85; blue, 80 }  ,fill opacity=1 ][line width=0.75]      (0, 0) circle [x radius= 3.35, y radius= 3.35]   ;
\draw [shift={(36,59)}, rotate = 0] [color={rgb, 255:red, 227; green, 85; blue, 80 }  ,draw opacity=1 ][fill={rgb, 255:red, 227; green, 85; blue, 80 }  ,fill opacity=1 ][line width=0.75]      (0, 0) circle [x radius= 3.35, y radius= 3.35]   ;
\draw [color={rgb, 255:red, 227; green, 85; blue, 80 }  ,draw opacity=1 ][line width=0.75]    (135,59) -- (135,144) ;
\draw [shift={(135,144)}, rotate = 90] [color={rgb, 255:red, 227; green, 85; blue, 80 }  ,draw opacity=1 ][fill={rgb, 255:red, 227; green, 85; blue, 80 }  ,fill opacity=1 ][line width=0.75]      (0, 0) circle [x radius= 3.35, y radius= 3.35]   ;
\draw [shift={(135,59)}, rotate = 90] [color={rgb, 255:red, 227; green, 85; blue, 80 }  ,draw opacity=1 ][fill={rgb, 255:red, 227; green, 85; blue, 80 }  ,fill opacity=1 ][line width=0.75]      (0, 0) circle [x radius= 3.35, y radius= 3.35]   ;

\draw (135,142.4) node [anchor=north west][inner sep=0.75pt]  [color={rgb, 255:red, 144; green, 19; blue, 254 }  ,opacity=1 ]  {$v_{0}$};
\draw (50.57,140.83) node [anchor=north west][inner sep=0.75pt]  [color={rgb, 255:red, 144; green, 19; blue, 254 }  ,opacity=1 ]  {$v_{1}$};
\draw (20,102.4) node [anchor=north west][inner sep=0.75pt]  [color={rgb, 255:red, 144; green, 19; blue, 254 }  ,opacity=1 ]  {$v_{2}$};
\draw (20,40.4) node [anchor=north west][inner sep=0.75pt]  [color={rgb, 255:red, 144; green, 19; blue, 254 }  ,opacity=1 ]  {$v_{3}$};
\draw (135,41.4) node [anchor=north west][inner sep=0.75pt]  [color={rgb, 255:red, 144; green, 19; blue, 254 }  ,opacity=1 ]  {$v_{4}$};
\draw (161,142.4) node [anchor=north west][inner sep=0.75pt]    {$v'_{0}$};
\draw (34.57,151.83) node [anchor=north west][inner sep=0.75pt]    {$v'_{1}$};
\draw (4,112.4) node [anchor=north west][inner sep=0.75pt]    {$v'_{2}$};
\draw (20,3.4) node [anchor=north west][inner sep=0.75pt]    {$v'_{3}$};
\draw (161,3.4) node [anchor=north west][inner sep=0.75pt]    {$v'_{4}$};

\end{tikzpicture}
}
  \caption{$r\leq\min\{L_{e_i},L_{e_j}\}$}
         \label{fig: small_r}
     \end{subfigure}%
     \hfill 
 \begin{subfigure}[b]{0.3\textwidth}
         \centering    
\resizebox{!}{0.67\textwidth}{

\tikzset{every picture/.style={line width=0.75pt}} 

\begin{tikzpicture}[x=0.75pt,y=0.75pt,yscale=-1,xscale=1]

\draw    (55,34) -- (55,74) ;
\draw  [draw opacity=0][fill={rgb, 255:red, 155; green, 155; blue, 155 }  ,fill opacity=0.4 ] (154,74) -- (154,159) -- (123.14,159) -- (55,74) -- cycle ;
\draw [color={rgb, 255:red, 0; green, 0; blue, 0 }  ,draw opacity=1 ][line width=0.75]    (55,74) -- (123.14,159) ;
\draw [color={rgb, 255:red, 0; green, 0; blue, 0 }  ,draw opacity=1 ][line width=0.75]    (55,74) -- (154,74) ;
\draw [line width=0.75]    (123.14,159) -- (154,159) ;
\draw  [draw opacity=0][fill={rgb, 255:red, 155; green, 155; blue, 155 }  ,fill opacity=0.4 ] (181,33) -- (181,159) -- (123.14,158) -- (55,73) -- (55,33) -- cycle ;
\draw [color={rgb, 255:red, 80; green, 227; blue, 194 }  ,draw opacity=1 ][line width=0.75]    (55,34) -- (181,34) ;
\draw [shift={(181,34)}, rotate = 0] [color={rgb, 255:red, 80; green, 227; blue, 194 }  ,draw opacity=1 ][fill={rgb, 255:red, 80; green, 227; blue, 194 }  ,fill opacity=1 ][line width=0.75]      (0, 0) circle [x radius= 3.35, y radius= 3.35]   ;
\draw [shift={(55,34)}, rotate = 0] [color={rgb, 255:red, 80; green, 227; blue, 194 }  ,draw opacity=1 ][fill={rgb, 255:red, 80; green, 227; blue, 194 }  ,fill opacity=1 ][line width=0.75]      (0, 0) circle [x radius= 3.35, y radius= 3.35]   ;
\draw    (154,159) -- (181,159) ;
\draw [color={rgb, 255:red, 80; green, 227; blue, 194 }  ,draw opacity=1 ][line width=0.75]    (181,34) -- (181,160) ;
\draw [shift={(181,160)}, rotate = 90] [color={rgb, 255:red, 80; green, 227; blue, 194 }  ,draw opacity=1 ][fill={rgb, 255:red, 80; green, 227; blue, 194 }  ,fill opacity=1 ][line width=0.75]      (0, 0) circle [x radius= 3.35, y radius= 3.35]   ;
\draw [color={rgb, 255:red, 208; green, 2; blue, 27 }  ,draw opacity=1 ][line width=0.75]    (154,74) -- (154,159) ;
\draw [shift={(154,159)}, rotate = 90] [color={rgb, 255:red, 208; green, 2; blue, 27 }  ,draw opacity=1 ][fill={rgb, 255:red, 208; green, 2; blue, 27 }  ,fill opacity=1 ][line width=0.75]      (0, 0) circle [x radius= 3.35, y radius= 3.35]   ;
\draw [shift={(154,74)}, rotate = 90] [color={rgb, 255:red, 208; green, 2; blue, 27 }  ,draw opacity=1 ][fill={rgb, 255:red, 208; green, 2; blue, 27 }  ,fill opacity=1 ][line width=0.75]      (0, 0) circle [x radius= 3.35, y radius= 3.35]   ;

\draw (153,158.4) node [anchor=north west][inner sep=0.75pt]  [color={rgb, 255:red, 144; green, 19; blue, 254 }  ,opacity=1 ]  {$v_{0}$};
\draw (152,56.4) node [anchor=north west][inner sep=0.75pt]  [color={rgb, 255:red, 144; green, 19; blue, 254 }  ,opacity=1 ]  {$v_{4}$};
\draw (184,157.4) node [anchor=north west][inner sep=0.75pt]    {$v'_{0}$};
\draw (34,18.4) node [anchor=north west][inner sep=0.75pt]    {$v'_{3}$};
\draw (185,19.4) node [anchor=north west][inner sep=0.75pt]    {$v'_{4}$};

\end{tikzpicture}
}
  \caption{$\min\{L_{e_i},L_{e_j}\}<r\leq\max\{L_{e_i},L_{e_j}\}$ and $r\leq\min\{L_{e'_i},L_{e'_j}\}$}
         \label{fig: small_r}
     \end{subfigure}%
     \hfill 
 \begin{subfigure}[b]{0.3\textwidth}
         \centering    
\resizebox{!}{0.67\textwidth}{

\tikzset{every picture/.style={line width=0.75pt}} 

\begin{tikzpicture}[x=0.75pt,y=0.75pt,yscale=-1,xscale=1]

\draw    (154,74) -- (154,159) ;
\draw    (55,34) -- (55,74) ;
\draw  [draw opacity=0][fill={rgb, 255:red, 155; green, 155; blue, 155 }  ,fill opacity=0.4 ] (154,74) -- (154,159) -- (154,159) -- (55,74) -- cycle ;
\draw [color={rgb, 255:red, 0; green, 0; blue, 0 }  ,draw opacity=1 ][line width=0.75]    (55,74) -- (154,159) ;
\draw [color={rgb, 255:red, 0; green, 0; blue, 0 }  ,draw opacity=1 ][line width=0.75]    (55,74) -- (154,74) ;
\draw [line width=0.75]    (154,159) ;
\draw  [draw opacity=0][fill={rgb, 255:red, 155; green, 155; blue, 155 }  ,fill opacity=0.4 ] (181,33) -- (181,159) -- (154,159) -- (55,73) -- (55,33) -- cycle ;
\draw [color={rgb, 255:red, 80; green, 227; blue, 194 }  ,draw opacity=1 ][line width=0.75]    (55,34) -- (181,34) ;
\draw [shift={(181,34)}, rotate = 0] [color={rgb, 255:red, 80; green, 227; blue, 194 }  ,draw opacity=1 ][fill={rgb, 255:red, 80; green, 227; blue, 194 }  ,fill opacity=1 ][line width=0.75]      (0, 0) circle [x radius= 3.35, y radius= 3.35]   ;
\draw [shift={(55,34)}, rotate = 0] [color={rgb, 255:red, 80; green, 227; blue, 194 }  ,draw opacity=1 ][fill={rgb, 255:red, 80; green, 227; blue, 194 }  ,fill opacity=1 ][line width=0.75]      (0, 0) circle [x radius= 3.35, y radius= 3.35]   ;
\draw    (154,159) -- (181,159) ;
\draw [color={rgb, 255:red, 80; green, 227; blue, 194 }  ,draw opacity=1 ][line width=0.75]    (181,34) -- (181,160) ;
\draw [shift={(181,160)}, rotate = 90] [color={rgb, 255:red, 80; green, 227; blue, 194 }  ,draw opacity=1 ][fill={rgb, 255:red, 80; green, 227; blue, 194 }  ,fill opacity=1 ][line width=0.75]      (0, 0) circle [x radius= 3.35, y radius= 3.35]   ;
\draw [color={rgb, 255:red, 208; green, 2; blue, 27 }  ,draw opacity=1 ][line width=0.75]    (154,74) ;
\draw [shift={(154,74)}, rotate = 0] [color={rgb, 255:red, 208; green, 2; blue, 27 }  ,draw opacity=1 ][fill={rgb, 255:red, 208; green, 2; blue, 27 }  ,fill opacity=1 ][line width=0.75]      (0, 0) circle [x radius= 3.35, y radius= 3.35]   ;
\draw [shift={(154,74)}, rotate = 0] [color={rgb, 255:red, 208; green, 2; blue, 27 }  ,draw opacity=1 ][fill={rgb, 255:red, 208; green, 2; blue, 27 }  ,fill opacity=1 ][line width=0.75]      (0, 0) circle [x radius= 3.35, y radius= 3.35]   ;

\draw (156,56.4) node [anchor=north west][inner sep=0.75pt]  [color={rgb, 255:red, 144; green, 19; blue, 254 }  ,opacity=1 ]  {$v_{4}$};
\draw (186,159.4) node [anchor=north west][inner sep=0.75pt]    {$v'_{0}$};
\draw (34,18.4) node [anchor=north west][inner sep=0.75pt]    {$v'_{3}$};
\draw (185,19.4) node [anchor=north west][inner sep=0.75pt]    {$v'_{4}$};

\end{tikzpicture}
}
  \caption{$\max\{L_{e_i},L_{e_j}\}<r\leq L_{e_i}+L_{e_j}$ and $r\leq\min\{L_{e'_i},L_{e'_j}\}$}
         \label{fig: small_r}
     \end{subfigure}
     \hfill

 \begin{subfigure}[b]{0.3\textwidth}
         \centering    
\resizebox{!}{0.67\textwidth}{

\tikzset{every picture/.style={line width=0.75pt}} 

\begin{tikzpicture}[x=0.75pt,y=0.75pt,yscale=-1,xscale=1]

\draw  [draw opacity=0][fill={rgb, 255:red, 155; green, 155; blue, 155 }  ,fill opacity=0.4 ] (154,74) -- (154,159) -- (123.14,159) -- (76,74) -- cycle ;
\draw [color={rgb, 255:red, 0; green, 0; blue, 0 }  ,draw opacity=1 ][line width=0.75]    (76,74) -- (155,74) ;
\draw [line width=0.75]    (123.14,159) -- (154,159) ;
\draw  [draw opacity=0][fill={rgb, 255:red, 155; green, 155; blue, 155 }  ,fill opacity=0.4 ] (181,34) -- (181,160) -- (123.14,159) -- (55,35) -- (55,34) -- cycle ;
\draw [color={rgb, 255:red, 0; green, 0; blue, 0 }  ,draw opacity=1 ][line width=0.75]    (55,34) -- (181,34) ;
\draw    (55,34) -- (123.14,159) ;
\draw    (154,159) -- (181,159) ;
\draw [color={rgb, 255:red, 80; green, 227; blue, 194 }  ,draw opacity=1 ][line width=0.75]    (181,34) -- (181,160) ;
\draw [shift={(181,160)}, rotate = 90] [color={rgb, 255:red, 80; green, 227; blue, 194 }  ,draw opacity=1 ][fill={rgb, 255:red, 80; green, 227; blue, 194 }  ,fill opacity=1 ][line width=0.75]      (0, 0) circle [x radius= 3.35, y radius= 3.35]   ;
\draw [shift={(181,34)}, rotate = 90] [color={rgb, 255:red, 80; green, 227; blue, 194 }  ,draw opacity=1 ][fill={rgb, 255:red, 80; green, 227; blue, 194 }  ,fill opacity=1 ][line width=0.75]      (0, 0) circle [x radius= 3.35, y radius= 3.35]   ;
\draw [color={rgb, 255:red, 208; green, 2; blue, 27 }  ,draw opacity=1 ][line width=0.75]    (154,74) -- (154,159) ;
\draw [shift={(154,159)}, rotate = 90] [color={rgb, 255:red, 208; green, 2; blue, 27 }  ,draw opacity=1 ][fill={rgb, 255:red, 208; green, 2; blue, 27 }  ,fill opacity=1 ][line width=0.75]      (0, 0) circle [x radius= 3.35, y radius= 3.35]   ;
\draw [shift={(154,74)}, rotate = 90] [color={rgb, 255:red, 208; green, 2; blue, 27 }  ,draw opacity=1 ][fill={rgb, 255:red, 208; green, 2; blue, 27 }  ,fill opacity=1 ][line width=0.75]      (0, 0) circle [x radius= 3.35, y radius= 3.35]   ;

\draw (158,158.4) node [anchor=north west][inner sep=0.75pt]  [color={rgb, 255:red, 144; green, 19; blue, 254 }  ,opacity=1 ]  {$v_{0}$};
\draw (158,60.4) node [anchor=north west][inner sep=0.75pt]  [color={rgb, 255:red, 144; green, 19; blue, 254 }  ,opacity=1 ]  {$v_{4}$};
\draw (186,157.4) node [anchor=north west][inner sep=0.75pt]    {$v'_{0}$};
\draw (185,19.4) node [anchor=north west][inner sep=0.75pt]    {$v'_{4}$};

\end{tikzpicture}
}
  \caption{$\min\{L_{e_i},L_{e_j}\}<r\leq\max\{L_{e_i},L_{e_j}\}$ and $\min\{L_{e'_i},L_{e'_j}\}<r\leq\max\{L_{e'_i},L_{e'_j}\}$}
         \label{fig: small_r}
     \end{subfigure}
     \hfill 
 \begin{subfigure}[b]{0.3\textwidth}
         \centering    
\resizebox{!}{0.67\textwidth}{

\tikzset{every picture/.style={line width=0.75pt}} 

\begin{tikzpicture}[x=0.75pt,y=0.75pt,yscale=-1,xscale=1]

\draw    (156,74) -- (156,159) ;
\draw  [draw opacity=0][fill={rgb, 255:red, 155; green, 155; blue, 155 }  ,fill opacity=0.4 ] (155,74) -- (155,159) -- (155,159) -- (87,74) -- cycle ;
\draw [color={rgb, 255:red, 0; green, 0; blue, 0 }  ,draw opacity=1 ][line width=0.75]    (87,74) -- (156,74) ;
\draw [color={rgb, 255:red, 208; green, 2; blue, 27 }  ,draw opacity=1 ][line width=0.75]    (156,74) ;
\draw [shift={(156,74)}, rotate = 0] [color={rgb, 255:red, 208; green, 2; blue, 27 }  ,draw opacity=1 ][fill={rgb, 255:red, 208; green, 2; blue, 27 }  ,fill opacity=1 ][line width=0.75]      (0, 0) circle [x radius= 3.35, y radius= 3.35]   ;
\draw  [draw opacity=0][fill={rgb, 255:red, 155; green, 155; blue, 155 }  ,fill opacity=0.4 ] (182,33) -- (182,159) -- (155,158) -- (56,34) -- (56,33) -- cycle ;
\draw [color={rgb, 255:red, 0; green, 0; blue, 0 }  ,draw opacity=1 ][line width=0.75]    (56,34) -- (182,34) ;
\draw    (56,34) -- (155,159) ;
\draw    (155,159) -- (182,159) ;
\draw [color={rgb, 255:red, 80; green, 227; blue, 194 }  ,draw opacity=1 ][line width=0.75]    (182,34) -- (182,160) ;
\draw [shift={(182,160)}, rotate = 90] [color={rgb, 255:red, 80; green, 227; blue, 194 }  ,draw opacity=1 ][fill={rgb, 255:red, 80; green, 227; blue, 194 }  ,fill opacity=1 ][line width=0.75]      (0, 0) circle [x radius= 3.35, y radius= 3.35]   ;
\draw [shift={(182,34)}, rotate = 90] [color={rgb, 255:red, 80; green, 227; blue, 194 }  ,draw opacity=1 ][fill={rgb, 255:red, 80; green, 227; blue, 194 }  ,fill opacity=1 ][line width=0.75]      (0, 0) circle [x radius= 3.35, y radius= 3.35]   ;

\draw (161,62.4) node [anchor=north west][inner sep=0.75pt]  [color={rgb, 255:red, 144; green, 19; blue, 254 }  ,opacity=1 ]  {$v_{4}$};
\draw (185,155.4) node [anchor=north west][inner sep=0.75pt]    {$v'_{0}$};
\draw (184,19.4) node [anchor=north west][inner sep=0.75pt]    {$v'_{4}$};

\end{tikzpicture}
}
  \caption{$\max\{L_{e_i},L_{e_j}\}<r\leq L_{e_i}+L_{e_j}$ and $\min\{L_{e'_i},L_{e'_j}\}<r\leq\max\{L_{e'_i},L_{e'_j}\}$}
         \label{fig: small_r}
     \end{subfigure}%
     \hfill 
 \begin{subfigure}[b]{0.3\textwidth}
         \centering    
\resizebox{!}{0.67\textwidth}{

\tikzset{every picture/.style={line width=0.75pt}} 

\begin{tikzpicture}[x=0.75pt,y=0.75pt,yscale=-1,xscale=1]

\draw  [draw opacity=0][fill={rgb, 255:red, 155; green, 155; blue, 155 }  ,fill opacity=0.4 ] (155,74) -- (155,159) -- (155,125) -- (99,74) -- cycle ;
\draw [color={rgb, 255:red, 0; green, 0; blue, 0 }  ,draw opacity=1 ][line width=0.75]    (99,74) -- (156,74) ;
\draw [color={rgb, 255:red, 0; green, 0; blue, 0 }  ,draw opacity=1 ][line width=0.75]    (155,74) -- (155,125) ;
\draw  [draw opacity=0][fill={rgb, 255:red, 155; green, 155; blue, 155 }  ,fill opacity=0.4 ] (182,34) -- (181,148) -- (181,148) -- (56,35) -- (56,34) -- cycle ;
\draw [color={rgb, 255:red, 0; green, 0; blue, 0 }  ,draw opacity=1 ][line width=0.75]    (56,34) -- (182,34) ;
\draw [color={rgb, 255:red, 0; green, 0; blue, 0 }  ,draw opacity=1 ][line width=0.75]    (182,34) -- (182,148) ;
\draw    (56,34) -- (181,148) ;
\draw [color={rgb, 255:red, 80; green, 227; blue, 194 }  ,draw opacity=1 ][line width=0.75]    (181,34) ;
\draw [shift={(181,34)}, rotate = 0] [color={rgb, 255:red, 80; green, 227; blue, 194 }  ,draw opacity=1 ][fill={rgb, 255:red, 80; green, 227; blue, 194 }  ,fill opacity=1 ][line width=0.75]      (0, 0) circle [x radius= 3.35, y radius= 3.35]   ;
\draw [color={rgb, 255:red, 208; green, 2; blue, 27 }  ,draw opacity=1 ][line width=0.75]    (155,74) ;
\draw [shift={(155,74)}, rotate = 0] [color={rgb, 255:red, 208; green, 2; blue, 27 }  ,draw opacity=1 ][fill={rgb, 255:red, 208; green, 2; blue, 27 }  ,fill opacity=1 ][line width=0.75]      (0, 0) circle [x radius= 3.35, y radius= 3.35]   ;

\draw (159,67.4) node [anchor=north west][inner sep=0.75pt]  [color={rgb, 255:red, 144; green, 19; blue, 254 }  ,opacity=1 ]  {$v_{4}$};
\draw (185,19.4) node [anchor=north west][inner sep=0.75pt]    {$v'_{4}$};

\end{tikzpicture}
}
  \caption{$\max\{L_{e'_i},L_{e'_j}\}<r\leq L_{e_i}+L_{e_j}$}
         \label{fig: small_r}
     \end{subfigure}%
     \hfill 
\caption{Possible intersections of $\lVert(G_k)_{r,\mathbf{L}}\rVert$ and $\Ima\lVert (G_k)_{r, \mathbf{L}\leq \mathbf{L}'}\rVert$ with $\sigma$ }
        \label{fig:six_cases}
\end{figure}

When $r\leq \min\{L_{e_i}, L_{e_j}\}$ (where $i\neq j$), $\Ima\lVert (G_k)_{r, \mathbf{L}\leq \mathbf{L}'}\rVert = \lVert (G_k)_{r,\mathbf{L}'} \rVert$. Let $\Lambda v_0v_4v_3$ ( resp. $\Lambda v_0'v_4'v_3'$) denote the intersection of $\lVert (G_k)_{r,\mathbf{L}} \rVert$ (resp. $\Ima\lVert (G_k)_{r, \mathbf{L}\leq \mathbf{L}'}\rVert$) with $\sigma$. 



Define 
        $$G_1^{\sigma}(s,t)=(1-s)((1-t)v_0+tv_0')+s((1-t)v_4+tv_4')$$
        $$G_2^{\sigma}(s,t)=(1-s)((1-t)v_3+tv_3')+s((1-t)v_4+tv_4')$$
where $0\leq s\leq 1$ and $0\leq t\leq 1$. Note that $G_1^{\sigma}(1,t)=G_2^{\sigma}(1,t)$ for all $t\in [0,1]$, we obtain a well-defined continuous function 
        $$G^{\sigma}(x,t)=\begin{cases} G_1^{\sigma}(s,t), & \mbox{if } x= (1-s) v_0+ sv_4\in v_0v_4 \\ G_2^{\sigma}(s,t), & \mbox{if } x= (1-s) v_3+ sv_4\in v_3v_4  \end{cases}$$
        
When $t$ is fixed, note that the inverse of $G^{\sigma}(x,t)$ is $$(G^{\sigma})^{-1}(y,t)=\begin{cases} (G_1^{\sigma})^{-1}(s,t), & \mbox{if } y= (1-s)((1-t)v_0+tv_0')+s((1-t)v_4+tv_4') \\ (G_2^{\sigma})^{-1}(s,t), & \mbox{if } y= (1-s)((1-t)v_3+tv_3')+s((1-t)v_4+tv_4')  \end{cases}$$

where $(G_1^{\sigma})^{-1}(s,t)=(1-s)v_0 + s v_4$ and $(G_2^{\sigma})^{-1}(s,t)=(1-s)v_3 + s v_4$. Hence when $t$ is fixed, $G^{\sigma}(x,t)$ is a homeomorphism onto its image.
\end{proof}

Lemma~\ref{lemma:2.4} implies that the diagram~\ref{eq:model:8} is commutative up to homotopy. Hence the following diagram is commutative for all $\mathbf{L}\leq \mathbf{L}'$ and $r\in \mathbb{R}_{>0}$:

\begin{equation}\label{eq:model:9}
\begin{tikzcd}
H_i(\lVert(G_k)_{r,\mathbf{L}}\rVert;\mathbb{F}) \arrow{r}{} \arrow[swap]{d}{\cong} & H_i(\lVert(G_k)_{r,\mathbf{L}'}\rVert;\mathbb{F}) \arrow{d}{\cong} \\%
H_i((\mathsf{Star}_k)^2_{r,\mathbf{L}};\mathbb{F}) \arrow{r}{}& H_i((\mathsf{Star}_k)^2_{r,\mathbf{L}'};\mathbb{F})
\end{tikzcd}
\end{equation}

The discussion above proves the next lemma.
\begin{lemma}\label{thm:2.6}
   For all $r\in\mathbb{R}_{>0}$, $PH_i(\lVert(G_k)_{r,-}\rVert;\mathbb{F})\cong PH_i((\mathsf{Star}_k)^2_{r,-};\mathbb{F})$ for all $i\in\mathbb{N}$.
\end{lemma}

Now we prove the main theorem of this section.
\begin{theorem}\label{thm:2.8}
Let $k\geq 3$. 
$PH_i((\mathsf{Star}_k)^2_{-,-};\mathbb{F})\cong PH_i(\lVert(G_k)_{-,-}\rVert;\mathbb{F})$ for all $i\in\mathbb{N}$. 
\end{theorem}
\begin{proof}
It suffices to show the following cube diagram commutes:

\[\begin{tikzcd}[row sep={40,between origins}, column sep={55,between origins}]
      & H_i(\lVert (G_k)_{r',\mathbf{L}}\rVert;\mathbb{F})
      \ar{rr}\ar{dd}\ar{dl} & & H_i(\lVert(G_k)_{r',\mathbf{L}'}\rVert;\mathbb{F})  \ar{dd}\ar{dl} \\
    H_i(\lVert(G_k)_{r,\mathbf{L}}\rVert;\mathbb{F}) \ar[crossing over]{rr} \ar{dd} & & H_i(\lVert (G_k)_{r,\mathbf{L}'}\rVert;\mathbb{F}) \\
      & H_i((\mathsf{Star}_k)^2_{r',\mathbf{L}};\mathbb{F})  \ar{rr} \ar{dl} & &  H_i((\mathsf{Star}_k)^2_{r',\mathbf{L}'};\mathbb{F})  \ar{dl} \\
      H_i((\mathsf{Star}_k)^2_{r,\mathbf{L}};\mathbb{F})\ar{rr} && H_i((\mathsf{Star}_k)^2_{r,\mathbf{L}'};\mathbb{F}) \ar[from=uu,crossing over]
\end{tikzcd}\]

Since $PH_i((\mathsf{Star}_k)^2_{-,-};\mathbb{F}):(\mathbb{R}_{>0},\leq)\op\times (\mathbb{R}_{>0},\leq)^k\rightarrow \cat{vect}_{\mathbb{F}}$ is a persistence module, the bottom face of the cube is commutative. Moreover, Theorem~\ref{thm:2.7} implies that the top face of the cube is commutative. In addition, Lemma~\ref{thm:2.6} implies that the front and the back of the cube are commutative, while Theorem~\ref{thm:2.1} implies that the side faces of the cube are commutative. Notice that, by Theorem~\ref{thm:2.1}, all vertical arrows in the diagram are isomorphisms. Therefore, $PH_i((\mathsf{Star}_k)^2_{-,-};\mathbb{F})\cong PH_i(\lVert(G_k)_{-,-}\rVert;\mathbb{F})$. 
\end{proof}

\begin{corollary}\label{Coro:2:11}
For any $k\geq 4$, the induced morphism $$PH_1((\mathsf{Star}_k)^2_{r'\leq r,\mathbf{L}\leq \mathbf{L}'};\mathbb{F}): PH_1((\mathsf{Star}_k)^2_{r,\mathbf{L}};\mathbb{F})\rightarrow PH_1((\mathsf{Star}_k)^2_{r', \mathbf{L}'};\mathbb{F})$$ is injective for all $r'\leq r\in (\mathbb{R}_{>0}, \leq)$ and  $\mathbf{L}\leq \mathbf{L}'\in(\mathbb{R}_{>0})^k$.
\end{corollary}
\begin{proof}
Since $(G_k)_{r,\mathbf{L}}\rightarrow (G_k)_{r',\mathbf{L}'}$ is an inclusion map and $\lVert (G_k)_{r,\mathbf{L}}\rVert$ and $\lVert (G_k)_{r',\mathbf{L}'}\rVert$ do not have a 2-cell for all $r,r'>0$. Therefore, the induced morphism $PH_1(\lVert(G_k)_{r'\leq r,\mathbf{L}\leq \mathbf{L}'}\rVert;\mathbb{F})$ is injective. Theorem~\ref{thm:2.8} implies that the induced morphism $PH_1((\mathsf{Star}_k)^2_{r'\leq r,\mathbf{L}\leq \mathbf{L}'};\mathbb{F})$ is also injective. 
\end{proof}

\section{The Indecomposable Direct Summands of $PH_1((\mathsf{Star}_k)^2_{-,-};\mathbb{F})$}\label{sec:indecomp}

In this section, we work particularly on the case when edge-length vector $\mathbf{L}=(L_1,L_2,\dots, L_k)\in\mathbb{R}^k$, where $L_2,\dots,L_k$ are arbitrary but fixed positive real numbers, and only the first coordinate $L$ varies. Under this convention, we compute and investigate the structure of the $2$-parameter persistence module $PH_1((\mathsf{Star}_k)^2_{-,-};\mathbb{F})$ using the persistent model $(G_k)_{-,-}$, which was introduced in the previous section.



\subsection{Basic Properties of the Hyperplane Arrangement of $(\mathsf{Star}_k)^2_{r,\mathbf{L}}$}


Let $(\mathsf{Star}_k)^2_{\mathbf{L}}$ be the metric star graph of $k$ edges with the edge-length vector $\mathbf{L}=(L_1,L_2,\dots, L_k)$ where $L_2,\dots, L_k$ are arbitrary but fixed positive real numbers. Without loss of generality, we assume that $L_k\leq L_{k-1}\leq \cdots\leq L_3 \leq L_2$. Construct a bifiltration with respect to the parameters $r$ and $L$:

\begin{equation*}
    \begin{aligned}
        (\mathsf{Star}_k)^2_{-,-}: (\mathbb{R}_{>0},\leq)\op\times (\mathbb{R}_{>0},\leq)&\rightarrow \cat{Top}\\
        (r,L_1)&\mapsto (\mathsf{Star}_k)^2_{r,\mathbf{L}}
    \end{aligned}
\end{equation*}

Note that $(\mathsf{Star}_k)^2_{r,\mathbf{L}}$ has the following critical hyperplanes in the parameter space:
    \begin{equation}\label{eq:hyp:1}
        \begin{cases} 
            r=L_i, \quad\mbox{for all } i=1, 2,\dots,k.\\
            r=L_1+L_i, \quad \mbox{for all } i=2,\dots,k.
        \end{cases} 
    \end{equation}
    and
    \begin{equation}\label{eq:hyp:2}
        \begin{aligned}
            r&=L_i+L_j, \quad \mbox{for all } i,j=2,\dots,k \mbox{ and } i\neq j.
        \end{aligned}
    \end{equation}
    
These hyperplanes correspond to loci in parameter space where changes in the combinatorial type of \( (\mathsf{Star}_k)^2_{r, \mathbf{L}} \) may occur.

\begin{theorem}\label{thm:4:3}
Let $k\geq 4$ and $\mathbf{L}=(L_1,L_2,\dots, L_k)\in\mathbb{R}_{>0}^k$ where $L_k\leq L_{k-1}\leq \cdots\leq L_3 \leq L_2$ are arbitrary but fixed positive real numbers. Then
\begin{enumerate}
    \item when $r\leq L_1 \mbox{ and } r\leq L_k$, $\rank(H_1((\mathsf{Star}_{k})^2_{r,\mathbf{L}}))=k^2-3k+1$;
    \item when $r\leq L_1 \mbox{ and } L_k<r\leq L_{k-1}$, $\rank(H_1((\mathsf{Star}_{k})^2_{r,\mathbf{L}}))=k^2-5k+5$;
    \item when $r>L_1 \mbox{ and } r\leq L_k$, $\rank(H_1((\mathsf{Star}_{k})^2_{r,\mathbf{L}}))=k^2-5k+5$;
    \item when $r>L_2$, $\rank(H_1((\mathsf{Star}_{k})^2_{r,\mathbf{L}}))=0$.
\end{enumerate}
\end{theorem}

\begin{proof}
By Theorem~\ref{thm:2.8}, it suffices to compute the rank of $H_1((G_k)_{r,\mathbf{L}})$. When $r\leq L_1$ and $r\leq L_k$, $(G_k)_{r,\mathbf{L}}$ has $k^2+k$ vertices, $2k^2-2k$ edges, and $1$ path component. Hence $\rank (H_1(\lVert (G_k)_{r,\mathbf{L}}\rVert))=2k^2-2k-(k^2+k)+1=k^2-3k+1$. 

When $r>L_1$ and $r\leq L_k$, $(G_k)_{r,\mathbf{L}}$ has $k^2+k-2$ vertices, $2k^2-2k-2(k-1)$ edges, and $1$ path component. Hence $\rank (H_1(\lVert (G_k)_{r,\mathbf{L}}\rVert))=2k^2-2k-2(k-1)-(k^2+k-2)+1=k^2-5k+5$.

When $r\leq L_1$ and $L_k<r\leq L_{k-1}$, $(G_k)_{r,\mathbf{L}}$ has $k^2+k-2$ vertices, $2(k-1)^2$ edges, and $1$ path component. Hence $\rank (H_1(\lVert (G_k)_{r,\mathbf{L}}\rVert))=2(k-1)^2-(k^2+k-2)+1=k^2-5k+5$. 

Theorem~\ref{thm:2.1} implies that $\dim H_1((\mathsf{Star}_k)^2_{r,\mathbf{L}};\mathbb{F})=0$ for all $r>L_2$ and $\mathbf{L}=(L_1,L_2, \dots, L_k)$.
\end{proof}

We define the \key{reduced hyperplane arrangement} of 
$(\mathsf{Star}_k)^2_{r,\mathbf{L}}$ to be the collection of hyperplanes in $\mathbb{R}^2$ given by the system of equations~\eqref{eq:hyp:1}. The next theorem shows that these are precisely the hyperplanes along which the rank of $H_1((\mathsf{Star}_k)^2_{r,\mathbf{L}})$ may change.\\

\paragraph{\textbf{Notation:}}
Since in this section we will mainly focus on the case where 
$\mathbf{L} = (L_1, L_2, \dots, L_k)$ with $L_2, \dots, L_k$ fixed, we introduce the notation
\[
\mathbf{L} = (L, L_2, \dots, L_k)
\quad\text{with } L := L_1,
\]
to emphasize that the first coordinate is a parameter while the remaining edge lengths are constant.

\begin{theorem}\label{thm:4.4}
Let $k\geq 4$ and $\mathbf{L}=(L,L_2,\dots, L_k)\in\mathbb{R}_{>0}^k$ where $L_2,\dots, L_k$ are arbitrary but fixed positive real numbers. Let $C$ be a chamber given by the reduced hyperplane arrangement of $(\mathsf{Star}_k)^2_{r,\mathbf{L}}$. For all $(r^\ast, L^\ast) \in C \cap \{(r, L) \in \mathbb{R}^2: r = L_i + L_j\}$ such that $(r^\ast, L^\ast)$ is not contained in any hyperplane of the reduced hyperplane arrangement of $(\mathsf{Star}_k)^2_{r,\mathbf{L}}$, there exists \( \varepsilon > 0 \) such that for all \( (r, L) \in B_\varepsilon(r^\ast, L^\ast) \subset C \), 
$$\rank H_1((\mathsf{Star}_k)^2_{r,\mathbf{L}})=\rank H_1((\mathsf{Star}_k)^2_{r^\ast,\mathbf{L}^\ast})$$
where $\mathbf{L}^\ast=(L^\ast,L_2,\dots,L_k)$ and $B_\varepsilon(r^\ast, L^\ast)$ is the (open) ball of radius $\varepsilon$  centered at $(r^\ast, L^\ast)$.
\end{theorem}

\begin{proof}
By Theorem~\ref{thm:2.8}, it suffices to show that there exists \( \varepsilon > 0 \) such that for all \( (r, L) \in B_\varepsilon(r^\ast, L^\ast) \subset C \), $\rank H_1\left( \left\| (G_k)_{r,L} \right\| \right) = \rank H_1\left( \left\| (G_k)_{r^\ast,L^\ast} \right\| \right)$. For all $(r^\ast, L^\ast)$, define $\varepsilon$ to be the minimal (Euclidean) distance between $(r^\ast, L^\ast)$ and hyperplanes in the reduced hyperplane arrangement. Since there are finitely many hyperplanes in the reduced hyperplane arrangement and $(r^\ast, L^\ast)$ is not contained in any such hyperplane, $\varepsilon>0$. Hence $B_\varepsilon(r^\ast, L^\ast) \subset C$. 

Note that the vertical hyperplane $r= L_i+L_j$ subdivides $B_\varepsilon(r^\ast, L^\ast)$ into two disjoint subsets: $A_{ij,\varepsilon}:=\{(r,L)\in B_\varepsilon(r^\ast, L^\ast): r\leq L_i+L_j\}$ and $B_\varepsilon(r^\ast, L^\ast)-A_{ij,\varepsilon}$. Assume for any $(r,L)\in A_{ij,\varepsilon}$, $(G_k)_{r,L}$ has $a$ vertices, $b$ edges, and $x$ path components. In particular, for each $i',j'$ such that $L_{i'}+L_{j'}=L_{i}+L_{j}$, vertices $i'j'$ and $j'i'$ are two path components of $(G_k)_{r,L}$ since $r>L_{i'}$ and $r>L_{j'}$. By the Euler-Poincaré formula, the rank of $H_1(\lVert (G_k)_{r,\mathbf{L}}\rVert)$ is $b-a+x$. On the other hand, for any $(r',L')\in B_\varepsilon(r^\ast, L^\ast)-A_{ij,\varepsilon}$, $(G_k)_{r',L'}$ has $a-2m$ vertices, $b$ edges, and $x-2m$ path components (where $m$ is the multiplicity of the vertical hyperplane $r=L_i+L_j$ in the (multi-)set of equations given by (\ref{eq:hyp:2})) since $r>L_i+L_j$. By the Euler-Poincaré formula,the rank of $H_1(\lVert (G_k)_{r',\mathbf{L}'}\rVert)$ is $b-(a-2m)+(x-2m)=b-a+x$.
\end{proof}

By Theorem~\ref{thm:4.4}, the rank of the first homology group of $(\mathsf{Star}_k)^2_{r,\mathbf{L}}$ remains constant within each chamber of the reduced hyperplane arrangement, where $\mathbf{L}=(L,L_2,\dots,L_k)$. Moreover, Corollary~\ref{Coro:2:11} ensures that the first homology groups $H_1((\mathsf{Star}_k)^2_{r,\mathbf{L}})$ and $H_1((\mathsf{Star}_k)^2_{r',\mathbf{L}'})$ are isomorphic when $(r,L)$ and $(r^\ast,L^\ast)$ (where $\mathbf{L}^\ast=(L^\ast,L_2, \dots,L_k)$) lying within the same chamber in the reduced hyperplane arrangement. Hence, for the purpose of determining the behavior of the first homology group of $(\mathsf{Star}_k)^2_{r,\mathbf{L}}$ as parameters $r$ and $L$ vary, it suffices to consider the reduced hyperplane arrangement rather than the complete one.

\begin{theorem}\label{cor:4.4}
    Let $k\geq 4$ and $\mathbf{L}=(L,L_2,\dots, L_k)\in\mathbb{R}_{>0}^k$ where $L_k\leq\cdots\leq L_2$ are arbitrary but fixed positive real numbers. When $L_{4}<r$ and $L<r$, $\rank (H_1(\mathsf{Star}_k)^2_{r,\mathbf{L}})=0$.
\end{theorem}
\begin{proof}
    We first consider the case when $L_{4}<r\leq L_{3}$ and $L<r\leq L+L_{k}$. By Theorem~\ref{thm:2.8}, it suffices to show that $\rank (H_1(\lVert (G_k)_{r,\mathbf{L}}\rVert))=0$ when $L_{4}<r\leq L_{3}$ and $L<r\leq L+L_{k}$. By Theorem~\ref{thm:4.4}, we may assume there are $h$ vertical hyperplanes of the form $r=L_i + L_j$ such that $L_i+L_j\leq L_{4}$. Note that $\lVert(G_k)_{r,\mathbf{L}}\rVert)$ has $k(k-1)+4-2h$ vertices, $4(k-1)$ edges, and $(k-2)(k-3)+2-2h$ path components. By the Euler-Poincaré formula, the rank of $H_1(\lVert (G_k)_{r,\mathbf{L}}\rVert)$ is $4(k-1)-(k(k-1)+4-2h)+(k-3)(k-2)+2-2h=0$. 

    By Corollary~\ref{Coro:2:11}, the induced morphism $PH_1((\mathsf{Star}_k)^2_{r'\leq r,\mathbf{L}\leq \mathbf{L}'};\mathbb{F})$ is injective for all $r'\leq r\in (\mathbb{R}_{>0}, \leq)$ and  $\mathbf{L}\leq \mathbf{L}'\in(\mathbb{R}_{>0})^k$. Therefore, $\rank (PH_1(\mathsf{Star}_k)^2_{r,\mathbf{L}}))=0$ when $L_{4}<r$ and $L<r$.
\end{proof}


\begin{theorem}\label{thm:4.5}
     Let $k\geq 4$ and $\mathbf{L}=(L,L_2,\dots, L_k)\in\mathbb{R}_{>0}^k$ where $L_k\leq\cdots\leq L_2$ are arbitrary but fixed positive real numbers. When $L_{4}<r\leq L_{3}$ and $r\leq L$, $\rank (H_1(\mathsf{Star}_k)^2_{r,\mathbf{L}}))=1$.
\end{theorem}
\begin{proof}
    By Theorem~\ref{thm:2.8}, it suffices to show that $\rank (H_1(\lVert (G_k)_{r,\mathbf{L}}\rVert))=1$ when $L_{4}<r\leq L_{3}$ and $r\leq L$. By Theorem~\ref{thm:4.4},  we may assume there are $h$ vertical hyperplanes of the form $r=L_i + L_j$ such that $L_i+L_j\leq L_{4}$. Note that $\lVert(G_k)_{r,\mathbf{L}}\rVert)$ has $k(k-1)+6-2h$ vertices, $6(k-1)$ edges, and $(k-3)(k-4)+1-2h$ path components.  By the Euler-Poincaré formula, the rank of $H_1(\lVert (G_k)_{r,\mathbf{L}}\rVert)$ is $6(k-1)-(k(k-1)+6-2h)+(k-3)(k-4)+1-2h=1$.
\end{proof}

\begin{theorem}\label{thm:4.6-1}
Let $2\leq z\preceqdot y\preceqdot x\leq k$ and $L_k\leq\cdots\leq L_2$ be arbitrary but fixed positive real numbers. Define $A_1=\{(r,L):L_x< r\leq L_y, L<r\leq L+L_k\}$ and $A_2=\{(r,L):L_y< r\leq L_z$, and $L<r\}$. Then for all $(r_0, L)\in A_1$ and $(r_0^\ast, L^\ast)\in A_2$,
$$\rank(H_1(\mathsf{Star}_k)^2_{r_0,\mathbf{L}})=\rank(H_1(\mathsf{Star}_k)^2_{r_0^\ast,\mathbf{L}^\ast})$$
where $\mathbf{L}=(L,L_2,\dots, L_k)$ and $\mathbf{L}^\ast=(L^\ast,L_2,\dots, L_k)$.
\end{theorem}
\begin{proof}
By Theorem~\ref{thm:2.8}, it suffices to show that $\rank (H_1(\lVert (G_k)_{r_0,\mathbf{L}}\rVert))=\rank (H_1(\lVert (G_k)_{r_0^\ast,\mathbf{L}^\ast}\rVert))$. By Theorem~\ref{thm:4.4},  we may assume there are $h$ vertical hyperplanes of the form $r=L_i + L_j$ such that $L_i+L_j< L_{y}$ and $m$ vertical hyperplanes of the form $r=L_i + L_j$ such that $L_i+L_j=L_{y}$. 

Assume $(G_k)_{r_0,\mathbf{L}}$ has $a$ vertices, $b$ edges, and $x$ path components. By the Euler-Poincaré formula, the rank of $H_1(\lVert (G_k)_{r_0,\mathbf{L}}\rVert)$ is $b-a+x$. Note that $(G_k)_{r_0^\ast,\mathbf{L}^\ast}$ has $a-2+2-2m$ vertices, $b-2(k-h)+2(k-h-1)$ edges, and $x-2m+2$ path components.  By the Euler-Poincaré formula, $$\rank (H_1(\lVert (G_k)_{r_0^\ast,\mathbf{L}^\ast}\rVert))=b-2(k-h)+2(k-h-1)-(a-2m)+(x-2m+2)=b-a+x$$

Therefore, for all $(r_0, L)\in A_1$ and $(r_0^\ast, L^\ast)\in A_2$,
$\rank(H_1(\mathsf{Star}_k)^2_{r_0,\mathbf{L}})=\rank(H_1(\mathsf{Star}_k)^2_{r_0^\ast,\mathbf{L}^\ast})$.
\end{proof}

\begin{theorem}\label{thm:4.7}
Let $2\leq y\preceqdot x\leq k$ and $L_k\leq\cdots\leq L_2$ be arbitrary but fixed positive real numbers. Define $R=\{(r,L): L_x< r\leq L_y, L<r\}$. Then for all $(r_0,L_1),(r_0^\ast,L_1^\ast)\in R$,

$$\rank(H_1(\mathsf{Star}_k)^2_{r_0,\mathbf{L}})=\rank(H_1(\mathsf{Star}_k)^2_{r_0^\ast,\mathbf{L}^\ast})$$
where $\mathbf{L}=(L_1,L_2,\dots, L_k)$ and $\mathbf{L}^\ast=(L_1^\ast,L_2,\dots, L_k)$.
\end{theorem}
\begin{proof}
    Since $(\mathsf{Star}_k)^2_{r_0,\mathbf{L}}$ and $(\mathsf{Star}_k)^2_{r_0^\ast,\mathbf{L}^\ast}$ have the same homotopy type when $(r_0,L_1)$ and $(r_0^\ast,L_1^\ast)$ lie in the same chamber in the hyperplane arrangement, the statement holds in this case.

    Now we assume $(r_0,L_1)$ and $(r_0^\ast,L_1^\ast)$ lie in different chambers. Since the homotopy type of $(\mathsf{Star}_k)^2_{r,\mathbf{L}}$ remains unchanged within each chamber, and since every chamber is connected and convex, we may assume that $r_0=r_0^\ast$. In addition, without loss of generality, we assume that there is exactly one hyperplane $r=L+L_i$ such that $L_1+L_i\geq r$ and $L_1^\ast+L_i< r$, and let $m$ denote the multiplicity of the hyperplane. 
    
    By Theorem~\ref{thm:2.8}, it suffices to show that $\rank (H_1(\lVert (G_k)_{r_0,\mathbf{L}}\rVert))=\rank (H_1(\lVert (G_k)_{r_0,\mathbf{L}^\ast}\rVert))$ for all $(r_0,L_1), (r_0,L_1^\ast)\in R$. Assume $(G_k)_{r_0,\mathbf{L}}$ has $a$ vertices, $b$ edges, and $x$ path components. By the Euler-Poincaré formula, the rank of $H_1(\lVert (G_k)_{r_0,\mathbf{L}}\rVert)$ is $b-a+x$. Note that $(G_k)_{r_0,\mathbf{L}^\ast}$ has  $a-2m$ vertices, $b$ edges, and $x-2m$ path components. By the Euler-Poincaré formula, the rank of $H_1(\lVert (G_k)_{r_0,\mathbf{L}^\ast}\rVert)$ is $b-(a-2m)+(x-2m)$. Therefore, $\rank(H_1(\mathsf{Star}_k)^2_{r_0,\mathbf{L}})=\rank(H_1(\mathsf{Star}_k)^2_{r_0,\mathbf{L}^\ast})$.
\end{proof}

Theorem~\ref{thm:4.7} asserts that the rank of $H_1(\mathsf{Star}_k)^2_{r,\mathbf{L}}$ is constant in the region $R=\{(r,L): L_x< r\leq L_y, L<r\}$. Moreover, by Corollary~\ref{Coro:2:11}, $PH_1((\mathsf{Star}_k)^2_{r^\ast\leq r,\mathbf{L}\leq \mathbf{L}^\ast}$ is injective for all $r^\ast\leq r$ and $\mathbf{L}\leq \mathbf{L}^\ast$,  where $\mathbf{L}=(L,L_2,\dots,L_k)$ and $\mathbf{L}^\ast=(L^\ast,L_2,\dots,L_k)$. Therefore, these injective maps are necessarily isomorphisms for all $(r,L)\leq(r^\ast,L^\ast)\in R$.
Consequently, no first-degree homological 
change occurs near the hyperplane $r = L + L_i$ for $i=1,\dots,k$, and thus these hyperplanes may be removed from the reduced hyperplane arrangement.
By a slight abuse of terminology, we continue to refer to this new hyperplane arrangement as the reduced hyperplane arrangement.

Figure~\ref{fig:hpl-arrange-0'} summarizes the result of Theorem~\ref{thm:4:3}, Theorem~\ref{cor:4.4}, Theorem~\ref{thm:4.5}, Theorem~\label{thm:4.6}, and Theorem~\ref{thm:4.7}.

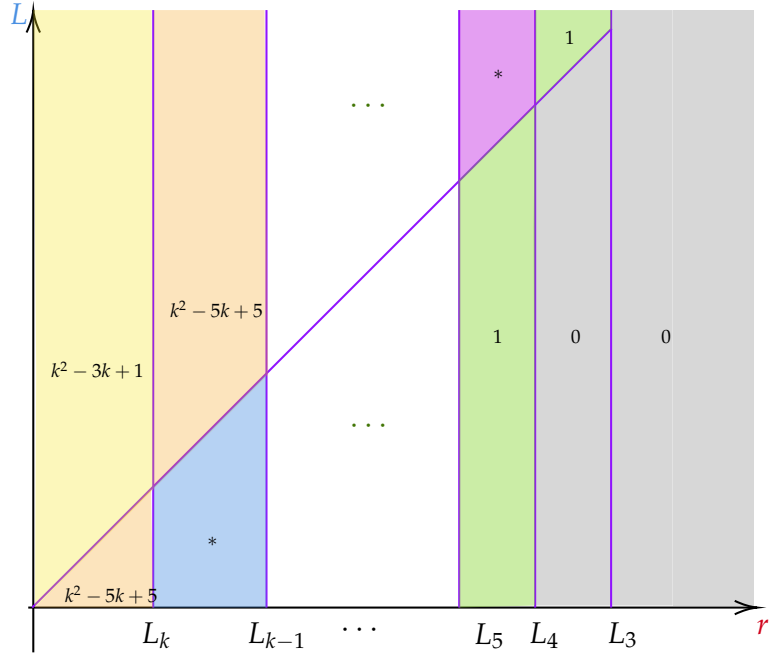
\begin{figure}[htbp!]
\centering
 \resizebox{!}{0.6\textwidth}{

\tikzset{every picture/.style={line width=0.75pt}} 

\begin{tikzpicture}[x=0.75pt,y=0.75pt,yscale=-1,xscale=1]

\draw  [draw opacity=0][fill={rgb, 255:red, 155; green, 155; blue, 155 }  ,fill opacity=0.4 ] (328.34,19) -- (328.34,309.31) -- (261.34,309.31) -- (260.5,124.5) -- (261.5,65) -- (298.5,28) -- (298.34,19) -- cycle ;
\draw  [draw opacity=0][fill={rgb, 255:red, 155; green, 155; blue, 155 }  ,fill opacity=0.4 ] (367.65,19) -- (328.34,19) -- (328.34,309.31) -- (367.65,309.31) -- cycle ;
\draw    (14.39,310.31) -- (367.5,310.31) ;
\draw [shift={(369.5,310.31)}, rotate = 180] [color={rgb, 255:red, 0; green, 0; blue, 0 }  ][line width=0.75]    (10.93,-3.29) .. controls (6.95,-1.4) and (3.31,-0.3) .. (0,0) .. controls (3.31,0.3) and (6.95,1.4) .. (10.93,3.29)   ;
\draw    (16.85,332.15) -- (16.85,21) ;
\draw [shift={(16.85,19)}, rotate = 90] [color={rgb, 255:red, 0; green, 0; blue, 0 }  ][line width=0.75]    (10.93,-3.29) .. controls (6.95,-1.4) and (3.31,-0.3) .. (0,0) .. controls (3.31,0.3) and (6.95,1.4) .. (10.93,3.29)   ;
\draw [color={rgb, 255:red, 144; green, 19; blue, 254 }  ,draw opacity=1 ][line width=0.75]    (16.3,310.31) -- (298.61,28) ;
\draw [color={rgb, 255:red, 144; green, 19; blue, 254 }  ,draw opacity=1 ][line width=0.75]    (75.28,310.31) -- (75.28,19) ;
\draw [color={rgb, 255:red, 144; green, 19; blue, 254 }  ,draw opacity=1 ][line width=0.75]    (130.45,310.31) -- (130.45,19) ;
\draw [color={rgb, 255:red, 144; green, 19; blue, 254 }  ,draw opacity=1 ][line width=0.75]    (261.34,310.31) -- (261.34,19) ;
\draw  [draw opacity=0][fill={rgb, 255:red, 248; green, 231; blue, 28 }  ,fill opacity=0.3 ] (74.5,19) -- (75.56,252.42) -- (17.67,310.31) -- (18.25,194.42) -- (16.85,19) -- cycle ;
\draw  [draw opacity=0][fill={rgb, 255:red, 245; green, 166; blue, 35 }  ,fill opacity=0.3 ] (129.5,19) -- (131.27,196.71) -- (75.56,252.42) -- (75.56,19) -- cycle ;
\draw  [draw opacity=0][fill={rgb, 255:red, 245; green, 166; blue, 35 }  ,fill opacity=0.3 ] (75.28,310.31) -- (17.67,310.31) -- (74.19,252.42) -- cycle ;
\draw [color={rgb, 255:red, 144; green, 19; blue, 254 }  ,draw opacity=1 ][line width=0.75]    (298.34,310.31) -- (298.34,19) ;
\draw [color={rgb, 255:red, 144; green, 19; blue, 254 }  ,draw opacity=1 ][line width=0.75]    (224.34,310.31) -- (224.34,19) ;
\draw  [draw opacity=0][fill={rgb, 255:red, 126; green, 211; blue, 33 }  ,fill opacity=0.4 ] (261.5,65) -- (224,102) -- (224.34,310.31) -- (261.34,309.31) -- cycle ;
\draw  [draw opacity=0][fill={rgb, 255:red, 126; green, 211; blue, 33 }  ,fill opacity=0.4 ] (298.34,19) -- (261.34,19) -- (261.5,65) -- (298.5,28) -- cycle ;
\draw  [draw opacity=0][fill={rgb, 255:red, 74; green, 144; blue, 226 }  ,fill opacity=0.4 ] (130.27,196.71) -- (130.45,310.31) -- (75.28,310.31) -- (75.56,251.42) -- cycle ;
\draw  [draw opacity=0][fill={rgb, 255:red, 189; green, 16; blue, 224 }  ,fill opacity=0.4 ] (224.34,19) -- (261.34,19) -- (261.5,65) -- (224,102) -- cycle ;

\draw (3.83,14.32) node [anchor=north west][inner sep=0.75pt]  [color={rgb, 255:red, 74; green, 144; blue, 226 }  ,opacity=1 ]  {$L$};
\draw (367.85,314.72) node [anchor=north west][inner sep=0.75pt]  [color={rgb, 255:red, 208; green, 2; blue, 27 }  ,opacity=1 ]  {$r$};
\draw (67.97,315.91) node [anchor=north west][inner sep=0.75pt]    {$L_{k}$};
\draw (119.58,315.82) node [anchor=north west][inner sep=0.75pt]    {$L_{k-1}$};
\draw (256.84,316.21) node [anchor=north west][inner sep=0.75pt]  [color={rgb, 255:red, 0; green, 0; blue, 0 }  ,opacity=1 ]  {$L_{4}$};
\draw (169.37,216.8) node [anchor=north west][inner sep=0.75pt]  [color={rgb, 255:red, 65; green, 117; blue, 5 }  ,opacity=1 ]  {$\cdots \ $};
\draw (24.15,188.69) node [anchor=north west][inner sep=0.75pt]  [font=\tiny]  {$k^{2} -3k+1$};
\draw (82.05,159.2) node [anchor=north west][inner sep=0.75pt]  [font=\tiny]  {$k^{2} -5k+5$};
\draw (30.71,297.92) node [anchor=north west][inner sep=0.75pt]  [font=\tiny]  {$k^{2} -5k+5$};
\draw (169.18,61.02) node [anchor=north west][inner sep=0.75pt]  [color={rgb, 255:red, 65; green, 117; blue, 5 }  ,opacity=1 ]  {$\cdots \ $};
\draw (294.97,315.91) node [anchor=north west][inner sep=0.75pt]  [color={rgb, 255:red, 0; green, 0; blue, 0 }  ,opacity=1 ]  {$L_{3}$};
\draw (274.32,27.48) node [anchor=north west][inner sep=0.75pt]  [font=\tiny,rotate=-359.24]  {$1$};
\draw (164.97,316.4) node [anchor=north west][inner sep=0.75pt]  [color={rgb, 255:red, 0; green, 0; blue, 0 }  ,opacity=1 ]  {$\cdots $};
\draw (229.97,316.4) node [anchor=north west][inner sep=0.75pt]  [color={rgb, 255:red, 0; green, 0; blue, 0 }  ,opacity=1 ]  {$L_{5}$};
\draw (239.32,173.48) node [anchor=north west][inner sep=0.75pt]  [font=\tiny,rotate=-359.24]  {$1$};
\draw (277.32,173.48) node [anchor=north west][inner sep=0.75pt]  [font=\tiny,rotate=-359.24]  {$0$};
\draw (321.32,173.48) node [anchor=north west][inner sep=0.75pt]  [font=\tiny,rotate=-359.24]  {$0$};
\draw (100.05,274.2) node [anchor=north west][inner sep=0.75pt]  [font=\tiny] {$\ast$}; 
\draw (239.32,46.48) node [anchor=north west][inner sep=0.75pt]  [font=\tiny,rotate=-359.24] {$\ast$}; 

\end{tikzpicture}
 }
   \caption{The rank of $H_1((\mathsf{Star}_k)^2_{r,\mathbf{L}})$ for all $(r,L)$ in the shaded region. In particular, $\rank H_1((\mathsf{Star}_k)^2_{r,\mathbf{L}})=0$ when $(r,L)$ lies in the gray region. The entries labeled “$\ast$’’ represent values that depend on $k$ and are not determined in the previous theorems.}
    \label{fig:hpl-arrange-0'}
\end{figure}

\subsection{Decomposing $PH_1((\mathsf{Star}_k)^2_{-,-};\mathbb{F})$}\label{sec:4.2-decomp}
Now we compute the indecomposable direct summands for $PH_1((\mathsf{Star}_k)^2_{-,-};\mathbb{F})$ with the edge length $\mathbf{L}=(L,L_2,\dots, L_k)$ where $L_2,\dots, L_k$ are arbitrary but fixed positive real numbers. Without loss of generality, we assume $L_k\leq L_{k-1}\leq\cdots\leq L_2$. We need to find a basis for $PH_1((\mathsf{Star}_k)^2_{r,\mathbf{L}};\mathbb{F})$ for each $(r,L)$ such that the basis of $PH_1((\mathsf{Star}_k)^2_{r,\mathbf{L}};\mathbb{F})$ and the basis of $PH_1((\mathsf{Star}_k)^2_{r',\mathbf{L}'};\mathbb{F})$  (where $\mathbf{L}'=(L',L_2,\dots, L_k)$) are compatible for all $(r,L)\leq (r',L')$.

For each $L>0$ and $k\geq 4$, we construct the \key{reduced} $(G_k)_{\mathbf{L}}$ as follows:
\begin{itemize}
    \item Remove the vertex $xy$ for all $xy\in V$ such that $x,y\neq 0$, where $V$ is the set of vertices of $(G_k)_{\mathbf{L}}$.Denote the resulting vertex set by $V'$.
    \item Remove the edges $\{x0, xy\}$ and $\{xy,0y\}$ for all $x,y\neq 0$. Moreover, add a new edge $\{x0,0y\}$ for all $x\neq y$ and $x,y\neq 0$. This gives the incidence function $\psi'$. Let $E'$ denote the resulting edge set after deleting these edges and inserting the new ones.
    \item (filtering function on $V'$)  $f_{V'}=\restr{f_V}{V'}$
    \item (filtering function on $E'$) For all $\{x_1y_1, x_2y_2\}\in E'$, 
    $$f_E(\{x_1y_1, x_2y_2\})=\min\{f_{V'}(x_1y_1), f_{V'}(x_2y_2)\}$$
\end{itemize}

For example, when $k=4$, the reduced $(G_4)_{\mathbf{L}}$ is provided in Figure~\ref{fig:14}. A spanning tree of the reduced $(G_4)_{\mathbf{L}}$ is given in Figure~\ref{fig:15}.

\begin{minipage}[t]{0.5\textwidth}
   \centering
    \includegraphics[width=0.9\textwidth]{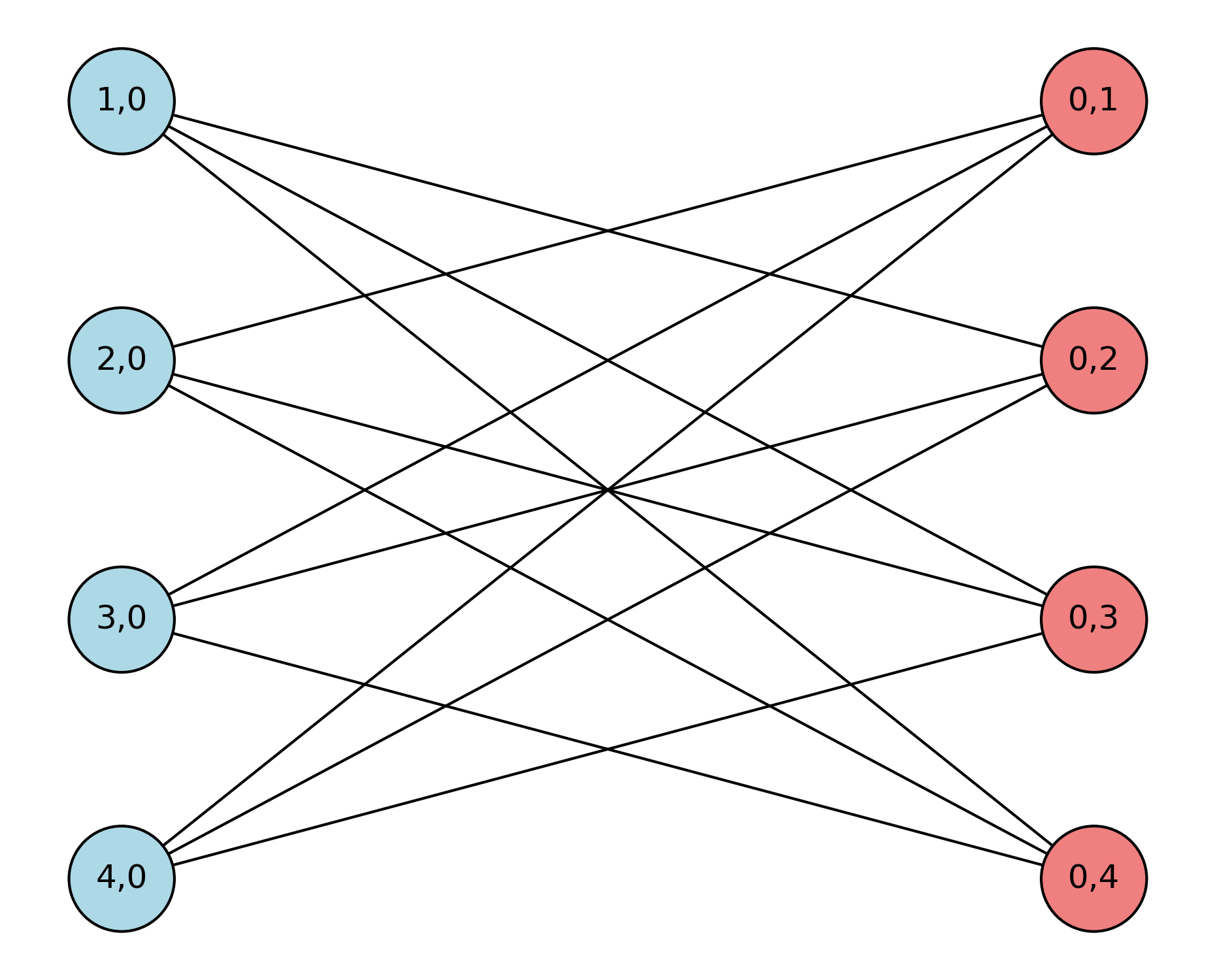}
    \captionof{figure}{The reduced $(G_4)_{\mathbf{L}}$.}
    \label{fig:14}
\end{minipage}%
\begin{minipage}[t]{0.5\textwidth}
    \centering
    \includegraphics[width=0.9\textwidth]{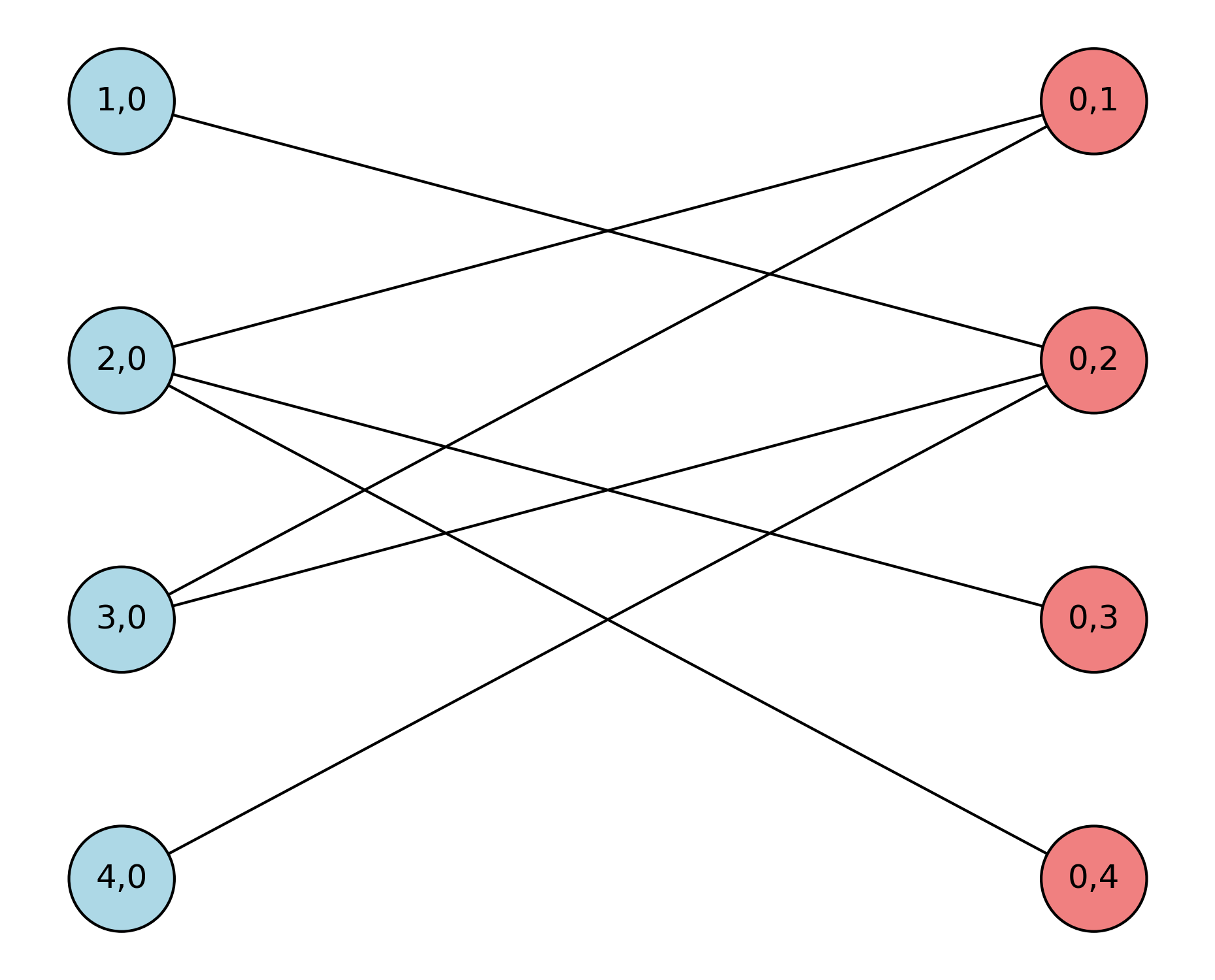}
    \captionof{figure}{A spanning tree of the reduced $(G_4)_{\mathbf{L}}$.}
    \label{fig:15}

 \end{minipage}

For each $r>0$, we set $$V'_r=f_{V'}^{-1}[r,+\infty)\mbox{\quad and \quad} E'_r=f_{E'}^{-1}[r,+\infty)$$ and define $\psi'_r: E'_r\rightarrow {V'_r\choose 2}$ by $\psi'_r(\{x_1y_1,x_2y_2\})=\psi'(\{x_1y_1,x_2y_2\})$ for all $\{x_1y_1,x_2y_2\}\in E'_r$. 

\begin{lemma}
$\psi'_r$ is an incidence function, and the resulting graph is a subgraph of the reduced $(G_k)_{\mathbf{L}}$.    
\end{lemma}
\begin{proof}
 It suffices to show that $\psi'_r$ is well-defined. For all $\{x0,0y\}\in E'_r$, note that $$r\leq f_{E'}(\{x0,0y\})=\min\{f_{V'}(x0), f_{V'}(0y)\}$$  
 Hence $f_{V'}(x0)\geq r$ and $f_{V'}(0y)\geq r$, which implies that $x0\in V'_r$ and $0y\in V'_r$. Therefore, $\psi'_r$ is well-defined.
\end{proof}

For each $r>0$ and $\mathbf{L}\in\mathbb{R}^k$, we 
define the \key{reduced $(G_k)_{r,\mathbf{L}}$}, denoted by $(G_k)'_{r,\mathbf{L}}$, to be
the subgraph of reduced $(G_k)_{\mathbf{L}}$  whose vertex set is $V'_r$ and whose edge set is $E'_r$, equipped with incidence function $\psi_r'$.

\begin{theorem}
    Given $k\geq 4$ and  $\mathbf{L}\in(\mathbb{R}_{>0})^k$, $H_1(\lVert(G_k)'_{r,\mathbf{L}}\rVert)\cong H_1(\lVert(G_k)_{r,\mathbf{L}}\rVert)$ for all $r>0$. 
\end{theorem}
\begin{proof}
  Assume $(G_k)_{r,\mathbf{L}}$ has $a$ vertices, $b$ edges, and $x$ path components. In particular, $(G_k)_{r,\mathbf{L}}$ has $a_i$ vertices of the form $xy$ where $x,y\neq 0$ with degree $i$, where $i=0,1,2$. For all vertices $x0\in V_r$ or $0y\in V_r$, note that $f_{V'}=\restr{f_V}{V'}$, therefore, $x0\in V'_r$ or $0y\in V'_r$. Hence $$\lvert V'_r\rvert=a-a_0-a_1-a_2$$
  Moreover, since $f_{E'}(\{x_1y_1, x_2y_2\})=\min\{f_{V'}(x_1y_1), f_{V'}(x_2y_2)\}$ for all $\{x_1y_1, x_2y_2\}\in E'$, $$\lvert E'_r\rvert=b-a_1-2a_2+a_2=b-a_1-a_2$$
  and the number of path components of $(G_k)_{r,\mathbf{L}}$ is $x-a_0$. By the Euler-Poincaré formula, 
  \begin{equation*}
      \begin{aligned}
         \rank(H_1(\lVert(G_k)'_{r,\mathbf{L}}\rVert))&=(b-a_1-a_2)-(a-a_0-a_1-a_2)+(x-a_0)\\
         &=b-a+x\\
         &=\rank(H_1(\lVert(G_k)_{r,\mathbf{L}}\rVert))
      \end{aligned}
\end{equation*}

Note that $\lVert(G_k)_{r,\mathbf{L}}\rVert$ and $\lVert(G_k)'_{r,\mathbf{L}}\rVert$ are finite simplicial complexes that are at most one-dimensional, hence both $H_1(\lVert(G_k)_{r,\mathbf{L}}\rVert)$ and $H_1(\lVert(G_k)'_{r,\mathbf{L}}\rVert)$ are finitely generated free abelian groups. Therefore, $$H_1(\lVert(G_k)_{r,\mathbf{L}}\rVert)\cong H_1(\lVert(G_k)'_{r,\mathbf{L}}\rVert)$$
\end{proof}

Notice that for all $r,L>0$,  $\lVert(G_k)_{r,\mathbf{L}}\rVert$ (where $\mathbf{L}=(L,L_2,\dots,L_k)$) is at most a 1-dimensional CW complex, therefore, the set consists of its fundamental cycles forms a basis for $H_1(\lVert(G_k)_{r,\mathbf{L}}\rVert;\mathbb{F})$. A standard method for enumerating fundamental cycles is to select a spanning tree of the reduced $(G_k)_{r,\mathbf{L}}$ and to obtain the fundamental cycles generated by the remaining edges: For each edge $e$ of the reduced $(G_k)_{r,\mathbf{L}}$ that is not contained in the spanning tree, there exists a unique path in the spanning tree connecting the two vertices incident to $e$. Adding $e$ to the path, we obtain a fundamental cycle of the reduced $(G_k)_{r,\mathbf{L}}$. In our context, however, we seek spanning trees for each reduced  $(G_k)_{r,\mathbf{L}}$ whose associated fundamental cycles give a basis for $H_1(\lVert(G_k)_{r,\mathbf{L}}\rVert;\mathbb{F})$ that is compatible across all parameters $r,L>0$.

Consider $k=4$. Assume $r\leq L$ and $r\leq L_4$. Since there are five edges in the reduced $(G_4)_{r,\mathbf{L}}$ that are not contained in the spanning tree, there are five fundamental cycles in the reduced $(G_4)_{r,\mathbf{L}}$, which are provided in Figure~\ref{fig:16}.

\begin{figure}[htbp!]
    \centering
    \begin{subfigure}[t]{0.33\textwidth}
        \centering
        \includegraphics[height=\textwidth]{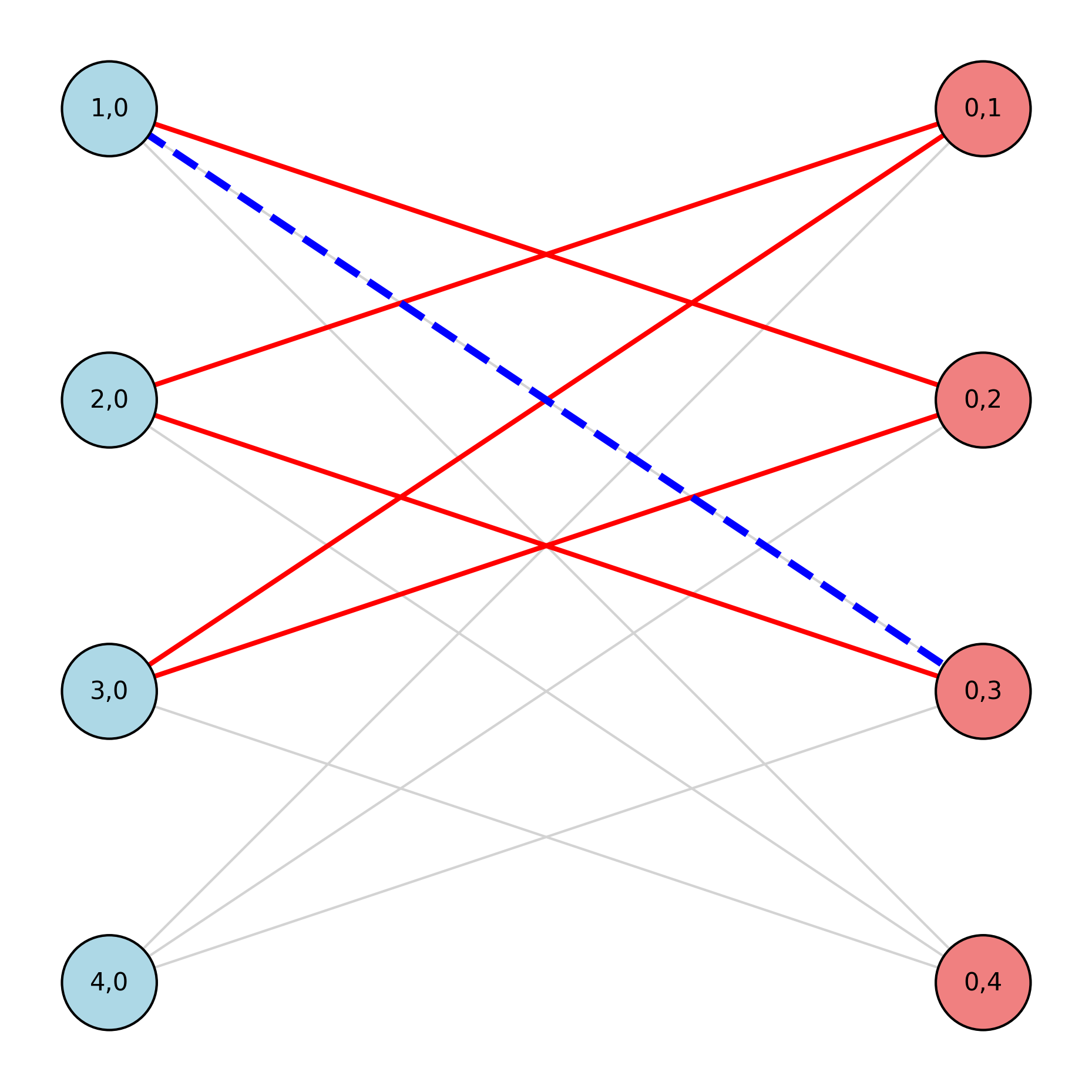}
        \caption{$X_1$}
    \end{subfigure}%
     ~ 
    \begin{subfigure}[t]{0.33\textwidth}
        \centering
        \includegraphics[height=\textwidth]{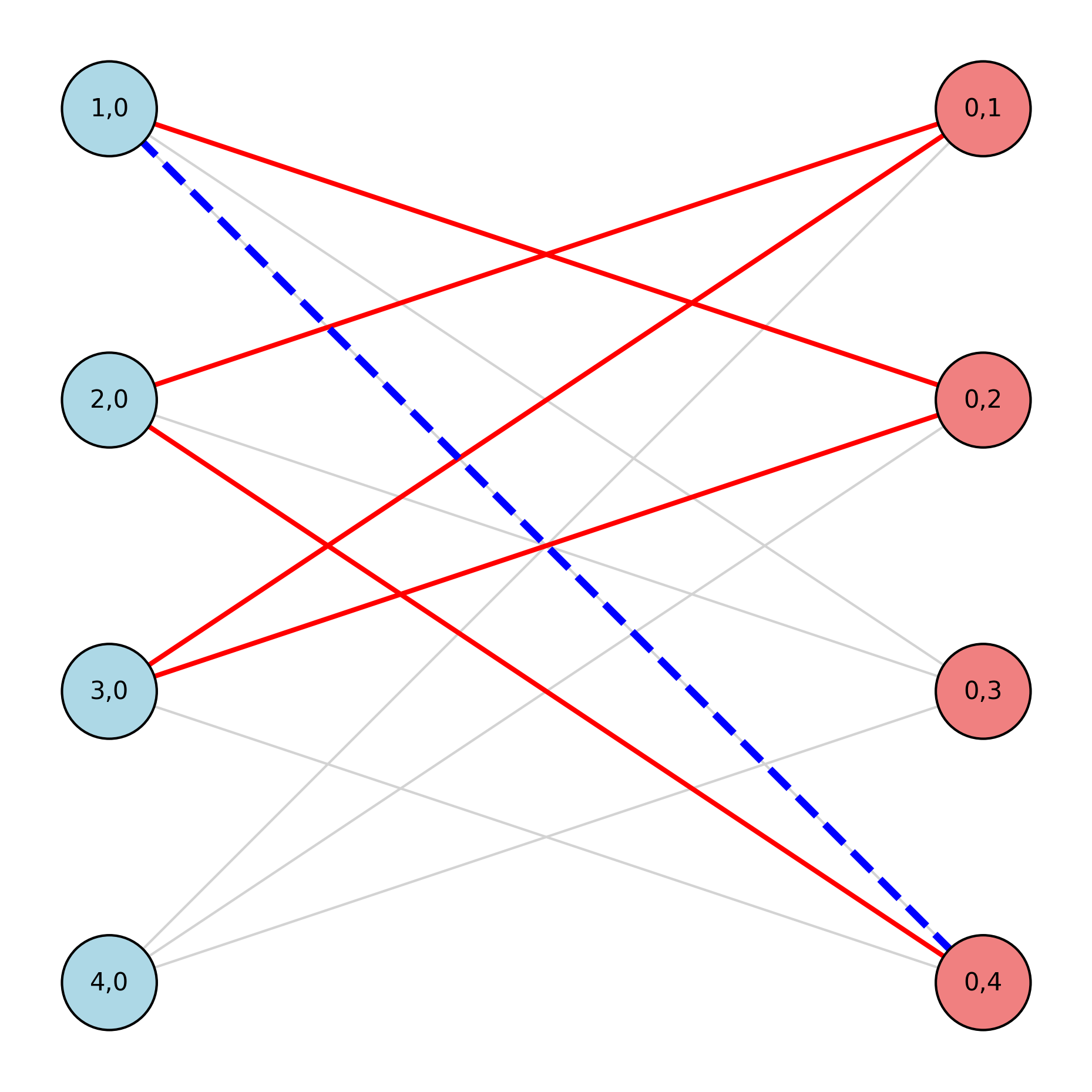}
        \caption{$X_2$}
    \end{subfigure}%
     ~ 
    \begin{subfigure}[t]{0.33\textwidth}
        \centering
        \includegraphics[height=\textwidth]{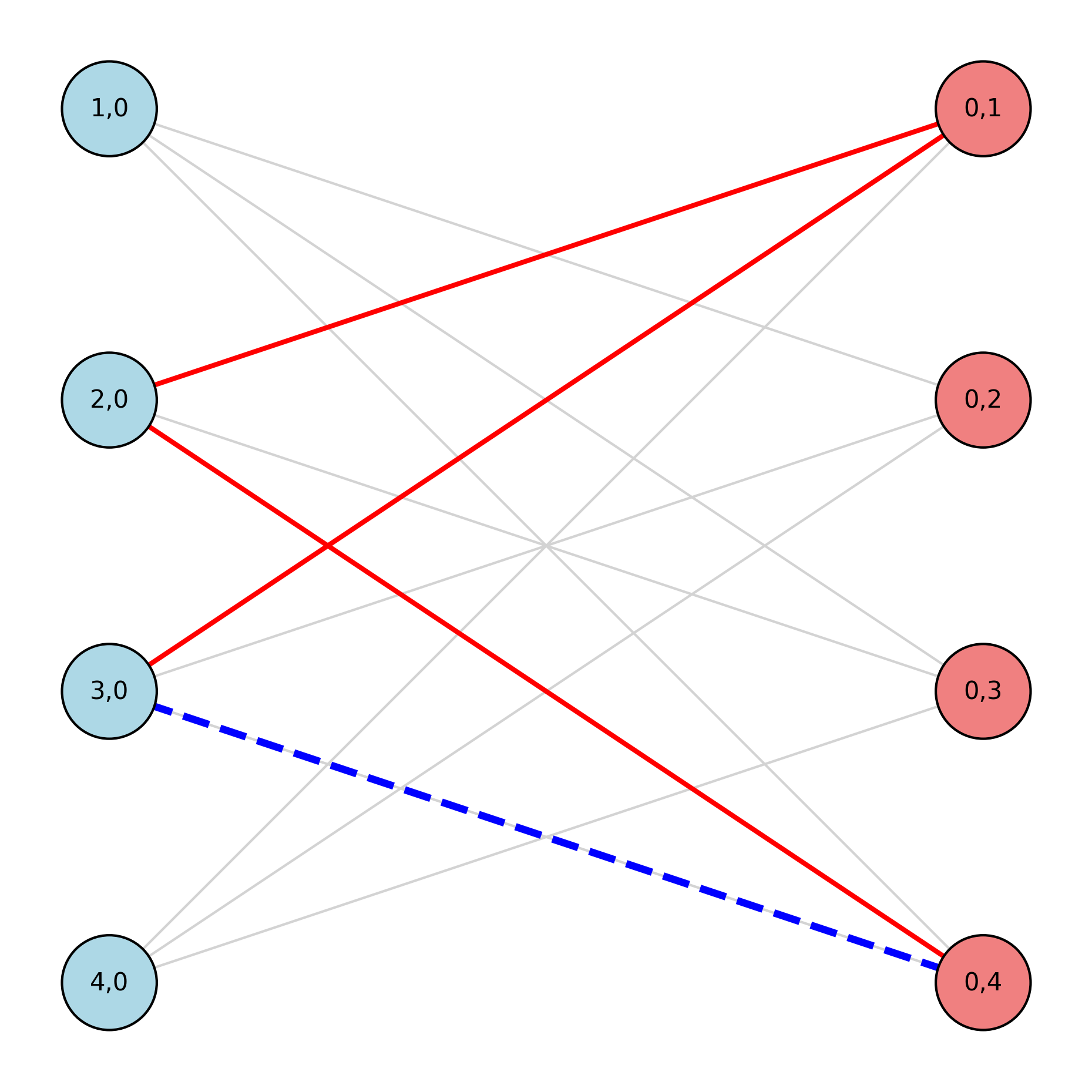}
        \caption{$X_3$}
    \end{subfigure}
    
    \begin{subfigure}[t]{0.33\textwidth}
        \centering
        \includegraphics[height=\textwidth]{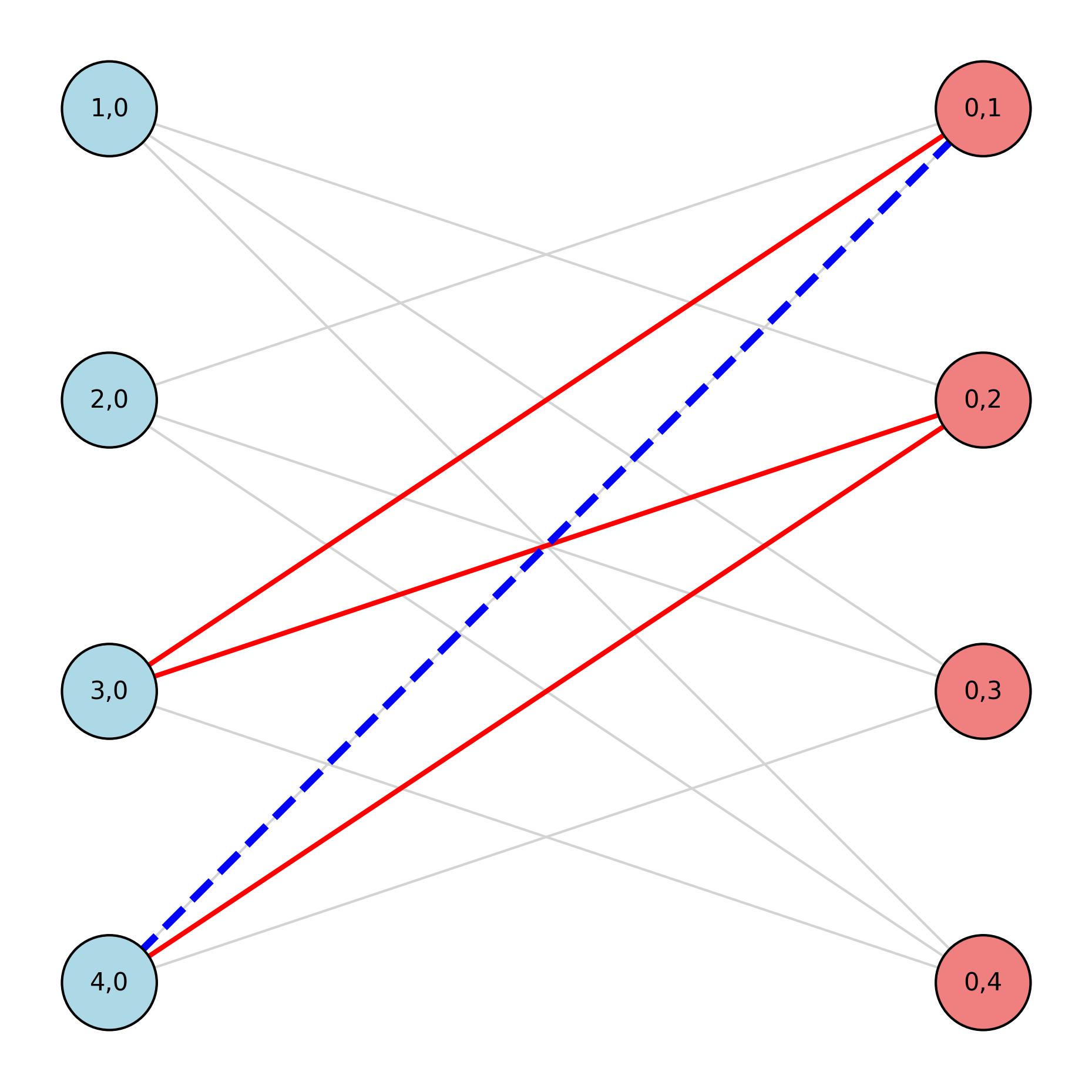}
        \caption{$X_4$}
    \end{subfigure}%
    \begin{subfigure}[t]{0.33\textwidth}
        \centering
        \includegraphics[height=\textwidth]{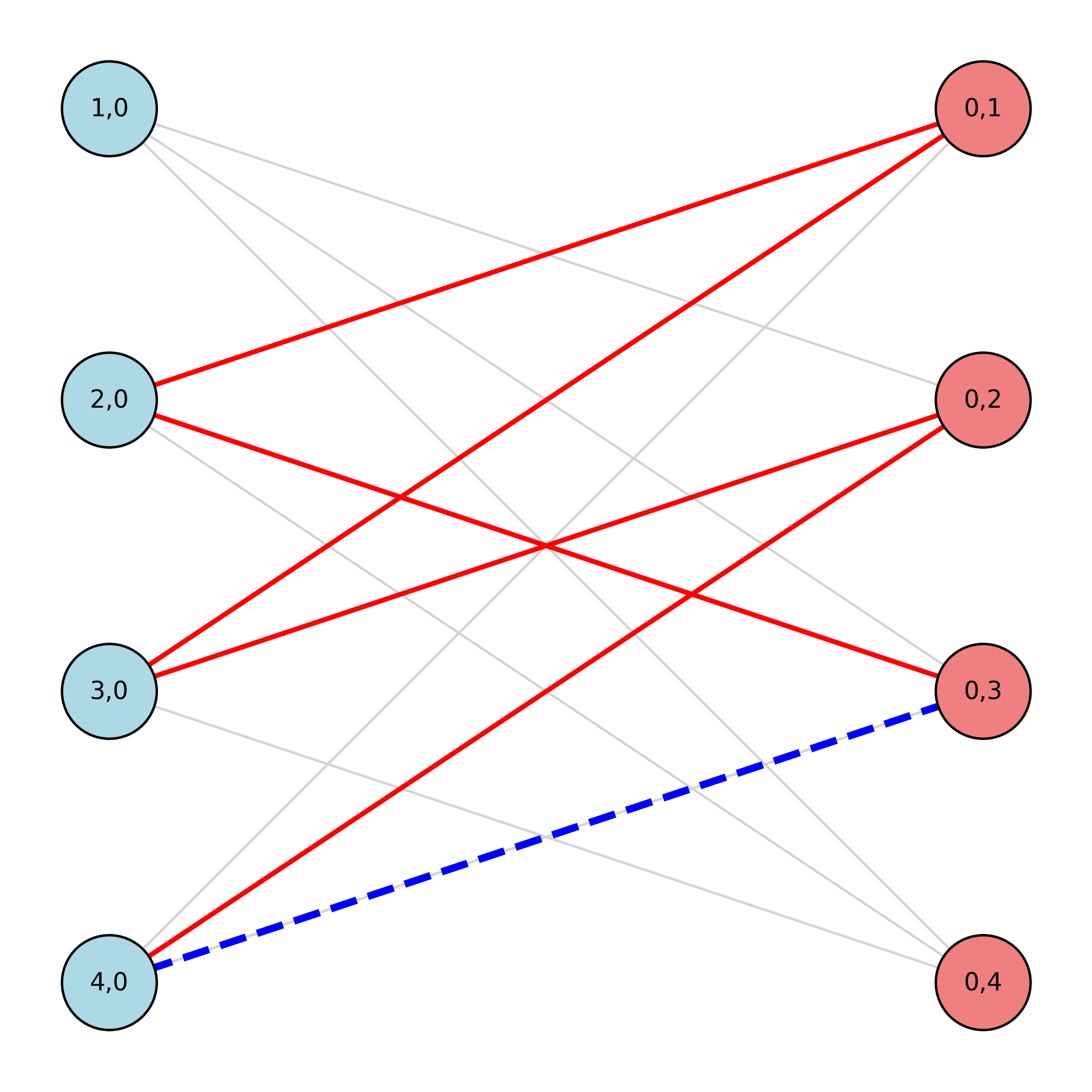}
        \caption{$X_5$}
    \end{subfigure}%
    \caption{Fundamental cycles of the reduced $(G_4)_{r,\mathbf{L}}$ when $r<L$ and $r<L_4$.}
    \label{fig:16}
\end{figure}

When $r\leq L$ and $L_4<r\leq L_3$, the fundamental cycles of the reduced $(G_4)_{r,\mathbf{L}}$ is given in Figure~\ref{fig:17}. It is the fundamental cycle $X_1$. When $L<r\leq L+L_4$ and $r\leq L_4$, the fundamental cycles (denoted by $Y$) of the reduced $(G_4)_{r,\mathbf{L}}$ is given in Figure~\ref{fig:18}, which can be obtained by taking the symmetric difference of $X_3$ and $X_5$.

\begin{minipage}[t]{0.5\textwidth}
   \centering
    \includegraphics[width=0.65\textwidth]{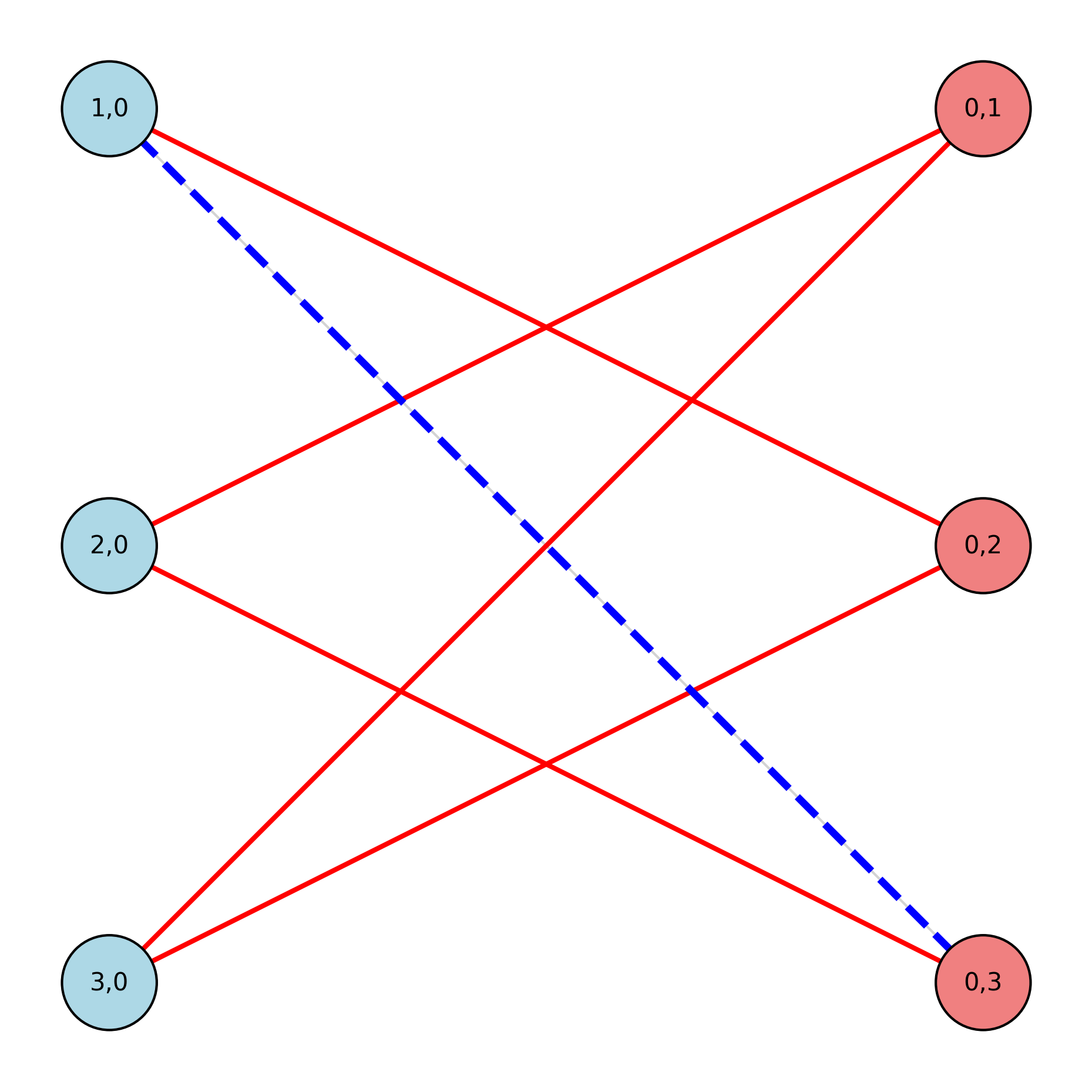}
    \captionof{figure}{The fundamental cycle of the reduced $(G_4)_{r,\mathbf{L}}$ when $r\leq L$ and $L_4<r\leq L_3$.}
    \label{fig:17}
\end{minipage}%
\begin{minipage}[t]{0.5\textwidth}
    \centering
    \includegraphics[width=0.65\textwidth]{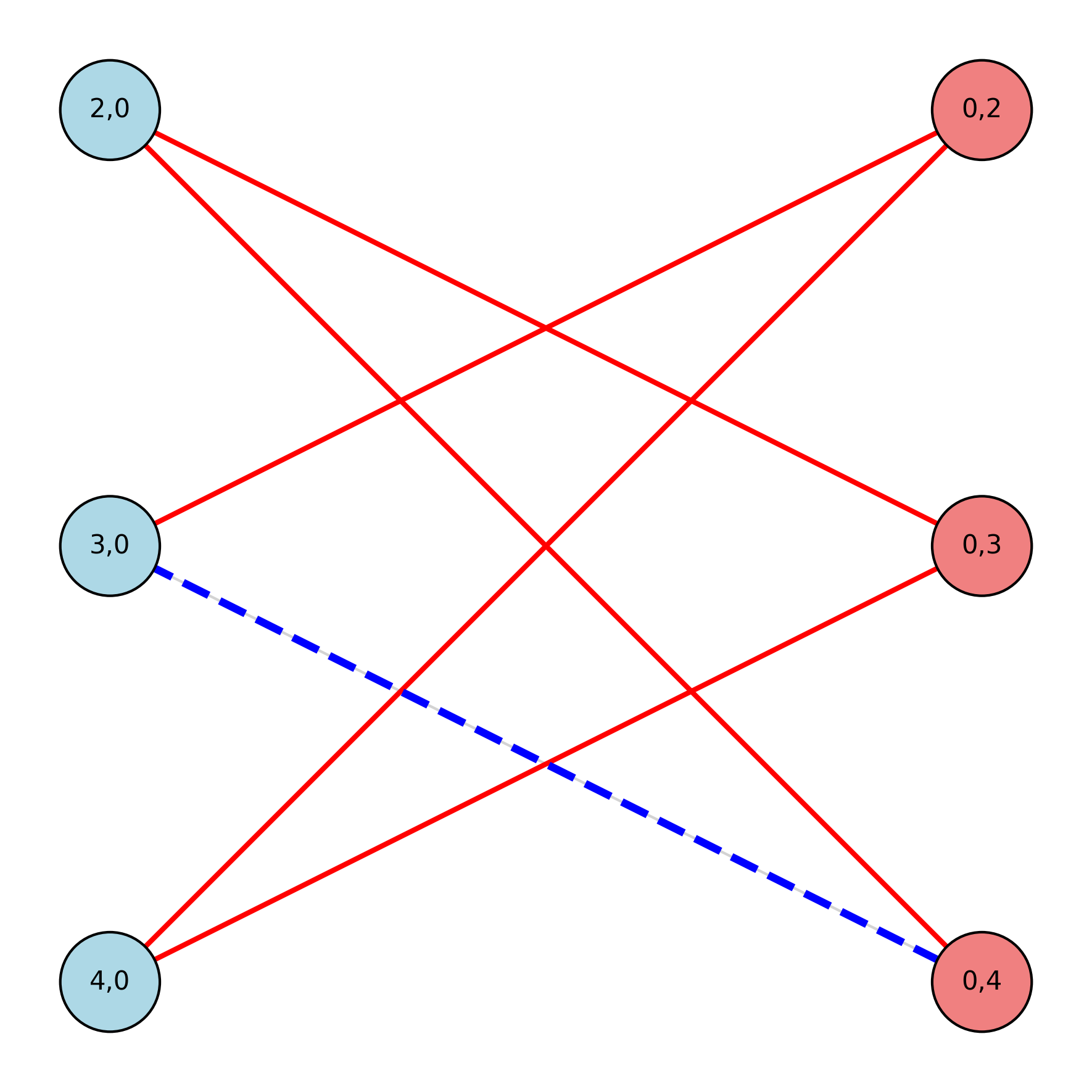}
    \captionof{figure}{The fundamental cycle of the reduced $(G_4)_{r,\mathbf{L}}$ when $L<r\leq L+L_4$ and $r\leq L_4$.}
    \label{fig:18}

 \end{minipage}

\vspace{0.2cm}

This combinatorial analysis of the fundamental cycles of the reduced $(G_4)_{r,\mathbf{L}}$ provides the data needed to track the behavior of $PH_1((\mathsf{Star}_4)^2_{r,\mathbf{L}};\mathbb{F})$ as parameters $r$ and $L$ vary. In particular, it enables an explicit computation of the indecomposable summands of $PH_1((\mathsf{Star}_4)^2_{-,-};\mathbb{F})$.

\begin{theorem}\label{thm:5.8}
    Let $\mathbf{L}=(L,L_2,L_3,L_4)$ where $L_4\leq L_3\leq L_2$ are fixed positive real numbers. The persistence module $PH_1((\mathsf{Star}_4)^2_{-,-};\mathbb{F})$ is interval-indecomposable and its indecomposable direct summands are the interval modules whose supports are given by the shaded regions in Figure~\ref{fig:19}.
    \begin{figure}[htbp!]
    \centering
    \resizebox{\textwidth}{!}{

\tikzset{every picture/.style={line width=0.75pt}} 

\begin{tikzpicture}[x=0.75pt,y=0.75pt,yscale=-1,xscale=1]

\draw    (2,223) -- (212.5,223) ;
\draw [shift={(214.5,223)}, rotate = 180] [color={rgb, 255:red, 0; green, 0; blue, 0 }  ][line width=0.75]    (10.93,-3.29) .. controls (6.95,-1.4) and (3.31,-0.3) .. (0,0) .. controls (3.31,0.3) and (6.95,1.4) .. (10.93,3.29)   ;
\draw    (22,243) -- (22,16) ;
\draw [shift={(22,14)}, rotate = 90] [color={rgb, 255:red, 0; green, 0; blue, 0 }  ][line width=0.75]    (10.93,-3.29) .. controls (6.95,-1.4) and (3.31,-0.3) .. (0,0) .. controls (3.31,0.3) and (6.95,1.4) .. (10.93,3.29)   ;
\draw [color={rgb, 255:red, 144; green, 19; blue, 254 }  ,draw opacity=0.4 ][line width=0.75]    (21.5,223) -- (214.5,30) ;
\draw [color={rgb, 255:red, 144; green, 19; blue, 254 }  ,draw opacity=0.4 ][line width=0.75]    (60.5,223) -- (213.5,70) ;
\draw [color={rgb, 255:red, 144; green, 19; blue, 254 }  ,draw opacity=0.4 ][line width=0.75]    (60.5,223) -- (60.5,24) ;
\draw [color={rgb, 255:red, 144; green, 19; blue, 254 }  ,draw opacity=0.4 ][line width=0.75]    (126,223) -- (126,24) ;
\draw [color={rgb, 255:red, 144; green, 19; blue, 254 }  ,draw opacity=0.4 ][line width=0.75]    (94,223) -- (94,24) ;
\draw [color={rgb, 255:red, 144; green, 19; blue, 254 }  ,draw opacity=0.4 ][line width=0.75]    (94,223) -- (213.5,103.5) ;
\draw [color={rgb, 255:red, 144; green, 19; blue, 254 }  ,draw opacity=0.4 ][line width=0.75]    (126,223) -- (214,135) ;
\draw    (223,223) -- (433.5,223) ;
\draw [shift={(435.5,223)}, rotate = 180] [color={rgb, 255:red, 0; green, 0; blue, 0 }  ][line width=0.75]    (10.93,-3.29) .. controls (6.95,-1.4) and (3.31,-0.3) .. (0,0) .. controls (3.31,0.3) and (6.95,1.4) .. (10.93,3.29)   ;
\draw    (243,243) -- (243,16) ;
\draw [shift={(243,14)}, rotate = 90] [color={rgb, 255:red, 0; green, 0; blue, 0 }  ][line width=0.75]    (10.93,-3.29) .. controls (6.95,-1.4) and (3.31,-0.3) .. (0,0) .. controls (3.31,0.3) and (6.95,1.4) .. (10.93,3.29)   ;
\draw [color={rgb, 255:red, 144; green, 19; blue, 254 }  ,draw opacity=0.4 ][line width=0.75]    (242.5,223) -- (435.5,30) ;
\draw [color={rgb, 255:red, 144; green, 19; blue, 254 }  ,draw opacity=0.4 ][line width=0.75]    (281.5,223) -- (434.5,70) ;
\draw [color={rgb, 255:red, 144; green, 19; blue, 254 }  ,draw opacity=0.4 ][line width=0.75]    (281.5,223) -- (281.5,24) ;
\draw [color={rgb, 255:red, 144; green, 19; blue, 254 }  ,draw opacity=0.4 ][line width=0.75]    (347,223) -- (347,24) ;
\draw [color={rgb, 255:red, 144; green, 19; blue, 254 }  ,draw opacity=0.4 ][line width=0.75]    (315,223) -- (315,24) ;
\draw [color={rgb, 255:red, 144; green, 19; blue, 254 }  ,draw opacity=0.4 ][line width=0.75]    (315,223) -- (434.5,103.5) ;
\draw [color={rgb, 255:red, 144; green, 19; blue, 254 }  ,draw opacity=0.4 ][line width=0.75]    (347,223) -- (435,135) ;
\draw    (440,223) -- (650.5,223) ;
\draw [shift={(652.5,223)}, rotate = 180] [color={rgb, 255:red, 0; green, 0; blue, 0 }  ][line width=0.75]    (10.93,-3.29) .. controls (6.95,-1.4) and (3.31,-0.3) .. (0,0) .. controls (3.31,0.3) and (6.95,1.4) .. (10.93,3.29)   ;
\draw    (460,243) -- (460,16) ;
\draw [shift={(460,14)}, rotate = 90] [color={rgb, 255:red, 0; green, 0; blue, 0 }  ][line width=0.75]    (10.93,-3.29) .. controls (6.95,-1.4) and (3.31,-0.3) .. (0,0) .. controls (3.31,0.3) and (6.95,1.4) .. (10.93,3.29)   ;
\draw [color={rgb, 255:red, 144; green, 19; blue, 254 }  ,draw opacity=0.4 ][line width=0.75]    (459.5,223) -- (652.5,30) ;
\draw [color={rgb, 255:red, 144; green, 19; blue, 254 }  ,draw opacity=0.4 ][line width=0.75]    (498.5,223) -- (651.5,70) ;
\draw [color={rgb, 255:red, 144; green, 19; blue, 254 }  ,draw opacity=0.4 ][line width=0.75]    (498.5,223) -- (498.5,24) ;
\draw [color={rgb, 255:red, 144; green, 19; blue, 254 }  ,draw opacity=0.4 ][line width=0.75]    (564,223) -- (564,24) ;
\draw [color={rgb, 255:red, 144; green, 19; blue, 254 }  ,draw opacity=0.4 ][line width=0.75]    (532,223) -- (532,24) ;
\draw [color={rgb, 255:red, 144; green, 19; blue, 254 }  ,draw opacity=0.4 ][line width=0.75]    (532,223) -- (651.5,103.5) ;
\draw [color={rgb, 255:red, 144; green, 19; blue, 254 }  ,draw opacity=0.4 ][line width=0.75]    (564,223) -- (652,135) ;
\draw  [draw opacity=0][fill={rgb, 255:red, 245; green, 166; blue, 35 }  ,fill opacity=0.4 ] (60.5,24) -- (60.5,24) -- (60.5,223) -- (21.5,223) -- (21.5,24) -- cycle ;
\draw  [draw opacity=0][fill={rgb, 255:red, 184; green, 233; blue, 134 }  ,fill opacity=0.4 ] (315,24) -- (314.5,152) -- (242.5,223) -- (243.5,23) -- cycle ;
\draw  [draw opacity=0][fill={rgb, 255:red, 189; green, 16; blue, 224 }  ,fill opacity=0.4 ] (498.5,24) -- (498.5,24) -- (498.5,184) -- (459.5,223) -- (459.5,24) -- cycle ;

\draw (5,5.4) node [anchor=north west][inner sep=0.75pt]  [color={rgb, 255:red, 74; green, 144; blue, 226 }  ,opacity=1 ]  {$L$};
\draw (216.5,226.4) node [anchor=north west][inner sep=0.75pt]  [color={rgb, 255:red, 208; green, 2; blue, 27 }  ,opacity=1 ]  {$r$};
\draw (49,227.4) node [anchor=north west][inner sep=0.75pt]    {$L_{4}$};
\draw (85,227.4) node [anchor=north west][inner sep=0.75pt]    {$L_{3}$};
\draw (119,227.4) node [anchor=north west][inner sep=0.75pt]    {$L_{2}$};
\draw (226,5.4) node [anchor=north west][inner sep=0.75pt]  [color={rgb, 255:red, 74; green, 144; blue, 226 }  ,opacity=1 ]  {$L$};
\draw (449.5,226.4) node [anchor=north west][inner sep=0.75pt]  [color={rgb, 255:red, 208; green, 2; blue, 27 }  ,opacity=1 ]  {$r$};
\draw (270,227.4) node [anchor=north west][inner sep=0.75pt]    {$L_{4}$};
\draw (306,227.4) node [anchor=north west][inner sep=0.75pt]    {$L_{3}$};
\draw (340,227.4) node [anchor=north west][inner sep=0.75pt]    {$L_{2}$};
\draw (443,5.4) node [anchor=north west][inner sep=0.75pt]  [color={rgb, 255:red, 74; green, 144; blue, 226 }  ,opacity=1 ]  {$L$};
\draw (654.5,226.4) node [anchor=north west][inner sep=0.75pt]  [color={rgb, 255:red, 208; green, 2; blue, 27 }  ,opacity=1 ]  {$r$};
\draw (487,227.4) node [anchor=north west][inner sep=0.75pt]    {$L_{4}$};
\draw (523,227.4) node [anchor=north west][inner sep=0.75pt]    {$L_{3}$};
\draw (557,227.4) node [anchor=north west][inner sep=0.75pt]    {$L_{2}$};
\draw (57,252) node [anchor=north west][inner sep=0.75pt]   [align=left] {Multiplicity=1};
\draw (278,252) node [anchor=north west][inner sep=0.75pt]   [align=left] {Multiplicity=1};
\draw (496,252) node [anchor=north west][inner sep=0.75pt]   [align=left] {Multiplicity=3};

\end{tikzpicture}
    }
    \caption{The
    indecomposable direct summands of $PH_1((\mathsf{Star}_4)^2_{-,-};\mathbb{F})$.}
    \label{fig:19}
\end{figure}
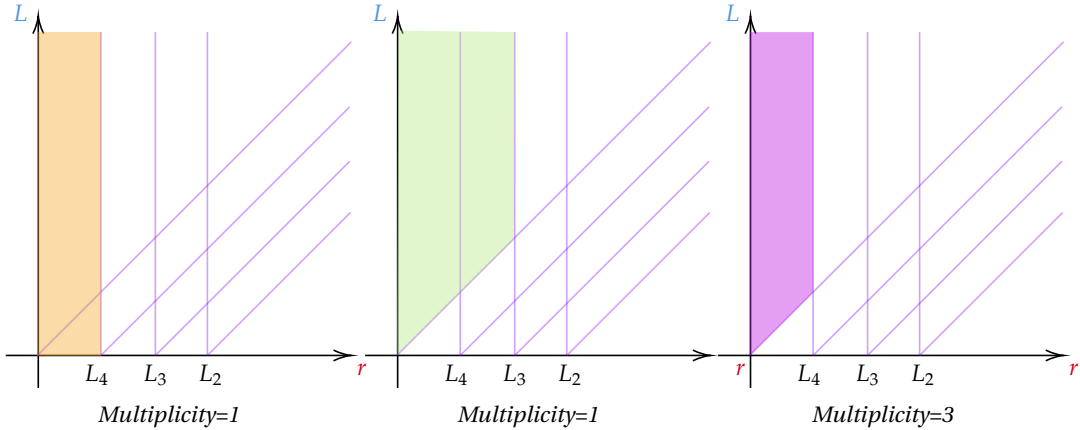
\end{theorem}
\begin{proof}
    Let $(P,\leq)$ be the poset of chambers of the parameter space of $(\mathsf{Star}_4)^2_{r,\mathbf{L}}$ and construct the representation $N$ over $(P,\leq)$ (the construction is provided in~\cite{li2024notes}) such that $PH_1((\mathsf{Star}_4)^2_{-,-};\mathbb{F})\cong N\circ \mathcal{F}$ where $\mathcal{F}:(\mathbb{R},\leq)\op\times (\mathbb{R},\leq)\rightarrow (P,\leq)$ is the canonical projection map. 

\begin{figure}[htbp!]
    \centering
    \begin{subfigure}{0.45\textwidth}
        \centering
        \begin{tikzcd}
\bullet & \bullet \arrow[l]\\
\bullet \arrow[u] &   
\end{tikzcd}
        \caption{The Hasse Diagram of the Poset of Chambers}
        \label{fig:Poset_P}
    \end{subfigure}
    \hfill
    \begin{subfigure}{0.45\textwidth}
        \centering
                \begin{tikzcd}
\mathbb{F}^{5} & \mathbb{F} \arrow[l] &\\
\mathbb{F} \arrow[u] & 
\end{tikzcd}
        \caption{Poset Representation $N$}
        \label{fig:Rep/P}
    \end{subfigure}
    
    \caption{}
\end{figure}
    
    Note that the support of $N$ is (a subgraph of) a type $A_3$ quiver. Gabriel's theorem implies that $N$ is interval-decomposable. Therefore, $PH_1((\mathsf{Star}_4)^2_{-,-};\mathbb{F})$ is interval-indecomposable. 
    
    To find the indecomposable direct summands of $PH_1((\mathsf{Star}_4)^2_{-,-};\mathbb{F})$, by Theorem~\ref{thm-sim-1'}, it suffices to find the indecomposable direct summands of $PH_1(\lVert(G_k)_{-,-}\rVert;\mathbb{F})$. We use the fundamental cycles of reduced $(G_k)_{r,\mathbf{L}}$ to provide a basis for each vector space in the representation $N$. We first fix an orientation on every edge of reduced $(G_k)_{r,\mathbf{L}}$ where $(r,L)$ is a representative of the chamber $\{(r,L)\in(\mathbb{R}_{>0})^2: r\leq L, r\leq L_4\}$. Note that for the other two chambers that have nonzero rank, there exists $(r',L')$ and $(r'',L'')$ from each chamber and $(r_0',L_0'),(r_0'',L_0'')\in \{(r,L)\in(\mathbb{R}_{>0})^2: r\leq L, r\leq L_4\}$ such that $(r',L')\leq (r'_0,L'_0)$ and $(r'',L'')\leq (r''_0,L''_0)$. Note that $(G_k)_{r,L}=(G_k)_{r'_0,L'_0}=(G_k)_{r''_0,L''_0}$, consequently, the reduced graphs $(G_k)_{r',\mathbf{L}'}$ and $(G_k)_{r'',\mathbf{L}''}$ embed as subgraphs of reduced $(G_k)_{r,\mathbf{L}}$, and we equip them with the orientations induced by restricting the globally chosen orientation on $(G_k)_{r,\mathbf{L}}$.

    
    Define the 1-chain $\sigma_i$ for each fundamental cycle $X_i$ as follows:
    \begin{equation}
        \begin{aligned}
            \sigma_1&=[10,02]-[30,02]+[30,01]-[20,01]+[20,03]-[10,03]\\
            \sigma_2&=[10,02]-[30,02]+[30,01]-[20,01]+[20,04]-[10,04]\\
            \sigma_3&=[20,01]-[30,01]+[30,04]-[20,04]\\
            \sigma_4&=[30,01]-[40,01]+[40,02]-[30,02]\\
            \sigma_5&=[20,01]-[30,01]+[30,02]-[40,02]+[40,03]-[20,03]
        \end{aligned}
    \end{equation}
    
    Define the 1-chain $\tau$ for each fundamental cycle $Y$ as follows:
    \begin{equation}
        \begin{aligned}
            \tau&=[20,03]-[40,03]+[40,02]-[30,02]+[30,04]-[20,04]
        \end{aligned}
    \end{equation}
    
    Note that $\tau=\sigma_3-\sigma_5$. Hence we obtained a compatible basis for the representation $N$. In particular, every linear transformation in the representation is an inclusion. 
    

     \begin{figure}[htbp!]
        \centering
                \begin{tikzcd}
<\sigma_1, \sigma_2, \sigma_3-\sigma_5, \sigma_4, \sigma_5> & <\tau> \arrow[l] &\\
<\sigma_1> \arrow[u] & 
\end{tikzcd}
        \caption{Compatible Basis for $N$}
        \label{fig:Rep/P}
    \end{figure}
Hence the indecomposable direct summands of $N$ are:

\begin{figure}[htbp!]
  \centering

  \begin{subfigure}[b]{0.33\textwidth}
    \centering
    \begin{tikzcd}
<\sigma_1> & 0 \arrow[l] &\\
<\sigma_1> \arrow[u] & 
\end{tikzcd}
    \caption{}
    \label{fig:sub-a}
  \end{subfigure}
  \hfill
  \begin{subfigure}[b]{0.33\textwidth}
    \centering
    \begin{tikzcd}
<\sigma_3-\sigma_5> & <\tau> \arrow[l] &\\
0 \arrow[u] & 
\end{tikzcd}
    \caption{}
    \label{fig:sub-b}
  \end{subfigure}
  \begin{subfigure}[b]{0.33\textwidth}
    \centering
    \begin{tikzcd}
<\sigma_i> & 0 \arrow[l] &\\
0 \arrow[u] & 
\end{tikzcd}
    \caption{$i=2,4,5$}
    \label{fig:sub-c}
  \end{subfigure}
  \caption{Indecomposable direct summands of $N$}
  \label{fig:Ind_Summand_N}
\end{figure}

Applying Theorem~\ref{thm-sim-1'}, we obtain the indecomposable direct summands of $PH_1((\mathsf{Star}_4)^2_{-,-};\mathbb{F})$ up to isomorphism, as given in Figure~\ref{fig:19}. (Each summand is an interval module supported on the shaded region.)
\end{proof}


The preceding argument establishes the base case of the induction. We now complete the inductive step in the following lemma.

\begin{lemma}\label{lem:5.10}
    For all $k\geq 3$ and $\mathbf{L}=(L,L_2,\dots, L_k)$ where $L_2,\dots, L_k$ are arbitrary but fixed positive real numbers, the 2-parameter persistence module $PH_1((\mathsf{Star}_k)^2_{-,-};\mathbf{F})$ is interval-decomposable.
\end{lemma}
\begin{proof}
  Theorem~\ref{thm:5.8} shows that the statement is true for $k=4$. Assume the statement is true when $k=N$. 

  Let $P_{N+1}$ be the poset of chambers for $(\mathsf{Star}_{N+1})^2_{r,\mathbf{L}^{(N+1)}}$ and construct the representation $M^{(N+1)}$ over $(P,\leq)$ (the construction is provided in~\cite{li2024notes}) such that  $$PH_1((\mathsf{Star}_{N+1})^2_{-,-};\mathbb{F})\cong M^{(N+1)}\circ \mathcal{F}^{(N+1)}$$ where $\mathcal{F}^{(N+1)}:(\mathbb{R}_{>0},\leq)\op\times (\mathbb{R}_{>0},\leq)\rightarrow P_{N+1}$ is the canonical projection map sending each $(r,L)$ to the chamber that contains $(r,L)$. 
  
  For all $r,L>0$, consider $(\mathsf{Star}_{N})^2_{r,\mathbf{L}^{(N)}}$ where $\mathsf{Star}_{N}$ is obtained from removing edge $e_{N+1}$ from $\mathsf{Star}_{N+1}$. Let $P_{N}$ denote the poset of chambers for $(\mathsf{Star}_{N})^2_{r,\mathbf{L}^{(N)}}$. Note that $P_N$ is a subposet of $P_{N+1}$ and the cardinality of $P_{N}$ equals the cardinality of $P_{N+1}$ minus $2$, because there are exactly two chambers in the hyperlane arrangement of $(\mathsf{Star}_{N+1})^2_{r,\mathbf{L}^{(N+1)}}$ whose elements satisfy $r\leq L_{N+1}$. Let $b,d$ be the elements in $P_{N+1}$ that are not in $P_N$. Without loss of generality, we assume $b\leq d$. Let $a,c$ be the elements of $P_N$ such that $a\leq c$, $b$ covers $a$ (denoted by $a\preceqdot b$), and $d$ covers $c$ (denoted by $c\preceqdot d$).

  
  Define a representation $\hat{M}^{(N+1)}$ as follows:
  $$\hat{M}^{(N+1)}p=\begin{cases} M^{(N+1)}p, & \mbox{if } p\neq a,c\\ \mathsf{Im}(M^{(N+1)}(a\leq b)), & \mbox{if } p=a\\
  \mathsf{Im}(M^{(N+1)}(c\leq d)), & \mbox{if } p=c\end{cases}$$
   $$\hat{M}^{(N+1)}(p\leq q)=\begin{cases} M^{(N+1)}(p\leq q), & \mbox{if } p,q\neq a,c\\ M^{(N+1)}(a\leq b)\circ M^{(N+1)}(p\leq a), & \mbox{if } q=a\\
   M^{(N+1)}(c\leq d)\circ M^{(N+1)}(p\leq c), & \mbox{if } q=c \mbox{ and } p\neq a\\
   \restr{M^{(N+1)}(b\leq d)}{\Ima M^{(N+1)}(a\leq b)}, & \mbox{if } p=a \mbox{ and } q=c\\
    M^{(N+1)}(a\leq q)\circ (M^{(N+1)}(a\leq b))^{-1}, & \mbox{if } p=a \mbox{ and } q\neq c\\
    M^{(N+1)}(c\leq q)\circ (M^{(N+1)}(c\leq d))^{-1}, & \mbox{if } p=c\\
   \end{cases}$$
where $(M^{(N+1)}(a\leq b))^{-1}$ is a morphism defined on ${\Ima M^{(N+1)}(a\leq b)}$ and $(M^{(N+1)}(c\leq d))^{-1}$ is a morphism defined on ${\Ima M^{(N+1)}(c\leq d)}$ because both $M^{(N+1)}(a\leq b)$ and $M^{(N+1)}(c\leq d)$ are injective (as shown in the proof of Corollary~\ref{Coro:2:11}).
It is clear that $\hat{M}^{(N+1)}(a\leq b)$ and $\hat{M}^{(N+1)}(c\leq d)$ are inclusions.

Define $\alpha: M \Rightarrow \hat{M}$ as follows: for each $p\in P_{N+1}$,

$$\alpha_p=\begin{cases} \id_{Mp}, & \mbox{if } p\neq a, c\\
 M^{(N+1)}(a\leq b), & \mbox{if } p=a\\
  M^{(N+1)}(c\leq d), & \mbox{if } p=c\\
   \end{cases}$$

Note that $\alpha$ is a homomorphism from $M$ to $\hat{M}$. Moreover, $\alpha_p$ is an isomorphism for all $p\in P_{N+1}$. Therefore, $\alpha$ is an isomorphism, i.e., $M\cong\hat{M}$. By a slight abuse of notation, we will write $M$ in place of $\hat{M}$ for the remainder of the proof.

Define a representation $A^{(N+1)}$ over $P_{N+1}$ as follows:

For each $p\in P_{N+1}$,
  $$A^{(N+1)}p=\begin{cases} M^{(N+1)}p, & \mbox{if } p\in P_N-\{a,c\} \\ \mathsf{Im}(M^{(N+1)}(a\leq b)), & \mbox{if } p=a \mbox{ or } p=b\\
  \mathsf{Im}(M^{(N+1)}(c\leq d)), & \mbox{if } p=c \mbox{ or } p=d\\
  \end{cases}
  $$
  
  For each $p\leq q\in P_{N+1}$,
  $$
  A^{(N+1)}(q\leq p)=\restr{M^{(N+1)}(q\leq p)}{A^{(N+1)}_q}$$

  Define a representation $B^{(N+1)}$ over $P_{N+1}$ as follows: 
  
  For each $p\in P_{N+1}$,
  $$B^{(N+1)}p=\begin{cases} 0, & \mbox{if } p\in P_N \\ \mathsf{coker}(M^{(N+1)}(a\leq b)), & \mbox{if } p=b\\
  \mathsf{coker}(M^{(N+1)}(c\leq d)), & \mbox{if } p=d\end{cases} 
  $$
  
  For each $p\leq q\in P_{N+1}$,
  $$B^{(N+1)}(q\leq p)=\restr{M^{(N+1)}(q\leq p)}{B^{(N+1)}_q}$$
  
  Since linear transformations $M^{(N+1)}(a\leq b)$ and $M^{(N+1)}(c\leq d)$ are inclusions, $A^{(N+1)}$ and $B^{(N+1)}$ are subrepresentations of $M^{(N+1)}$ and $M^{(N+1)}=A^{(N+1)}\oplus B^{(N+1)}$. Note that $M^{(N+1)}p=M^{(N)}p$ for all $p\in P_N\subseteq P_{N+1}$. By the induction hypothesis, $M^{(N)}$ is interval-decomposable, hence there exists a compatible basis (denoted by $Y_1,\dots, Y_{N^2-3N+1}$) for $M^{(N+1)}p$ and $p\in P_N$, where $Y_i$ is a linear combination of some fundamental cycles of $(G_N)_{\mathbf{L}^{(N)}}$. Since $A^{(N+1)}(a\leq b)$ and $A^{(N+1)}(c\leq d)$ are inclusion maps in $P_{N+1}$, the basis $Y_1,\dots, Y_{N^2-3N+1}$ is also a compatible basis for $A^{(N+1)}$. Hence $A^{(N+1)}$ is interval-decomposable. On the other hand, note that the support of $B^{(N+1)}$ is a type $A_2$ quiver; therefore, $B^{(N+1)}$ is interval-decomposable. 
\end{proof}

To organize the inductive argument and ensure compatibility of the fundamental cycles produced at each $(r,L)$, we now introduce the notion of a weight-biased spanning tree.

\begin{definition}
Let $G = (V,E)$ be a finite undirected connected simple graph with no loops, and let $w_V: V \to \mathbb{R}_{\ge 0}$ be a weight function assigning a nonnegative real value to each vertex. For each edge $\{u,v\} \in E$, define its \key{priority} by
    \[
    p(\{u,v\}) = \min\{w_V(u),\, w_V(v)\}.
    \]
Let $T=(V_T,E_T)$ be a spanning tree of $G$. $T$ is called a \key{(weight-)biased spanning tree} of $G$ if $T$ maximizes the total edge priority 
$$
\sum_{\{u,v\}\in E_T} p(\{u,v\})
$$
among all spanning trees of $G$.
\end{definition}

We now describe a greedy procedure that constructs such a tree as a subgraph of $G$.
\begin{enumerate}
    \item (\textbf{Initialization}) Let 
    \[
    v_0 = \arg\max_{v \in V} w_V(v)
    \]
    be a vertex of maximum weight. Set $T = (\{v_0\}, \varnothing)$.
    
    \item (\textbf{Growth rule}) At each step, among all edges $\{u,v\}$ of $G$ with exactly one endpoint in $T$, choose an edge of maximum priority $p(\{u,v\})$. Add this edge to $T$, together with its new endpoint.

    \item (\textbf{Termination}) Repeat until $T$ spans all vertices of $G$.
\end{enumerate}

This is Prim's algorithm for maximum spanning trees with weight function $p$, therefore, the output graph is a biased spanning tree.

\begin{lemma}\label{lemma:4:12}
    Let $G = (V_G,E_G)$ be a finite undirected connected simple graph (with no loops) with a weight function $w_G: V_G \to \mathbb{R}_{\ge 0}$. Let $v_0\in V_G$ be a vertex such that $w_G(v_0)=\min\limits_{v\in V_G}\{w_G(v)\}$. Define $H= (V_H,E_H)$ to be the (full) subgraph of $G$ such that $H=G-\{v_0\}$. If $H$ is connected, then any biased spanning tree of $H$ can be extended to a biased spanning tree of $G$.
\end{lemma}
\begin{proof}
Let $T=(V_T,E_T)$ be a spanning tree of $G$. Define $F=T-\{v_0\}$ and let $T'$ be a tree in $H=G-\{v_0\}$ that contains $F$. Since $T'$ and $H$ have the same vertex sets, $T'$ is a spanning tree of $H$. Note that $p(\{u,v_0\})=w_G(v_0)$ since $w_G(v_0)=\min\limits_{v\in V_G}\{w_G(v)\}$, hence
\begin{equation}\label{eq:4.6-1}
    \begin{aligned}
       \sum_{\{u,v\}\in E_T} p(\{u,v\})&=\sum_{\{u,v\}\in E_F} p(\{u,v\})+ \sum_{\{u,v_0\}\in E_T} p(\{u,v_0\})\\
       &\leq \sum_{\{u,v\}\in E_{T'}} p(\{u,v\})+\max\limits_{\{u,v_0\}\in E_G} p(\{u,v_0\})
    \end{aligned}
\end{equation} 

Let $T''$ be a biased spanning tree of $H$. Hence for any spanning tree $T'$ of $H$, 
\begin{equation}\label{eq:4.6-2}
     \sum_{\{u,v\}\in E_{T'}} p(\{u,v\})\leq \sum_{\{u,v\}\in E_{T''}} p(\{u,v\})
\end{equation}

Since $G$ is connected, there exists $u_0\in V_G$ such that $\{u_0, v_0\}\in E_G$. Define $T^{\ast}=(V^{\ast},E^{\ast})$ to be the graph where $V^{\ast}=V_{T''}\cup\{v_0\}$ and $E^{\ast}=E_{T''}\cup\{u_0,v_0\}$. It is clear that $T^{\ast}$ is a spanning tree of $G$. Moreover, equation (\ref{eq:4.6-1}) and (\ref{eq:4.6-2}) implies that for any spanning tree $T$ of $G$, 
\begin{equation}\label{eq:4.6-3}
    \begin{aligned}
       \sum_{\{u,v\}\in E_T} p(\{u,v\})&\leq \sum_{\{u,v\}\in E_{T'}} p(\{u,v\})+\max\limits_{\{u,v_0\}\in E_G} p(\{u,v_0\})\\
       &\leq \sum_{\{u,v\}\in E_{T''}} p(\{u,v\}) + p(\{u_0,v_0\})
    \end{aligned}
\end{equation}

Therefore, $T^{\ast}$ is a biased spanning tree of $G$.
\end{proof}

  Let $\mathbf{L}^{(N+1)}=(L, L_2, \dots, L_{N+1})$ be the edge-length vector of $(G_{N+1})_{\mathbf{L}^{(N+1)}}$. Without loss of generality, we assume $ L_{N+1}\leq L_{N}\leq L_{N-1}\leq \cdots \leq L_2$. Define a weight function $w_{N+1}$ on set of vertices of the reduced $(G_{N+1})_{\mathbf{L}^{(N+1)}}$ where $w_{N+1}(x0)=L_x$ and $w_{N+1}(0y)=L_y$ for all $x,y=1,\dots,N+1$. 
  
  Note that $\mathbf{L}^{(N)}:=(L, L_2, \dots, L_{N})$ is an edge-length vector for $(G_N)_{\mathbf{L}^{(N)}a}$, where $G_N$ is obtained by removing the vertices $(0,N+1)$ and $(N+1, 0)$ from $G_{N+1}$. Equipped $(G_N)_{\mathbf{L}^{(N)}}$ with the weight function $w_{N}=\restr{w_{N+1}}{V'_{N}}$,  where $V'_{N}$ is the vertex set of reduced $(G_N)_{\mathbf{L}^{(N)}}$. Let $T_N$ be the biased spanning tree of reduced $(G_N)_{\mathbf{L}^{(N)}}$ (when $r<L$ and $L<L_N$). Then we can construct a graph $T_{N+1}$ from $T_N$ by adding vertices (0,N+1) and (N+1,0) and adding edges $\{(0,N+1), (2,0)\}$ and $\{(N+1,0), (0,2)\}$ to $T_N$. It is easy to verify that $T_{N+1}$ is also a spanning tree and $T_{N}$ is a subgraph of $T_{N+1}$. By Lemma~\ref{lemma:4:12}, $T_{N+1}$ is a biased spanning tree of $(G_{N+1})_{\mathbf{L}^{(N+1)}}$. Therefore, every fundamental cycle induced by $T_N$ in reduced graph $(G_N)_{\mathbf{L}^{(N)}}$ remains a fundamental cycle of reduced $(G_{N+1})_{\mathbf{L}^{(N+1)}}$. On the other hand, the fundamental cycles of $(G_{N+1})_{\mathbf{L}^{(N+1)}}$ that are not fundamental cycles of $(G_N)_{\mathbf{L}^{(N)}}$ appear no earlier than $f_{V}((0,N+1))$, that is, they lies in $(\mathsf{Star}_K)^2_{r,L}$ where $r\leq L_{N+1}$. 

The indecomposable direct summands of $PH_1((\mathsf{Star}_k)^2_{-,-};\mathbb{F})$ and their multiplicities can be computed from the inductive construction when $k\geq 5$. Consider $(G_k)_{\mathbf{L}^{(k)}}$ with edge-length vector $\mathbf{L}^{(k)}=(L, L_2, \dots, L_k)$ and $L_k\leq L_{k-1}\leq\cdots\leq L_2$. Let $T_i$ be the infinite trapezoidal region bounded by $r=0$, $r=L$, and $r=L_i$. Let $R_i$ be the infinite rectangular region bounded by $r=0$, $L=0$, and $r=L_i$. By the construction in the induction step of the proof of Theorem~\ref{thm:5.10}, the inclusion $(G_{k-1})_{\mathbf{L}^{(k-1)}} \hookrightarrow (G_{k})_{\mathbf{L}^{(k)}}$ induces a monomorphism $$PH_1(\lVert(G_{k-1})_{r,\mathbf{L}^{(k-1)}}\rVert;\mathbb{F})\ \hookrightarrow\ 
PH_1(\lVert(G_{k})_{r,\mathbf{L}^{(k)}}\rVert;\mathbb{F})$$
so every trapezoidal region \(T_i\) and rectangular region \(R_i\) present at stage \(k{-}1\)
appears unchanged in $PH_1(\lVert(G_{k})_{r,\mathbf{L}^{(k)}}\rVert;\mathbb{F})$.
Consequently,
$$
PH_1(\lVert(G_{k})_{-,-}\rVert;\mathbb{F}) \cong\
\Big(\,\bigoplus_{i=2}^{k-1} (\mathbb{F}T_i)^{m_i}\ \oplus\ (\mathbb{F}R_i)^{n_i}\Big)\ \oplus\ \hat{B}^{(k)},
$$
where $\mathbb{F}T_i$ and $\mathbb{F}R_i$ denote the interval modules supported on $T_i$ and $R_i$, respectively; 
$m_i,n_i\in\mathbb{Z}_{\ge 0}$ are their multiplicities; and $\hat{B}^{(k)}$ is a subrepresentation of $PH_1(\lVert(G_{k})_{-,-}\rVert;\mathbb{F})$ which is a pullback of a poset representation $B^{(k)}$ over the poset of the chambers, i.e., $B^{(k)}\circ \mathcal{F}^{(k)}=\hat{B}^{(k)}$.
Note that the support of $B^{(k)}$ is a type $A_2$ quiver and the non--trivial linear transformation in the support of $B^{(k)}$ is injective, the representation $B^{(k)}$ has $2k-6$ interval modules supported on the rectangular region $R_{k}$ and $2$ interval modules supported on the trapezoidal region $T_{k}$. 


The following table lists the support of indecomposable direct summands of $PH_1(\lVert (G_{k})_{-,-} \rVert;\mathbb{F})$ together with their multiplicities. Since every indecomposable direct summand of $PH_1(\lVert (G_{k})_{-,-} \rVert;\mathbb{F})$ is an interval module, this provides a complete characterization of the indecomposable decomposition of $PH_1(\lVert (G_k)_{-,-}\rVert;\mathbb{F})$.

\begin{center}
\begin{tabular}{|c|c|c|}
\hline
\thead{} & \thead{Support of the Summands of \medskip\newline $PH_1(\lVert (G_{k})_{-,-} \rVert;\mathbb{F})$} & \thead{Multiplicity} \\ 
 \hline
 $k=3$ & $T_3$ & $1$\\
\hline
\multirow{3}{*}{$k=4$} & $R_4$ & $1$\\ \cline{2-3}
                                        & $T_3$ & $1$\\ \cline{2-3}
                                        & $T_4$ & $3$ \\ \hline
\multirow{5}{*}{$k\geq 5$} & $R_4$ & 1\\ \cline{2-3}
                                        & $T_3$ & $1$\\ \cline{2-3}
                                        & $T_4$ & $3$ \\
                                        \cline{2-3}
                                        & $R_i (5\leq i\leq k)$ & $2i-6$ \\
                                        \cline{2-3}
                                        & $T_i (5\leq i\leq k)$ & 2 \\
                                        \hline
\end{tabular}
\label{table:1}
\captionof{table}{}
\end{center}

Since the category 
$\cat{vect}_{\mathbb{F}}^{(\mathbb{R}_{>0},\leq)\op\times(\mathbb{R}_{>0},\leq)}$
is Krull--Schmidt, the decomposition of the persistence module $PH_1(\lVert (G_{k})_{-,-} \rVert;\mathbb{F})$ is unique up to isomorphism. Moreover, since $PH_1(\lVert (G_{k})_{-,-} \rVert;\mathbb{F}) \cong 
PH_1((\mathsf{Star}_k)^2_{-,-};\mathbb{F})$ 
(as shown in Theorem~\ref{thm:2.8}), Table~\ref{table:1} gives a complete characterization of the indecomposable direct summands of 
$PH_1((\mathsf{Star}_k)^2_{-,-};\mathbb{F})$, up to isomorphism.

We now combine the above results with Lemma~\ref{lem:5.10} to obtain the main theorem of the paper.

\begin{theorem}[Decomposition of $PH_1((\mathsf{Star}_k)^2_{-,-};\mathbb{F})$]\label{thm:5.10}
Let $k \ge 3$ and let $\mathbf{L} = (L,L_2,\dots,L_k) \in (\mathbb{R}_{>0})^k$ with $L_k \le \cdots \le L_2$ arbitrary but fixed. For each $i=2,\dots,k$, let $T_i$ denote the (infinite) trapezoidal region in the $(r,L)$–plane bounded by $r = 0$, $r = L$, and $r = L_i$, and let $R_i$ denote the (infinite) rectangular region bounded by $r = 0$, $L = 0$, and $r = L_i$. 
The $2$-parameter persistence module $PH_1((\mathsf{Star}_k)^2_{-,-};\mathbb{F})$ is interval-decomposable. More precisely, it decomposes as a finite direct sum of interval modules supported on the regions $T_i$ and $R_i$ up to isomorphism, with the following indecomposable summands and multiplicities:

\begin{itemize}
  \item For $k = 3$:
  \begin{itemize}
    \item one indecomposable interval module supported on $T_3$.
  \end{itemize}

  \item For $k = 4$:
  \begin{itemize}
    \item one indecomposable interval module supported on $R_4$,
    \item one indecomposable interval module supported on $T_3$,
    \item three indecomposable interval modules supported on $T_4$.
  \end{itemize}

  \item For $k \ge 5$:
  \begin{itemize}
    \item one indecomposable interval module supported on $R_4$,
    \item one indecomposable interval module supported on $T_3$,
    \item three indecomposable interval modules supported on $T_4$,
    \item for each $i$ with $5 \le i \le k$, there are $2i-6$ indecomposable
          interval modules supported on $R_i$,
    \item for each $i$ with $5 \le i \le k$, there are $2$ indecomposable
          interval modules supported on $T_i$.
  \end{itemize}
\end{itemize}
\end{theorem}

As a direct application of Theorem~\ref{thm:5.10}, we obtain the following corollary for the indecomposable summands of $PH_1((\mathsf{Star}_k)^2_{-,\mathbf{L}};\mathbb{F})$ when $\mathbf{L}=(L,a,\dots,a)$ with $a>0$.

\begin{corollary}\label{Coro:1}
For any $k\geq 4$, $a>0$, and $\mathbf{L}=(L, a, \dots, a)\in(\mathbb{R}_{>0})^k$, the persistence module $PH_1((\mathsf{Star}_k)^2_{-,\mathbf{L}};\mathbb{F})$ is interval-decomposable, and its indecomposable direct summands are given in Figure~\ref{h1:fig:1} and Figure~\ref{h1:fig:2}, where the number provided for each region is the dimension of the vector space assigned to each point of the parameter space.

\begin{figure}[htbp!]
\begin{minipage}{0.55\textwidth}
\centering
\resizebox{0.8\textwidth}{!}{

\tikzset{every picture/.style={line width=0.75pt}} 

\begin{tikzpicture}[x=0.75pt,y=0.75pt,yscale=-1,xscale=1]

\draw  (6.75,299.04) -- (373.46,299.04)(43.42,46.75) -- (43.42,327.07) (366.46,294.04) -- (373.46,299.04) -- (366.46,304.04) (38.42,53.75) -- (43.42,46.75) -- (48.42,53.75)  ;
\draw  [draw opacity=0][fill={rgb, 255:red, 126; green, 211; blue, 33 }  ,fill opacity=0.62 ] (43.42,299.04) -- (139.57,203.47) -- (139.82,298.78) -- cycle ;
\draw  [draw opacity=0][fill={rgb, 255:red, 126; green, 211; blue, 33 }  ,fill opacity=0.62 ] (138.9,52.47) -- (138.9,106.03) -- (139.57,203.47) -- (43.42,299.04) -- (43.93,52.47) -- cycle ;

\draw (139.47,313.34) node   [align=left] {a};
\draw (35.35,203.5) node   [align=left] {a};
\draw (17.07,17.96) node [anchor=north west][inner sep=0.75pt]    {$L$};
\draw (370,302.52) node [anchor=north west][inner sep=0.75pt]    {$r$};

\end{tikzpicture}
}
\captionof{figure}{$\mbox{Multiplicity}=k(k-5)+5$}
\label{h1:fig:1}
\end{minipage}%
\begin{minipage}{0.5\textwidth}
\centering
\resizebox{0.8\textwidth}{!}{

\tikzset{every picture/.style={line width=0.75pt}} 

\begin{tikzpicture}[x=0.75pt,y=0.75pt,yscale=-1,xscale=1]

\draw  (6.75,299.04) -- (373.46,299.04)(43.42,46.75) -- (43.42,327.07) (366.46,294.04) -- (373.46,299.04) -- (366.46,304.04) (38.42,53.75) -- (43.42,46.75) -- (48.42,53.75)  ;
\draw  [draw opacity=0][fill={rgb, 255:red, 248; green, 231; blue, 28 }  ,fill opacity=0.8 ] (138.9,52.47) -- (138.9,106.03) -- (139.57,203.47) -- (43.42,299.04) -- (43.93,52.47) -- cycle ;

\draw (139.47,313.34) node   [align=left] {a};
\draw (35.35,203.5) node   [align=left] {a};
\draw (17.07,17.96) node [anchor=north west][inner sep=0.75pt]    {$L$};
\draw (370,302.52) node [anchor=north west][inner sep=0.75pt]    {$r$};

\end{tikzpicture}
}
\captionof{figure}{$\mbox{Multiplicity}=2k-4$}
\label{h1:fig:2}
\end{minipage}
\end{figure}
\end{corollary}

\begin{proof}


Let $a>0$. When $L_k=a$ for all $i=2,\dots,k$, the hyperplane $r=a$ has multiplicity $k-1$. Equivalently, $R_i = R_4$ and $T_i = T_4$ for all $4 \leq i \leq k$. An easy computation shows that the total number of interval summands supported on $R_4$ is $k^2-5k+5$, while the number of summands supported on $T_4$ is $2k-4$.
\end{proof}

\bibliographystyle{alpha}
\bibliography{app}
\appendix
\section{Proof of Lemma 3.3}\label{appendix:lemma-3.3}
\begin{proof}
Let $i:\lVert(G_k)_{r,\mathbf{L}}\rVert\hookrightarrow (\mathsf{Star}_k)^2_{r,\mathbf{L}}$ be the canonical inclusion map. We want to show that there is a continuous map $r:(\mathsf{Star}_k)^2_{r,\mathbf{L}} \rightarrow \lVert(G_k)_{r,\mathbf{L}}\rVert$ such that $r\circ i=\id_{\lVert(G_k)_{r,\mathbf{L}}\rVert}$ and $i\circ r\simeq \id_{(\mathsf{Star}_k)^2_{r,\mathbf{L}}}$ which is constant on $\lVert(G_k)_{r,\mathbf{L}}\rVert$.

The three types of 2-cells of $(\mathsf{Star}_k)^2_{r,\mathbf{L}}$, when $e_i\neq e_j$, are provided in Figure~\ref{fig:2-cell_star}.

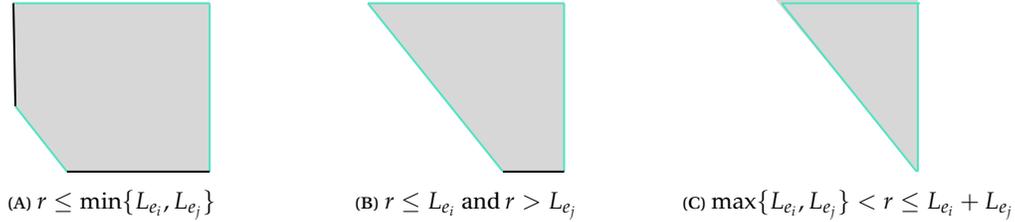
\begin{figure}[htbp!]
     \begin{subfigure}[b]{0.3\textwidth}
         \centering

\tikzset{every picture/.style={line width=0.75pt}} 

\begin{tikzpicture}[x=0.75pt,y=0.75pt,yscale=-1,xscale=1]

\draw  [draw opacity=0][fill={rgb, 255:red, 155; green, 155; blue, 155 }  ,fill opacity=0.4 ] (129,23) -- (129,108) -- (57,108) -- (31,75) -- (30,23) -- cycle ;
\draw [color={rgb, 255:red, 80; green, 227; blue, 194 }  ,draw opacity=1 ][line width=0.75]    (31,75) -- (57,108) ;
\draw [color={rgb, 255:red, 80; green, 227; blue, 194 }  ,draw opacity=1 ][line width=0.75]    (30,23) -- (129,23) ;
\draw [color={rgb, 255:red, 80; green, 227; blue, 194 }  ,draw opacity=1 ][line width=0.75]    (129,23) -- (129,108) ;
\draw [line width=0.75]    (30,23) -- (31,75) ;
\draw [line width=0.75]    (57,108) -- (129,108) ;

\end{tikzpicture}
         \caption{$r\leq\min\{L_{e_i},L_{e_j}\}$}
     \end{subfigure}%
     \hfill
     \begin{subfigure}[b]{0.3\textwidth}
         \centering
\tikzset{every picture/.style={line width=0.75pt}} 

\begin{tikzpicture}[x=0.75pt,y=0.75pt,yscale=-1,xscale=1]

\draw  [draw opacity=0][fill={rgb, 255:red, 155; green, 155; blue, 155 }  ,fill opacity=0.4 ] (129,23) -- (129,108) -- (98.14,108) -- (30,23) -- cycle ;
\draw [color={rgb, 255:red, 80; green, 227; blue, 194 }  ,draw opacity=1 ][line width=0.75]    (30,23) -- (98.14,108) ;
\draw [color={rgb, 255:red, 80; green, 227; blue, 194 }  ,draw opacity=1 ][line width=0.75]    (30,23) -- (129,23) ;
\draw [color={rgb, 255:red, 80; green, 227; blue, 194 }  ,draw opacity=1 ][line width=0.75]    (129,23) -- (129,108) ;
\draw [line width=0.75]    (98.14,108) -- (129,108) ;

\end{tikzpicture}
         \caption{$r\leq L_{e_i}$ and $r>L_{e_j}$}
     \end{subfigure}%
     \hfill
     \begin{subfigure}[b]{0.35\textwidth}
         \centering

\tikzset{every picture/.style={line width=0.75pt}} 

\begin{tikzpicture}[x=0.75pt,y=0.75pt,yscale=-1,xscale=1]

\draw  [draw opacity=0][fill={rgb, 255:red, 155; green, 155; blue, 155 }  ,fill opacity=0.4 ] (129,21.29) -- (129,108) -- (57.29,21.29) -- cycle ;
\draw [color={rgb, 255:red, 80; green, 227; blue, 194 }  ,draw opacity=1 ][line width=0.75]    (60.29,23) -- (128.43,108) ;
\draw [color={rgb, 255:red, 80; green, 227; blue, 194 }  ,draw opacity=1 ][line width=0.75]    (60.29,23) -- (130,23) ;
\draw [color={rgb, 255:red, 80; green, 227; blue, 194 }  ,draw opacity=1 ][line width=0.75]    (129,23) -- (129,108) ;

\end{tikzpicture}
         \caption{$\max\{L_{e_i},L_{e_j}\}<r\leq L_{e_i}+L_{e_j}$}
     \end{subfigure}
        \caption{Three types of 2-cells of $(\mathsf{Star}_k)^2_{r,\mathbf{L}}$, when $e_i\neq e_j$. The 1-cells colored in green represent the free facets of each 2-cell.}
        \label{fig:2-cell_star}
\end{figure}    

The interior of a 2-cell of $(\mathsf{Star}_k)^2_{r,\mathbf{L}}$ consists of all arrangements such that both robots are in the interior of some edges of $\mathsf{Star}_k$ and the distance between the robots is strictly greater than $r$. An 1-cell of $(\mathsf{Star}_k)^2_{r,\mathbf{L}}$ corresponds to one of the following arrangements:
\begin{enumerate}
    \item One robot moves in the interior of an edge of $\mathsf{Star}_k$, and the other robot is at the center of $\mathsf{Star}_k$;
    \item One robot moves in the interior of an edge of $\mathsf{Star}_k$, and the other robot is at a leaf of $\mathsf{Star}_k$;
    \item Both robots move in the interior of some edges of $\mathsf{Star}_k$, and the distance between the robots is equal to $r$.
\end{enumerate}

Depending on the parameters $r$ and $\mathbf{L}$, a 1-cell of $(\mathsf{Star}_k)^2_{r,\mathbf{L}}$ may or may not be a free facet of the 2-cell that contains it. The 1-cells colored in green represent the free facets of each 2-cell of $(\mathsf{Star}_k)^2_{r,\mathbf{L}}$ is provided in Figure~\ref{fig:2-cell_star}, assuming $e_i\neq e_j$.

We aim to construct a homotopy compatible with the parameters $r$ and $\mathbf{L}$, which can be achieved by carefully gluing a homotopy on each 2-cell of $(\mathsf{Star}_k)^2_{r,\mathbf{L}}$ to its 1-dimensional subspace. The cell structure of $(\mathsf{Star}_k)^2_{r,\mathbf{L}}$ is a special case of the cell structure of the restricted nth configuration space of a finite metric graph, which has been studied in \cite{dover2013homeomorphism}. A 2-cell of $(\mathsf{Star}_k)^2_{r,\mathbf{L}}$ is $\lVert e_i\rVert\times \lVert e_j\rVert-\{(x,y):\delta(x,y)<r\}$ for some edges $e_i,e_j$ of $\mathsf{Star}_k$, which is a subspace of $\lVert e_i\rVert\times \lVert e_j\rVert$.



Let $\sigma$ be a 2-cell of $(\mathsf{Star}_k)^2_{r,\mathbf{L}}$. 

\paragraph{Case 1:}When $i=j$ and $0<r<L_{e_i}$, the 2-cell $\sigma$ is an isosceles right triangle. When $r=L_{e_i}$, the triangle degenerates to a 0-cell.
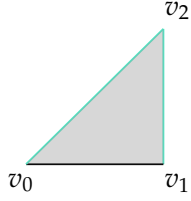
\begin{figure}[htbp!]
    \centering
    \resizebox{!}{0.18\textwidth}{
\tikzset{every picture/.style={line width=0.75pt}} 

\begin{tikzpicture}[x=0.75pt,y=0.75pt,yscale=-1,xscale=1]

\draw [color={rgb, 255:red, 80; green, 227; blue, 194 }  ,draw opacity=1 ][line width=0.75]    (86,21) -- (86,92) ;
\draw [line width=0.75]    (14,92) -- (86,92) ;
\draw [color={rgb, 255:red, 80; green, 227; blue, 194 }  ,draw opacity=1 ]   (14,92) -- (86,21) ;
\draw  [draw opacity=0][fill={rgb, 255:red, 155; green, 155; blue, 155 }  ,fill opacity=0.4 ] (86,21) -- (14,92) -- (86,92) -- cycle ;

\draw (3,95.4) node [anchor=north west][inner sep=0.75pt]    {$v_{0}$};
\draw (85,95.4) node [anchor=north west][inner sep=0.75pt]    {$v_{1}$};
\draw (85,5) node [anchor=north west][inner sep=0.75pt]    {$v_{2}$};

\end{tikzpicture}
    }
    \caption{A 2-cell $\sigma$ when $i=j$ and $0<r<L_{e_i}$ where the length of each leg is $L_{e_i}-r$.}
    \label{fig:enter-label}
\end{figure}

Any point $x\in\sigma=\Delta v_0v_1v_2$ can be written as $x=a_0v_0+a_1v_1+a_2v_2$ where $a_0, a_1, a_2\in\mathbb{R}_{\geq 0}$ and $a_0+a_1+a_2=1$.

$$K^{\sigma}(x,t)=\begin{cases} a_0v_0+(a_1+ta_2)v_1+(1-t)a_2v_2 , & \mbox{if } x\in\Delta v_0v_1v_2;\\ x , & \mbox{otherwise } \end{cases}$$

Moreover, for all $x\in\Delta v_0v_1v_2$, define  
$$r^{\sigma}(x)=a_0v_0+(a_1+a_2)v_1$$

\paragraph{Case 2:}When $0<r\leq \min\{L_{e_i}, L_{e_j}\}$, we subdivide $\sigma$ into three triangles, as shown in Figure~\ref{fig: small_r}. 
\begin{figure}[htbp!]
     \begin{subfigure}[b]{0.3\textwidth}
         \centering

\tikzset{every picture/.style={line width=0.75pt}} 

\begin{tikzpicture}[x=0.75pt,y=0.75pt,yscale=-1,xscale=1]

\draw  [draw opacity=0][fill={rgb, 255:red, 155; green, 155; blue, 155 }  ,fill opacity=0.4 ] (129,23) -- (129,108) -- (57,108) -- (31,75) -- (30,23) -- cycle ;
\draw [color={rgb, 255:red, 80; green, 227; blue, 194 }  ,draw opacity=1 ][line width=0.75]    (31,75) -- (57,108) ;
\draw [color={rgb, 255:red, 80; green, 227; blue, 194 }  ,draw opacity=1 ][line width=0.75]    (30,23) -- (129,23) ;
\draw [color={rgb, 255:red, 80; green, 227; blue, 194 }  ,draw opacity=1 ][line width=0.75]    (129,23) -- (129,108) ;
\draw [line width=0.75]    (30,23) -- (31,75) ;
\draw [line width=0.75]    (57,108) -- (129,108) ;
\draw  [dash pattern={on 4.5pt off 4.5pt}]  (31,75) -- (129,23) ;
\draw  [dash pattern={on 4.5pt off 4.5pt}]  (57,108) -- (129,23) ;

\draw (129,104.4) node [anchor=north west][inner sep=0.75pt]    {$v_{0}$};
\draw (44.57,104.83) node [anchor=north west][inner sep=0.75pt]    {$v_{1}$};
\draw (14,66.4) node [anchor=north west][inner sep=0.75pt]    {$v_{2}$};
\draw (14,4.4) node [anchor=north west][inner sep=0.75pt]    {$v_{3}$};
\draw (129,5.4) node [anchor=north west][inner sep=0.75pt]    {$v_{4}$};

\end{tikzpicture}
         \caption{$r\leq\min\{L_{e_i},L_{e_j}\}$}
         \label{fig: small_r}
     \end{subfigure}%
     \hfill
     \begin{subfigure}[b]{0.3\textwidth}
         \centering

\tikzset{every picture/.style={line width=0.75pt}} 

\begin{tikzpicture}[x=0.75pt,y=0.75pt,yscale=-1,xscale=1]

\draw  [draw opacity=0][fill={rgb, 255:red, 155; green, 155; blue, 155 }  ,fill opacity=0.4 ] (129,23) -- (129,108) -- (98.14,108) -- (30,23) -- cycle ;
\draw [color={rgb, 255:red, 80; green, 227; blue, 194 }  ,draw opacity=1 ][line width=0.75]    (30,23) -- (98.14,108) ;
\draw [color={rgb, 255:red, 80; green, 227; blue, 194 }  ,draw opacity=1 ][line width=0.75]    (30,23) -- (129,23) ;
\draw [color={rgb, 255:red, 80; green, 227; blue, 194 }  ,draw opacity=1 ][line width=0.75]    (129,23) -- (129,108) ;
\draw [line width=0.75]    (98.14,108) -- (129,108) ;
\draw  [dash pattern={on 4.5pt off 4.5pt}]  (129,23) -- (98.14,108) ;

\draw (128,105.4) node [anchor=north west][inner sep=0.75pt]    {$v_{0}$};
\draw (88.57,105.83) node [anchor=north west][inner sep=0.75pt]    {$v_{1}$};
\draw (15,5.4) node [anchor=north west][inner sep=0.75pt]    {$v_{2}$};
\draw (127,5.4) node [anchor=north west][inner sep=0.75pt]    {$v_{3}$};

\end{tikzpicture}
         \caption{$L_{e_j}<r\leq L_{e_i}$}
         \label{fig: r}
     \end{subfigure}%
     \hfill
     \begin{subfigure}[b]{0.35\textwidth}
         \centering
\tikzset{every picture/.style={line width=0.75pt}} 

\begin{tikzpicture}[x=0.75pt,y=0.75pt,yscale=-1,xscale=1]

\draw  [draw opacity=0][fill={rgb, 255:red, 155; green, 155; blue, 155 }  ,fill opacity=0.4 ] (129,21.29) -- (129,108) -- (57.29,21.29) -- cycle ;
\draw [color={rgb, 255:red, 80; green, 227; blue, 194 }  ,draw opacity=1 ][line width=0.75]    (60.29,23) -- (128.43,108) ;
\draw [color={rgb, 255:red, 80; green, 227; blue, 194 }  ,draw opacity=1 ][line width=0.75]    (60.29,23) -- (130,23) ;
\draw [color={rgb, 255:red, 80; green, 227; blue, 194 }  ,draw opacity=1 ][line width=0.75]    (129,23) -- (129,108) ;

\draw (124,104.4) node [anchor=north west][inner sep=0.75pt]    {$v_{0}$};
\draw (44.57,2.83) node [anchor=north west][inner sep=0.75pt]    {$v_{1}$};
\draw (123.57,2.4) node [anchor=north west][inner sep=0.75pt]    {$v_{2}$};

\end{tikzpicture}
         \caption{$\max\{L_{e_i},L_{e_j}\}<r\leq L_{e_i}+L_{e_j}$}
         \label{fig: r_large}
     \end{subfigure}
        \caption{Three types of 2-cells of $(\mathsf{Star}_k)^2_{r,\mathbf{L}}$, when $e_i\neq e_j$. The 1-cells colored in green represent the free facets of each 2-cell.}
        \label{fig:three graphs}
\end{figure}
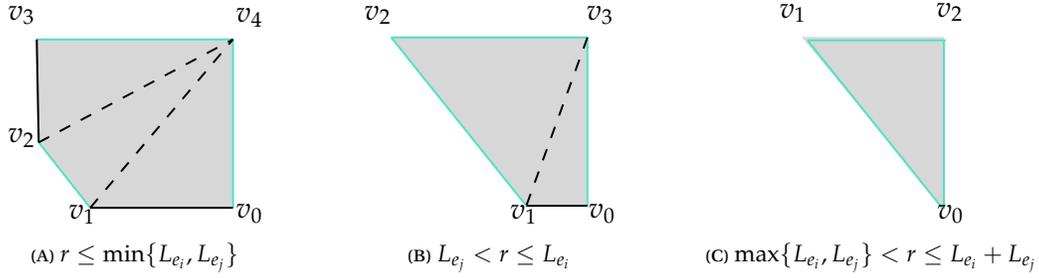

Hence every point $x$ in the polytopal cell can be written as a unique convex combination of the vertices of the triangle that contains the point. For example, any point $x\in\Delta v_1v_2v_4$ can be written as $x=a_1v_1+a_2v_2+a_4v_4$ where $a_1, a_2, a_4\in\mathbb{R}_{\geq 0}$ and $a_1+a_2+a_4=1$ . Define 
$$G^{\sigma}_1(x,t)=\begin{cases} (a_1-ta)v_1+(a_2-ta)v_2+(a_4+2ta)v_4 , & \mbox{if } x\in\Delta v_1v_2v_4 \mbox{ and }\\ & a=\min\{a_1, a_2\}; \\ x , & \mbox{otherwise } \end{cases}$$

$$G^{\sigma}_2(x,t)=\begin{cases} (a_2-ta)v_2+(a_4-ta)v_4+(a_3+2ta)v_3 , & \mbox{if } x\in\Delta v_2v_3v_4 \mbox{ and }\\ & a=\min\{a_2, a_4\}; \\ x , & \mbox{otherwise } \end{cases}$$

$$G^{\sigma}_3(x,t)=\begin{cases} (a_1-ta)v_1+(a_4-ta)v_4+(a_0+2ta)v_0 , & \mbox{if } x\in\Delta v_0v_1v_4 \mbox{ and }\\ & a=\min\{a_1, a_2\}; \\ x , & \mbox{otherwise } \end{cases}$$

$$F^{\sigma}_1(x,t)=\begin{cases}(1-t)a_2v_2+a_4v_4+(a_3+ta_2)v_3 , & \mbox{if } x\in\Delta v_2v_3v_4; \\ x , & \mbox{otherwise } \end{cases}$$

$$F^{\sigma}_2(x,t)=\begin{cases} (1-t)a_1v_1+a_4v_4+(a_0+ta_1)v_0 , & \mbox{if } x\in\Delta v_0v_1v_4; \\ x , & \mbox{otherwise } \end{cases}$$

Let $G^{\sigma}(x,t)=G^{\sigma}_3\circ G^{\sigma}_2\circ G^{\sigma}_1$ and $F^{\sigma}(x,t)=F^{\sigma}_2\circ F^{\sigma}_1$. Here, $G^{\sigma}(x,t)$ is a homotopy of the polytopal cell $\sigma$ relative its boundary $\partial\sigma$ minus $[v_1, v_2]$.

Moreover, for all $x\in\sigma$, define 
$$r^{\sigma}(x)=\begin{cases}(a_2+a_3+a)v_3+(a_4-a)v_4 , & \mbox{if } x\in\Delta v_2v_3v_4 \mbox{ and } a=\min\{a_2, a_4\}; \\
(a_0+a_1+a)v_0+(a_4-a)v_4 , & \mbox{if } x\in\Delta v_0v_1v_4 \mbox{ and } a=\min\{a_1, a_2\};
\\
(a_4+2a_1-a)v_4 + (a_2-a_1+a)v_3 , & \mbox{if } x\in\Delta v_1v_2v_4 \mbox{ and } a=\min\{a_2-a_1, 2a_1+a_4\}.
\end{cases}$$

\paragraph{Case 3:}When $L_{e_j}<r\leq L_{e_i}$, we subdivide $\sigma$ into two triangles, as shown in Figure~\ref{fig: r}. Hence every point $x$ in $\sigma$ can be written as a unique convex combination of the vertices of the triangle that contains the point. For example, any point $x\in\Delta v_1v_2v_3$ can be written as $x=a_1v_1+a_2v_2+a_3v_3$ where $a_1, a_2, a_3\in\mathbb{R}_{\geq 0}$ and $a_1+a_2+a_3=1$ . Define 

$$G^{\sigma}_1(x,t)=\begin{cases} (a_1-ta)v_1+(a_2-ta)v_2+(a_3+2ta)v_3 , & \mbox{if } x\in\Delta v_1v_2v_3 \mbox{ and }\\ & a=\min\{a_1, a_2\}; \\ x , & \mbox{otherwise } \end{cases}$$

$$G^{\sigma}_2(x,t)=\begin{cases} (a_1-ta)v_1+(a_3-ta)v_3+(a_0+2ta)v_0 , & \mbox{if } x\in\Delta v_0v_1v_3 \mbox{ and }\\ & a=\min\{a_1, a_3\}; \\ x , & \mbox{otherwise } \end{cases}$$

$$F_1^{\sigma}(x,t)=\begin{cases}a_1v_1 + (1-t)a_2v_2 + (a_3+ ta_2)v_3 , & \mbox{if } x\in\Delta v_1v_2v_3; \\ x , & \mbox{otherwise } \end{cases}$$

$$F_2^{\sigma}(x,t)=\begin{cases}(a_0+ta_1)v_0+(1-t)a_1v_1+a_3v_3 , & \mbox{if } x\in\Delta v_0v_1v_3; \\ x , & \mbox{otherwise } \end{cases}$$

Then $G^{\sigma}(x,t)=G^{\sigma}_2\circ G^{\sigma}_1$ is a homotopy of the polytopal cell $\sigma$ relative its boundary $\partial\sigma$ minus $[v_1, v_2]$. Define $F^{\sigma}(x,t)=F^{\sigma}_2\circ F^{\sigma}_1$.

Moreover, for all $x\in\sigma$, define 
$$r^{\sigma}(x)=\begin{cases}v_3 , & \mbox{if } x\in\Delta v_1v_2v_4 \mbox{ and } a_1\leq a_2;
\\
(a_1-a_2+a)v_0+(2a_2+a_3-a)v_3, & \mbox{if } x\in\Delta v_1v_2v_4 \mbox{ and } a_1\geq a_2, a=\min\{a_1-a_2,2a_2+a_3\};
\\
(a_0+a_1+a)v_0+(a_3-a)v_3 , & \mbox{if } x\in\Delta v_0v_1v_3 \mbox{ and } a=\min\{a_1,a_3\}.
\end{cases}$$

\paragraph{Case 4:}When $\max\{L_{e_i},L_{e_j}\}<r\leq L_{e_i}+L_{e_j}$, $\sigma$ is a cell as shown in Figure~\ref{fig: r_large}. Hence, every point $x$ in $\sigma$ can be written as a unique convex combination of the vertices of the triangle that contains the point. Any point $x\in\Delta v_0v_1v_2$ can be written as $x=a_0v_0+a_1v_1+a_2v_2$ where $a_0, a_1, a_2\in\mathbb{R}_{\geq 0}$ and $a_0+a_1+a_2=1$. Then $$G^{\sigma}(x,t)= (a_0-ta_0)v_0+(a_1-ta_1)v_1+(a_2+ta_0+ta_1)v_2$$ is a homotopy of the polytopal cell $\sigma$ to $\{v_2\}$.

Moreover, for all $x\in\sigma$, define 
$$r^{\sigma}(x)=v_2$$

Define $$G(x,t)=\begin{cases} G^{\sigma}(x,t), & \mbox{if }x\in\sigma \mbox{ where $\sigma$ is a 2-cell of }(\mathsf{Star}_k)^2_{r,\mathbf{L}}\end{cases}$$

$$F(x,t)=\begin{cases} F^{\sigma}(x,t), & \mbox{if }x\in\sigma \mbox{ where $\sigma$ is a 2-cell of }(\mathsf{Star}_k)^2_{r,\mathbf{L}}\end{cases}$$

$$K(x,t)=\begin{cases} K^{\sigma}(x,t), & \mbox{if }x\in\sigma \mbox{ where $\sigma$ is a 2-cell of }(\mathsf{Star}_k)^2_{r,\mathbf{L}}\end{cases}$$

$$r(x)=\begin{cases} r^{\sigma}(x), & \mbox{if }x\in\sigma \mbox{ where $\sigma$ is a 2-cell of }(\mathsf{Star}_k)^2_{r,\mathbf{L}}\end{cases}$$

By construction, it is clear that $r\circ i=\id_{\lVert (G_k)_{r,\mathbf{L}}\rVert}$.

For each $\sigma$, $G^{\sigma}(x,t)$ is a homotopy of $\sigma$ relative to $\partial\sigma$ minus a free facet of $\sigma$, hence the restrictions of $G^{\sigma}(x,t)$ to the non-free facets are identity maps. Hence $G(x,t)$ is well-defined and continuous. Let $Y$ be the image of $G(x,1)$ which is the deformation retract of  $(\mathsf{Star}_k)^2_{r,\mathbf{L}}$. Restrict $F(x,t)$ to $Y$, then $F(x,t)$ is a continuous on $Y$.

Let $I=[0,1]$ be the unit interval. Define
$$H(x,t):(\mathsf{Star}_k)^2_{r,\mathbf{L}}\times I\rightarrow (\mathsf{Star}_k)^2_{r,\mathbf{L}}$$
where $H(x,t):=F(x,t)\circ G(x,t)\circ K(x,t)$. Note that 
$$H(x,0)=\id_{(\mathsf{Star}_k)^2_{r,\mathbf{L}}}$$
and 
$$i\circ r(x)=F(x,1)\circ G(x,1)\circ K(x,1)=H(x,1)$$ 
we conclude that $H(x,t):=F(x,t)\circ G(x,t)\circ K(x,t)$ gives the desired homotopy $i\circ r\simeq \id_{(\mathsf{Star}_k)^2_{r,\mathbf{L}}}$.
\end{proof}

\end{document}